\documentclass[12 pt]{amsart}
\title{Lifts of twisted K3 surfaces to characteristic 0}
\author{Daniel Bragg}
\address{Department of Mathematics, University of California, Berkeley, 
Berkeley, CA 94720}
\email{braggdan@berkeley.edu}
\usepackage{macros}
\usepackage{mathtools}
\begin{document}
	\begin{abstract}
		Deligne \cite{MR638598} showed that every K3 surface over an algebraically closed field of positive characteristic admits a lift to characteristic 0. We show the same is true for a twisted K3 surface. To do this, we study the versal deformation spaces of twisted K3 surfaces, which are particularly interesting when the characteristic divides the order of the Brauer class. We also give an algebraic construction of certain moduli spaces of twisted K3 surfaces over $\Spec\mathbf{Z}$ and apply our deformation theory to study their geometry. As an application of our results, we show that every derived equivalence between twisted K3 surfaces in positive characteristic is orientation preserving. 
	\end{abstract}
\maketitle

\setcounter{tocdepth}{1}
\tableofcontents

\section{Introduction}

A \textit{twisted K3 surface} is a pair $(X,\alpha_{\Br})$ where $X$ is a K3 surface and $\alpha_{\Br}\in\Br(X)$. We will show that every twisted K3 surface in characteristic $p$ lifts to characteristic 0.

\begin{theorem}\label{thm:basic lifting}
    Let $(X,\alpha_{\Br})$ be a twisted K3 surface over an algebraically closed field $k$ of characteristic $p>0$. Let $\alpha\in\H^2(X,\mu_n)$ be a class whose image in the Brauer group is $\alpha_{\Br}$ and let $L$ be an ample line bundle on $X$. There exists a DVR $R$ with residue field $k$ and field of fractions of characteristic 0 and a triple $(\widetilde{X},\widetilde{\alpha},\widetilde{L})$, where $\widetilde{X}$ is a K3 surface over $R$ such that $\widetilde{X}\otimes_Rk\cong X$, $\widetilde{\alpha}\in\H^2(\widetilde{X},\mu_n)$ is a class such that $\widetilde{\alpha}|_X=\alpha$, and $\widetilde{L}$ is a line bundle on $\widetilde{X}$ such that $\widetilde{L}|_X\cong L$.
\end{theorem}

In particular, the image $\widetilde{\alpha}_{\Br}$ of $\widetilde{\alpha}$ in $\Br(\widetilde{X})$ satisfies $\widetilde{\alpha}_{\Br}|_X=\alpha_{\Br}$. Thus, $(\widetilde{X},\widetilde{\alpha}_{\Br})$ is a twisted K3 surface over $R$ lifting $(X,\alpha_{\Br})$.
 

Without the Brauer class, this result is due to Deligne \cite{MR638598}, with further refinements by Ogus \cite{Ogus78}. We also consider the more general problem of the existence of lifts of a twisted K3 surface together with a collection of line bundles. In the non--twisted case this problem was considered by Lieblich--Olsson \cite{LO15} and Lieblich--Maulik \cite{MR3934849}. For both these problems, we give conditions under which the appropriate universal deformation space is formally smooth, which implies that such a lift exists over the ring of Witt vectors $W=W(k)$. We defer the precise statements of these results: the existence of lifts with multiple line bundles is given in Theorem \ref{thm:mega lifting}, and the question of smoothness is considered in Theorem \ref{thm:formally smooth def space, one line bundle} (for twisted K3 surfaces with one line bundle) and Corollary \ref{cor:formally smooth def space, twisted version} (for twisted K3 surfaces with multiple line bundles). Even forgetting the twisting, our methods yield stronger results for the existence of lifts of K3 surfaces together with collections of line bundles than we have seen in the literature (see Corollary \ref{cor:formally smooth def space}).

We outline the basic strategy behind the proof of Theorem \ref{thm:basic lifting}.
The usual procedure for producing a lift consists of two steps: first, using formal deformation theory one constructs lifts to every finite order, and second, one shows that the resulting formal system algebraizes. This strategy is carried out by Deligne \cite{MR638598} in his study of the lifting problem for K3 surfaces. In this case, a key input is the result of Rudakov and Shafarevich \cite{RS83} that $\H^0(X,T^1_X)=0$ (this result has been subsequently reproved using cohomological methods, see Lang--Nygaard \cite{MR585962} and Nygaard \cite{MR554382}). This result is equivalent to the vanishing of $\H^2(X,T^1_X)$, which implies that the formal deformation problem is essentially \textit{trivial}: any K3 surface $X$ deforms over any infinitesimal thickening. More precisely, the universal deformation space $\Def_X$ is smooth over $W$, and moreover we have
\[
  \Def_X\cong\Spf W[[t_1,\dots,t_{20}]].
\]
However, the resulting systems will generally not algebraize. Thus, Deligne considers instead deformations of a pair $(X,L)$, where $L$ is an ample line bundle on $X$. The algebraization of systems of such pairs is guaranteed by a theorem of Grothendieck. However, such pairs are no longer unobstructed in general, and so the deformation theoretic step requires a further analysis. Deligne first shows that the inclusion
\[
  \Def_{(X,L)}\subset\Def_X\cong\Spf W[[t_1,\dots,t_{20}]]
\]
is a closed formal subscheme of dimension 19 over $W$ defined by one equation. By analyzing de Rham and crystalline cohomology, he then shows that $\Def_{(X,L)}$ is flat over $W$, and hence the desired formal system exists. This last step was improved upon by Ogus \cite{Ogus78}, who showed that in fact $\Def_{(X,L)}$ is frequently smooth over $W$.

We consider a twisted K3 surface $(X,\alpha_{\Br})$. We show that, as a consequence of the vanishing of $\H^3(X,\mathscr{O}_X)$, such objects are unobstructed. Thus, the universal deformation space $\Def_{(X,\alpha_{\Br})}$ is smooth over $W$, and moreover we have
\[
  \Def_{(X,\alpha_{\Br})}\cong\Spf W[[t_1,\dots,t_{20},s]].
\]
Hence, as before, there are many formal systems over $\Spf W$. The difficulty again lies in the algebraization step. To algebraize the underlying system of K3 surfaces, we might carry along an ample line bundle on $X$. However, even if the underlying system of K3 surfaces algebraizes, a formal system of Brauer classes will typically not algebraize (this is the essential reason why the Brauer group functor is not representable). To remedy this, we need to include some extra data related to the class $\alpha_{\Br}$. We will consider triples $(X,\alpha,L)$, where $\alpha$ is a lift of $\alpha_{\Br}$ along the map
\[
  \H^2(X,\mu_n)\to\Br(X)
\]
for some $n$. Here, $\H^2(X,\mu_n)$ denotes the second flat (fppf) cohomology of the sheaf $\mu_n$ of $n$th roots of unity on $X$. In \S\ref{sec:algebraization} we show that if $X$ is a smooth proper surface then formal families of such triples $(X,\alpha,L)$ algebraize (Proposition \ref{prop:02}). The idea is to show that every formal family of flat cohomology classes is induced by a formal family of Azumaya algebras, whose algebraization follows from Grothendieck's existence theorem for coherent sheaves. Moreover, such a choice of $\alpha$ is the minimal amount of extra data needed to ensure algebraization. We mention that this section is largely independent of the rest of the paper, and might be skipped on a first reading.


With this motivation, we then study the universal deformation spaces associated to a triple $(X,\alpha,L)$. For technical reasons, it turns out to be useful to consider also deformations of $\mu_n$ and $\mathbf{G}_m$-gerbes. Lacking a suitable reference, we include some abstract results along these lines in Appendix \ref{appendix}. We show that the inclusion
\[
  \Def_{(X,\alpha)}\subset\Def_{(X,\alpha_{\Br})}\cong\Spf W[[t_1,\dots,t_{20},s]]
\]
of deformation functors is a closed formal subscheme defined by one equation. It follows that
\[
    \Def_{(X,\alpha,L)}\subset\Def_{(X,\alpha_{\Br})}\cong\Spf W[[t_1,\dots,t_{20},s]]
\]
is a closed subscheme defined by two equations. We analyze these deformation spaces using obstruction classes associated to classes in $\H^2(X,\mu_n)$. We show that if $n$ is coprime to $p$, then such classes deform uniquely along any thickening of $X$. Thus, in this case our main result follows quickly from \cite{MR638598}, and is well known to experts. However, if $p$ divides $n$, there are additional obstructions to deforming such classes, and it is therefore this case that is the main contribution of this paper. 

We analyze these obstructions in \S\ref{sec:deformations of flat coho classes}, and compute them in terms of cup product with the Kodaira--Spencer class. 
In \S\ref{sec:dlog map} we study the interaction between deformations of a line bundle $L$ on $X$ and deformations of classes $\alpha\in\H^2(X,\mu_n)$. Using this analysis, we give conditions under which the deformation space $\Def_{(X,\alpha,L)}$ is smooth over $W$. This requires some precise computations in the de Rham cohomology of K3 surfaces, particularly in the supersingular case. Combined with the algebraization result of Proposition \ref{prop:02}, these results imply Theorem \ref{thm:basic lifting} outside of a small locus of exceptional cases. We also consider in \S\ref{ssec:multiple line bundles} deformations with multiple line bundles. Even neglecting the twisting, our results in this section seem to be new in some cases.

To obtain results in the case when the universal deformation space is not smooth, we use global methods. We define in \S\ref{sec:moduli of twisted K3 surfaces} a certain moduli stack $\ms M^n_d$ over $\Spec\mathbf{Z}$ parametrizing tuples $(X,\alpha,L)$, where $X$ is a K3 surface, $\alpha\in\H^2(X,\mu_n)$, and $L$ is an ample line bundle on $X$ of degree $2d$. The proof that this stack is algebraic is a consequence of the following result (proved in \S\ref{sec:algebraization}), which may be of independent interest. Given a morphism $f:X\to S$ of algebraic spaces, we let $R^mf_*\mu_n$ denote the $m$th higher pushforward from the big flat site of $X$ to that of $S$. Equivalently, $R^mf_*\mu_n$ is the flat sheafification of the functor on the category of $S$ schemes defined by
\[
    T\mapsto\H^m(X\times_ST,\mu_n)
\]
where the right side denotes cohomology in the flat topology.
\begin{theorem}\label{thm:representability for R2}
    Let $f:X\to S$ be a smooth proper morphism of algebraic spaces of relative dimension 2 with geometrically connected fibers. Let $n$ be a positive integer. Assume that $R^1f_*\mu_n=0$ and that for all geometric points $s\in S$ we have $\H^1(X_s,\ms O_{X_s})=0$. The sheaf $R^2f_*\mu_n$ is a group algebraic space of finite presentation over $S$. 
\end{theorem}
If $n$ is invertible on $S$, this follows from fundamental theorems of \'{e}tale cohomology \cite[Th. finitude, Th\'{e}or\`{e}me 1.1]{SGA4.5} (and no vanishing assumptions are needed). When $n$ is not invertible the result is more subtle. When $S$ has equal characteristic $p$, more general representability results are proven in \cite{bragg2021representability}. However, in this paper we are particularly interested in the case when $S$ has mixed characteristic. Our proof instead generalizes the method of proof of \cite[Theorem 2.1.6]{BL17}, and relies on de Jong and Lieblich's solution of the period index problem for function fields of surfaces and Lieblich's study of asymptotic properties of moduli spaces of twisted sheaves.

We make a few observations on the geometric structure of these moduli stacks; for instance, we show that the morphism
\[
	\ms M^n_d\to\Spec\mathbf{Z}
\]
is flat and a local complete intersection of relative dimension 19. We deduce Theorem \ref{thm:basic lifting} as a consequence. In \S\ref{sec:multiple line bundles, global results}, we consider the analogous moduli spaces for twisted K3 surfaces equipped with multiple line bundles, and deduce similar geometric results.

We also consider in \S\ref{sec:moduli of primitive twisted K3 surfaces} a certain refined moduli stack 
\[
    \mc M_d^n\to\Spec\mathbf{Z}
\]
over the integers parametrizing tuples $(X,\alpha,L)$ where $X$ is a K3 surface, $L$ is a primitive ample class of degree $2d$, and $\alpha\in\H^2(X,\mu_n)$ is a class which is primitive with respect to $L$, in a certain sense. The fiber $\mc M_d^n\otimes\mathbf{C}$ over the complex numbers recovers the moduli stack of twisted complex K3 surfaces constructed by Brakkee \cite{MR4126898} using analytic methods. We show that these stacks have some advantageous geometric properties; for instance, they are smooth over $\Spec \mathbf{Z}[\frac{1}{2dn}]$, and the fibers $\mc M_d^n\otimes\mathbf{F}_p$ are generically smooth. We give a brief description of their singular loci. The geometry of the moduli stacks $\mc M_d^n\otimes\mathbf{F}_p$ seem particularly interesting when $p$ divides $n$, and we think are deserving of further study.


We hope that our lifting results will be of general utility in the study of twisted K3 surfaces in positive characteristic. We record in \S\ref{sec:application to twisted derived equivalences} one instance where this is the case by resolving the last open cases of the conjecture that derived equivalences of twisted K3 surfaces are orientation preserving.

\vspace{.5cm}
\noindent\textbf{Summary:} In \S\ref{sec:algebraization}, we show that formal families of flat cohomology classes on a surface algebraize (Proposition \ref{prop:02}). We also prove Theorem \ref{thm:representability for R2}. This section is independent of the rest of the paper. In \S\ref{sec:deformations of flat coho classes}, we discuss obstruction classes for flat cohomology classes and their relation with Kodaira--Spencer classes. We then specialize to K3 surfaces in positive characteristic. In \S\ref{sec:K3 surfaces in positive characteristic} we recall some relevant definitions and cohomological invariants. In \S\ref{sec:formal deformations} we consider the universal deformation spaces associated to gerbes over K3 surfaces. We prove that they are prorepresentable and describe them in explicit coordinates. In \S\ref{sec:dlog map} we make some computations in the de Rham cohomology of K3 surfaces, and derive conditions under which formal deformation spaces are formally smooth. Combined with the algebraization result of Proposition \ref{prop:02}, this implies our main lifting results in many cases. In \S\ref{sec:moduli of twisted K3 surfaces} we define global moduli spaces of twisted K3 surfaces, and complete the proof of Theorem \ref{thm:basic lifting}. We also define some refined moduli spaces, extending those defined by Brakkee over the complex numbers \cite{MR4126898}. Finally, in \S\ref{sec:application to twisted derived equivalences} we give an application to twisted derived equivalences. In Appendix \ref{appendix}, we consider deformations of gerbes. Our main results are the definition of obstruction classes and a criterion for prorepresentability (generalizing results of Artin--Mazur \cite{MR0457458}), both of which are used in \S\ref{sec:formal deformations}.

\vspace{.5cm}
\noindent\textbf{Conventions:} We work throughout over an algebraically closed field $k$ of characteristic $p>0$ with ring of Witt vectors $W=W(k)$. If $X$ is a scheme, $\H^m(X,\mu_n)$ will always denote cohomology with respect to the flat (fppf) topology. We will frequently write $\alpha$ for a class in $\H^2(X,\mu_n)$ and let $\alpha_{\Br}$ denote the image of $\alpha$ under the map
\[
  \H^2(X,\mu_n)\to\H^2(X,\mathbf{G}_m).
\]

\noindent
\textbf{Funding:} This work was supported by the National Science Foundation [DMS-1902875].

\noindent
\textbf{Acknowledgements:} We thank Max Lieblich and Martin Olsson for interesting discussions on the content of this paper. We are grateful for the thorough reading of this paper by the anonymous referees.


\section{Algebraization of flat cohomology classes on a surface}\label{sec:algebraization}

In this section we will show that formal families of classes in the flat cohomology group $\H^2(X,\mu_n)$ of a family of surfaces algebraize (Proposition \ref{prop:02}). We will also prove Theorem \ref{thm:representability for R2}, which is the stronger statement that the relative $\H^2(X,\mu_n)$ of a family of surfaces is representable. This section is independent of the rest of the paper.

\subsection{Deformations of Azumaya algebras}\label{ssec:algebraization}

We describe the deformation theory of Azumaya algebras, with particular attention to the case when the degree is divisible by the characteristic. We begin by briefly recalling some of the theory of Azumaya algebras. For more details, we refer the reader to \cite[Chapter 18]{Huy16}, and the associated references. Let $X$ be a scheme and let $n$ be a positive integer. An \textit{Azumaya algebra} on $X$ of degree $n$ is a sheaf $\ms A$ of associative (possibly noncommutative) unital $\ms O_X$-algebras on $X$ which is \'{e}tale locally isomorphic to the matrix algebra $\mathrm{M}_n(\ms O_X)$. Thus, $\ms A$ has rank $n^2$ as an $\ms O_X$-module. Consider the commutative diagram
\begin{equation}\label{eq:diagram01}
  \begin{tikzcd}
    &1\arrow{d}&1\arrow{d}\\
    1\arrow{r}&\mu_n\arrow{r}\arrow{d}&\SL_n\arrow{r}\arrow{d}&\PGL_n\arrow{r}\arrow[equal]{d}&1\\
    1\arrow{r}&\mathbf{G}_m\arrow{r}\arrow{d}{\cdot n}&\GL_n\arrow{r}\arrow{d}{\det}&\PGL_n\arrow{r}&1\\
    &\mathbf{G}_m\arrow[equals]{r}\arrow{d}&\mathbf{G}_m\arrow{d}\\
    &1&1
  \end{tikzcd}
\end{equation}
of group schemes on $X$, which has exact rows and columns. By the Skolem--Noether theorem, the automorphism sheaf of $\mathrm{M}_n(\ms O_X)$ is isomorphic to $\PGL_{n}$, and so an Azumaya algebra $\ms A$ of degree $n$ gives rise to a class $[\ms A]$ in the nonabelian cohomology $\H^1(X,\PGL_n)$ which (by definition) classifies $\PGL_n$-torsors. An Azumaya algebra $\ms A$ of degree $n$ has an associated class in $\H^2(X,\mu_n)$, which is the image $\partial([\ms A])$ of $[\ms A]$ under the boundary map
\begin{equation}\label{eq:boundary map for Azumaya algebras}
    \partial:\H^1(X,\PGL_n)\to\H^2(X,\mu_n)
\end{equation}
induced by the top row of~\eqref{eq:diagram01}.

Azumaya algebras can be understood in terms of twisted sheaves on gerbes (see \S\ref{ssec:gerbes}).
 A coherent sheaf $\ms E$ on a $\mu_n$ or $\mathbf{G}_m$-gerbe is said to be $m$\textit{-twisted} if the inertial action is via the character $\lambda\mapsto\lambda^m$ (see eg. \cite[Definition 12.3.2]{MR3495343}). We refer to a $1$-twisted sheaf simply as a \textit{twisted} sheaf.
\begin{lemma}\label{lem:Azumaya algebras and twisted sheaves}
    Let $\pi:\ms X\to X$ be a $\mu_n$-gerbe representing a class $\alpha\in\H^2(X,\mu_n)$. There is a natural bijection between the set of isomorphism classes of Azumaya algebras on $X$ of degree $n$ such that $\partial([\ms A])=\alpha$ and the set of isomorphism classes of locally free twisted sheaves on $\ms X$ with rank $n$ and trivial determinant.
\end{lemma}
\begin{proof}
    If $\ms E$ is a locally free twisted sheaf on $\ms X$, then $\ms A:=\pi_*\sEnd(\ms E)$ is an Azumaya algebra on $X$. If $\ms E$ has rank $n$ and trivial determinant, then we have $\partial([\ms A])=\alpha$ (see \cite[Proposition 3.1.2.1]{MR2388554}). Conversely, if $\ms A$ is an Azumaya algebra on $X$ of degree $n$ such that $\partial([\ms A])=\alpha$, then $\ms A=\pi_*\sEnd(\ms E)$ for a unique locally free twisted sheaf $\ms E$ of rank $n$ and trivial determinant.
\end{proof}


Let $\ms A$ be an Azumaya algebra on $X$ of degree $n$. There is a trace map
\begin{equation}\label{eq:ur trace map}
    \tr:\ms A\to\ms O_X
\end{equation}
defined by gluing the usual trace maps $\mathrm{M}_n(\ms O_X)\to\ms O_X$ on an \'{e}tale cover. We set $s\ms A:=\ker(\tr)$, so that we have a short exact sequence
\begin{equation}\label{eq:trace SES}
    0\to s\ms A\to\ms A\xrightarrow{\tr}\ms O_X\to 0.
\end{equation}
We define $p\ms A:=\ms A/\ms O_X$ by the short exact sequence
\begin{equation}\label{eq:i SES}
    0\to\ms O_X\xrightarrow{i}\ms A\to p\ms A\to 0
\end{equation}
where $i$ is the canonical inclusion defining the algebra structure on $\ms A$. Consider the pairing $\ms A\otimes\ms A\to\ms O_X$ defined by $s\otimes s'\mapsto \tr(ss')$. This induces a perfect pairing $s\ms A\otimes p\ms A\to \ms O_X$, and hence a canonical isomorphism $s\ms A\cong(p\ms A)^{\vee}$. As a consequence, using the canonical isomorphism $\ms A\cong\ms A^{\vee}$, the composition
\[
    \ms A\cong\ms A^{\vee}\xrightarrow{i^{\vee}}\ms O_X
\]
is the trace map $\tr$, and the composition
\[
    \ms O_X\xrightarrow{\tr^{\vee}}\ms A^{\vee}\cong\ms A
\]
is $i$. The composition $\tr\circ i:\ms O_X\to\ms O_X$ is multiplication by $n$. It follows that if $n$ is invertible in $k$, then both~\eqref{eq:trace SES} and~\eqref{eq:i SES} are split, by $\frac{1}{n}i$ and $\frac{1}{n}\tr$ respectively. Furthermore, the composition $s\ms A\hookrightarrow\ms A\twoheadrightarrow p\ms A$ gives an isomorphism $s\ms A\cong p\ms A$. If $n$ is not invertible, there may not exist such an isomorphism (see Remark \ref{rem:p divides n azumaya algebra}).

Our interest in the sheaves $s\ms A$ and $p\ms A$ is due to their relationship with deformations of $\ms A$. Let $X\subset X'$ be a closed immersion defined by an ideal $ I\subset\ms O_{X'}$ such that $ I^2=0$. As explained by de Jong \cite[\S 3]{MR2060023}, the sheaf $p\ms A$ is isomorphic to the sheaf of derivations $\sDer_{\ms O_X}(\ms A,\ms A)$, and thus controls the deformation theory of $\ms A$, in the following sense: there exists a functorial obstruction class $o(\ms A/X')\in\H^2(X,p\ms A\otimes I)$ which vanishes if and only if there exists an Azumaya algebra $\ms A'$ on $X'$ such that $\ms A'|_X\cong\ms A$.


We will consider the refined deformation problem of finding a lift of $\ms A$ with prescribed class in $\H^2(X,\mu_n)$. We will show that, under certain assumptions, this problem is controlled by $s\ms A$. More precisely, define
\[
    \H^i(X,\ms A\otimes I)_0:=\ker(\tr:\H^i(X,\ms A\otimes I)\to\H^i(X, I)).
\]
There is a map 
\begin{equation}\label{eq:trace to traceless map}
    \H^2(X,s\ms A\otimes I)\to\H^2(X,\ms A\otimes I)_0.
\end{equation}
We will only consider the case when this map is an isomorphism. This holds for instance if the degree of $\ms A$ is invertible in $\ms O_X$ (by the splitting of~\eqref{eq:trace SES}), or if $\H^1(X,\ms O_X)=0$ (eg. if $X$ is an infinitesimal deformation of a K3 surface). Without this assumption a more subtle analysis is required.


\begin{proposition}\label{prop:obstructions for azumaya algebras, A_0 version}
    Let $\ms A$ be an Azumaya algebra on $X$ of degree $n$ such that~\eqref{eq:trace to traceless map} is an isomorphism. Set $\alpha=\partial([\ms A])\in\H^2(X,\mu_n)$. Let $\alpha'\in\H^2(X',\mu_n)$ be a class such that $\alpha'|_X=\alpha$. There exists a functorial obstruction class 
    \[
        o(\ms A/\alpha')\in\H^2(X,s\ms A\otimes I)
    \]
    which vanishes if and only if there exists an Azumaya algebra $\ms A'$ on $X'$ such that $\ms A'|_X\cong\ms A$ and $\partial([\ms A'])=\alpha'$.
\end{proposition}
\begin{proof}  
    Using Lemma \ref{lem:Azumaya algebras and twisted sheaves}, we rephrase the problem in terms of twisted sheaves. Let $\pi:\ms X\to X$ and $\pi':\ms X'\to X'$ be $\mu_n$-gerbes corresponding to $\alpha$ and $\alpha'$. Let $\ms E$ be a locally free twisted sheaf on $\ms X$ of rank $n$ and trivial determinant such that $\pi_*\sEnd(\ms E)\cong\ms A$. There is an obstruction class $o(\ms E)\in\H^2(\ms X,\sEnd(\ms E)\otimes\pi^* I)$ which vanishes if and only if there exists a locally free sheaf $\ms E'$ on $\ms X'$ such that $\ms E'|_{\ms X}\cong\ms E$. Any such deformation $\ms E'$ is necessarily also twisted. The trace map
    \begin{equation}\label{eq:trace for the gerbe}
        \tr:\H^2(\ms X,\sEnd(\ms E)\otimes \pi^* I)\to\H^2(\ms X,\pi^* I)
    \end{equation}
    sends $o(\ms E)$ to the obstruction to deforming the determinant of $\ms E$, which by assumption is trivial. It follows that $o(\ms E)$ is contained in the kernel $\H^2(\ms X,\sEnd(\ms E)\otimes\pi^* I)_0$ of~\eqref{eq:trace for the gerbe}.
    
    Fix an isomorphism $\varphi:\det\ms E\iso\ms O_{\ms X}$. There is a refined obstruction class $o(\ms E,\varphi)\in\H^2(\ms X,s\sEnd(\ms E)\otimes\pi^* I)$ whose vanishing is equivalent to the existence of a tuple $(\ms E',\tau,\varphi')$, where $\ms E'$ is a locally free sheaf on $\ms X'$, $\tau:\ms E'|_{\ms X}\iso\ms E$ is an isomorphism, and $\varphi':\det(\ms E')\iso\ms O_{\ms X'}$ is a trivialization of the determinant of $\ms E'$ such that $\varphi'|_{\ms X}$ is identified via $\det\tau$ with $\varphi$. The class $o(\ms E)$ is the image of $o(\ms E,\varphi)$ under the map
    \begin{equation}\label{eq:trace on coho for gerbe}
        \H^2(\ms X,s\sEnd(\ms E)\otimes\pi^* I)\to\H^2(\ms X,\sEnd(\ms E)\otimes\pi^* I)_0.
    \end{equation}
    Via pushforward, this map is identified with the isomorphism
    \[
        \H^2(X,s\ms A\otimes I)\iso\H^2(X,\ms A\otimes I)_0.
    \]
    In particular,~\eqref{eq:trace on coho for gerbe} is an isomorphism, and so the class $o(\ms E/\varphi)$ does not depend on the choice of $\varphi$. We define $o(\ms A/\alpha')$ to be the image of $o(\ms E/\varphi)$ in $\H^2(X,s\ms A\otimes I)$.
    
    We claim that $o(\ms A/\alpha')$ has the desired properties. Indeed, if $o(\ms A/\alpha')$ vanishes, then there exists a deformation $\ms E'$ of $\ms E$ with trivial determinant, and $\ms A'=\pi'_*\sEnd(\ms E')$ is an Azumaya algebra on $X'$ of degree $n$ such that $\partial([\ms A'])=\alpha'$. Conversely, suppose that there exists such an Azumaya algebra, and let $\ms E'$ be the corresponding locally free sheaf on $\ms X'$. We have that $\ms E'|_{\ms X}\cong\ms E$. It follows that $o(\ms E)=0$, and because~\eqref{eq:trace on coho for gerbe} is an isomorphism, also $o(\ms E,\varphi)=0$. We conclude that $o(\ms A/\alpha')=0$.
\end{proof}



\begin{remark}\label{rem:p divides n azumaya algebra}
  Let us say that an Azumaya algebra $\ms A$ of degree $n$ is \textit{unobstructed} if $\H^2(X,p\ms A)=0$ and \textit{relatively unobstructed} if $\H^2(X,s\ms A)=0$. An unobstructed Azumaya algebra deforms automatically along any square zero thickening of $X$, while a relatively unobstructed Azumaya algebra deforms provided we have a deformation of the corresponding flat cohomology class. If $p$ does not divide $n$, then $s\ms A$ and $p\ms A$ are isomorphic, so $\ms A$ is unobstructed if and only if it is relatively unobstructed. If $p$ divides $n$, then they are not equivalent. Indeed, suppose that $X$ is a smooth projective surface. As observed by de Jong \cite[\S 3]{MR2060023}, because there is an inclusion $\ms O_X\subset s\ms A$, if $\H^0(X,\omega_X)\neq 0$ then we have $\H^0(X,s\ms A\otimes\omega_X)\neq 0$. This group is Serre dual to $\H^2(X,p\ms A)$, which is therefore also nonzero. On the other hand, it may simultaneously be the case that $\H^2(X,s\ms A)=0$. An example is given by taking $\ms A$ to be an Azumaya algebra on a K3 surface such that $\H^0(X,\ms A)=k$ (eg. $\ms A=\sEnd(\ms E)$ for a simple locally free sheaf $\ms E$ of rank $p$). Because $\H^1(X,\ms O_X)=0$ we then have $\H^0(X,p\ms A)=0$. This group is Serre dual to $\H^2(X,s\ms A\otimes\omega_X)=\H^2(X,s\ms A)$, which therefore vanishes. In particular, in this case $s\ms A$ and $p\ms A$ are not isomorphic. 
\end{remark}

\subsection{Existence of relatively unobstructed Azumaya algebras on surfaces}

We will use the following existence result for Azumaya algebras on surfaces. It is a consequence of the solution of the period-index problem for function fields of surfaces, due to de Jong \cite{MR2060023} (when the degree is coprime to $p$) and Lieblich \cite{MR2388554} (in general), combined with Lieblich's results \cite{MR2770444} on the asymptotic properties of moduli spaces of twisted sheaves on surfaces.

\begin{theorem}\label{thm:period index 00}
    Let $X$ be a smooth proper surface over an algebraically closed field $k$ and let $n$ be a positive integer. For any class $\alpha\in\H^2(X,\mu_n)$, there exists an Azumaya algebra $\ms A$ on $X$ of degree $n$ such that $\partial([\ms A])=\alpha$ with the property that $\H^2(X,\ms A)_0=0$.
\end{theorem}
\begin{proof}
    We begin by explaining the translation of the problem into the language of twisted sheaves and gerbes. Let $\pi:\ms X\to X$ be a $\mu_n$-gerbe representing $\alpha$. If $\ms E$ is a locally free twisted sheaf on $\ms X$, then $\ms A:=\pi_*\sEnd(\ms E)$ is an Azumaya algebra on $X$. If $\ms E$ has rank $n$ and trivial determinant, then we have $\partial([\ms A])=\alpha$ (see \cite[Proposition 3.1.2.1]{MR2388554}). Moreover, we have that $\Ext^2(\ms E,\ms E)_0=\H^2(X,\ms A)_0$, where $\Ext^2(\ms E,\ms E)_0=\ker(\tr:\Ext^2(\ms E,\ms E)\to\H^2(\ms X,\ms O_{\ms X}))$. Thus, our task is to show that there exists a locally free twisted sheaf $\ms E$ on $\ms X$ with rank $n$, trivial determinant, and $\Ext^2(\ms E,\ms E)_0=0$.
    
    The proof of this fact uses two inputs. First, we claim that there is a $\mu$-semistable locally free twisted sheaf $\ms G$ with rank $n$ and trivial determinant. Let $K$ be the function field of $X$. We use Lieblich's characteristic free period index theorem for surfaces \cite[Theorem 4.2.2.3]{MR2388554} to find a locally free twisted sheaf $\ms F_K$ on $\ms X_K$ of rank $m$, where $m$ is the period of $\alpha$ (the order of the image of $\alpha$ in the Brauer group). Let $\ms F$ be the reflexive hull of the pushforward of $\ms F_K$ to $\ms X$. The sheaf $\ms F$ is 1-twisted, locally free, and $\mu$-stable. Taking an appropriate elementary transformation, we may arrange so that the determinant of $\ms F$ is trivial (see eg. the proof of \cite[Proposition 3.2.3.4]{MR2388554}). The sheaf $\ms G=\ms F^{\oplus n/m}$ is a twisted sheaf of rank $n$ and trivial determinant. Moreover, $\ms G$ is polystable, and in particular semistable, as desired.
    
    We now wish to produce the desired $\ms E$. Fix an integer $\Delta$. Consider the moduli space $M(\Delta)$ of $\mu$-semistable twisted sheaves on $\ms X$ with rank $n$, discriminant $\Delta$, and trivialized determinant (denoted $\mathbf{Tw}^{\mathrm{ss}}(n,\ms O_{\ms X},\Delta)$ in \cite[Notation 4.2.2]{MR2770444}). By \cite[Theorem 5.3.1]{MR2770444}, if $\Delta$ is sufficiently large, then any irreducible component of $M(\Delta)$ contains a point corresponding to a twisted sheaf which is locally free and satisfies $\Ext^2(\ms E,\ms E)_0$ (here, we note that while the reference \cite{MR2770444} is written under the blanket assumption that $n$ is coprime to the characteristic, this assumption is not needed for the proof of the cited result). It remains to show that the spaces $M(\Delta)$ are nonempty for sufficiently large $\Delta$. This follows from the existence of the sheaf $\ms G$ above: we have that the moduli space $M(\Delta)$ is nonempty for some $\Delta$, which implies their nonemptiness for all sufficiently large $\Delta$ (see eg. the proof of \cite[Proposition 3.2.3.4]{MR2388554}). We conclude the result.
    
\end{proof}

\begin{remark}
    For our applications to lifting Brauer classes in this paper, we only need the special case of Theorem \ref{thm:period index 00} when $X$ is a K3 surface and $\alpha\in\H^2(X,\mu_n)$ is a class having order $n$. In this case, the proof of \ref{thm:period index 00} may be significantly shortened, and in particular we may avoid the analysis of asymptotic properties of moduli spaces of twisted sheaves in \cite{MR2770444}. Indeed, as in the proof of \ref{thm:period index 00}, let $\ms X$ be a $\mu_n$-gerbe representing $\alpha$. Using \cite[Theorem 4.2.2.3]{MR2388554} and taking the reflexive hull and an elementary transformation we find a locally free twisted sheaf on $\ms X$ of rank $n$ and trivial determinant. As the period of $\alpha$ is $n$, such a sheaf is \textit{simple}, in the sense that $k=\mathrm{End}(\ms E)$. Consider the Azumaya algebra $\ms A:=\pi_*\sEnd(\ms E)$. The map $i$ induces an isomorphism
    \begin{equation}\label{eq:map on global sections simple}
        k=\H^0(X,\ms O_X)\iso\H^0(X,\ms A)
    \end{equation}
    on global sections. Because $X$ is K3,~\eqref{eq:map on global sections simple} is Serre dual to the trace map
    \[
        \tr:\H^2(X,\ms A)\to\H^2(X,\ms O_X).
    \]
    It follows that the trace map is injective, so $\H^2(X,\ms A)_0=0$.
\end{remark}


Let $k$ be an algebraically closed field and let $(R,\mf m)$ be a complete noetherian local ring with residue field $k$. Let $n$ be a positive integer.

\begin{proposition}\label{prop:02}
   Let $\mf X$ be a smooth proper surface over $R$ and write $X_i=\mf X\otimes_R R/\mf m^{i+1}$. Suppose that $\H^1(X_0,\ms O_{X_0})=0$. If $\left\{\alpha_i\in\H^2(X_i,\mu_n)\right\}_{i\geq 0}$ is a compatible system of cohomology classes, then there exists a unique class $\widetilde{\alpha}\in\H^2(\mf X,\mu_n)$ such that $\widetilde{\alpha}|_{X_i}=\alpha_i$ for all $i$.
\end{proposition}
\begin{proof}
    We first note that, using flatness and filtering by powers of $\mathfrak{m}$, the assumption that $\H^1(X_0,\ms O_{X_0})$ implies that $\H^1(X_i,\ms O_{X_i})$ and $\H^1(X_i,\mathfrak{m}^{i+1}\ms O_{X_i})$ vanish for all $i\geq 0$.
    
    
    By Theorem \ref{thm:period index 00}, we may find an Azumaya algebra $\ms A$ on $X$ of degree $n$ such that $\partial([\ms A])=\alpha_0$ and such that $\H^2(X,\ms A)_0=0$. Set $\ms A_0=\ms A$. Applying Proposition \ref{prop:obstructions for azumaya algebras, A_0 version} repeatedly, we find for each $i\geq 1$ an Azumaya algebra $\ms A_i$ on $X_i$ such that $\ms A_{i}|_{X_{i-1}}\cong \ms A_{i-1}$ and such that $\partial([\ms A_i])=\alpha_i$. By Grothendieck's existence theorem there exists an Azumaya algebra $\widetilde{\ms A}$ on $\mf X$ restricting to $\ms A_i$ on $X_i$. The class $\widetilde{\alpha}:=\partial([\widetilde{\ms A}])\in\H^2(\mf X,\mu_n)$ restricts to $\alpha_i$ for each $i$.
  
  We now show the uniqueness. By subtracting, we are reduced to showing that if $\widetilde{\alpha}\in\H^2(\mf X,\mu_n)$ is a class such that $\widetilde{\alpha}|_{X_i}=0$ for all $i\geq 0$ then $\widetilde{\alpha}=0$. Let $\ms X\to\mf X$ be a corresponding $\mu_n$-gerbe. A trivialization of the restriction of the gerbe to $X_i$ gives an $R/\mathfrak{m}^{i+1}$-point of the Weil restriction $f_*\ms X$, where $f:\mf X\to\Spec R$ is the structure morphism. By assumption, we may find a compatible system of such trivializations, and hence a compatible family of $R/\mathfrak{m}^{i+1}$-points of $f_*\ms X$. By Theorem 1.5 of \cite{MR2239345}, $f_*\ms X$ is algebraic, so this family comes from an $R$-point of $f_*\ms X$. We conclude that $\ms X\to\mathfrak{X}$ is a trivial gerbe, and hence $\widetilde{\alpha}=0$.
  
  
  
\end{proof}

\subsection{A representability result}\label{ssec:proof of representability theorem}

We now prove Theorem \ref{thm:representability for R2}. We recall the notation. Let $f:X\to S$ be a morphism of algebraic spaces. Let $n$ be a nonzero integer. We let $Rf_*$ denote the derived pushforward from the category of sheaves of abelian groups on the big flat (fppf) site of $X$ to that of $S$.
In \cite[Theorem 2.1.6]{BL17} it is shown that if $X\to S$ is a family of K3 surfaces and $p$ is a prime then $R^2f_*\mu_p$ is an algebraic space. Theorem \ref{thm:representability for R2} is a generalization of this result, and the idea of proof is the same.

\begin{proof}[Proof of Theorem \ref{thm:representability for R2}]
    We follow closely the proof of \cite[Theorem 2.1.6]{BL17}. Write $\ms S=R^2f_*\mu_n$. We first claim that the diagonal $\ms S\to\ms S\times_S\ms S$ is representable by closed immersions of finite presentation. In the case when $f$ is a family of K3 surfaces and $n=p$ is a prime, this is shown in \cite[Proposition 2.17]{BL17}. Replacing $p$ with $n$ and using our assumption that $R^1f_*\mu_n=0$, the proof of loc. cit. applies unchanged to give the result.
    
    Let $\Az$ be the stack on $S$ whose objects over an $S$-scheme $T\to S$ are Azumaya algebras $\ms A$ on $X_T$ such that for every geometric point $t\to T$ the restriction $\ms A_t:=\ms A|_{X_t}$ has degree $n$ and the map
    \[
        \tr:\H^2(X_t,\ms A_{t})\to\H^2(X_t,\ms O_{X_t})
    \]
    is injective. As described in the proof of \cite[Theorem 2.1.6]{BL17}, $\Az$ is an Artin stack locally of finite presentation over $S$, and the nonabelian boundary map $\partial$~\eqref{eq:boundary map for Azumaya algebras} gives rise to a map
    \[
        \chi:\Az\to \ms S.
    \]
    Arguing as in \cite[Proposition 2.1.10]{BL17} and using Proposition \ref{prop:obstructions for azumaya algebras, A_0 version} (and our assumption on the vanishing of $\H^1(X_t,\ms O_{X_t})$), we deduce that the map $\chi$ is representable by smooth Artin stacks. Furthermore, by Theorem \ref{thm:period index 00}, $\chi$ is surjective on geometric points.
    
    Now, any smooth cover of $\Az$ by a scheme gives rise to a smooth cover of $\ms S$ by a scheme. We have shown that the diagonal of $\ms S$ is representable, so $\ms S$ is an Artin stack of finite presentation over $S$. But $\ms S$ is a sheaf, so by \cite[04SZ]{stacks-project} $\ms S$ is an algebraic space.
\end{proof}






\section{Deformations of cohomology classes for $\mu_n$ and $\mathbf{G}_m$}\label{sec:deformations of flat coho classes}


Let $X$ be a scheme and let $n$ be a positive integer. We are interested in the flat cohomology groups of the group schemes $\mathbf{G}_m$ and $\mu_n$. These groups are related by the Kummer sequence
\begin{equation}\label{eq:Kummer sequence}
    1\to\mu_n\to\mathbf{G}_m\xrightarrow{\cdot n}\mathbf{G}_m\to 1
\end{equation}
which is exact in the flat topology. If $A$ is a sheaf of abelian groups on $X$, we let $A(n)$ denote the complex
\begin{equation}\label{eq:definition of complex}
  A(n)=[A\xrightarrow{\cdot n}A]
\end{equation}
where the right hand side has terms in degrees 0 and 1. With this notation, we interpret the Kummer sequence~\eqref{eq:Kummer sequence} as a quasi--isomorphism $\mu_n\xrightarrow{\sim}\mathbf{G}_m(n)$ of complexes of sheaves on the flat site of $X$. In particular, this gives a canonical resolution of $\mu_n$ by a complex of smooth group schemes. 
By a theorem of Grothendieck \cite[Th\'{e}or\`{e}me 11.7]{MR0244271},
we have identifications
\begin{equation}\label{eq:isomorphism1}
  \H^m_{\fl}(X,\mu_{n})=\H^m_{\fl}(X,\mathbf{G}_m(n))=\H^m_{\et}(X,\mathbf{G}_{m}(n)).
\end{equation}

We consider the following deformation situation. Let $X\hookrightarrow X'$ be an infinitesimal thickening whose defining ideal $I$ satisfies $I^2=0$.
Let $\alpha\in\H^2(X,\mu_n)$ be a flat cohomology class and let $\alpha_{\Br}\in\H^2(X,\mathbf{G}_m)$ be the image of $\alpha$ under the map $\H^2(X,\mu_n)\to\H^2(X,\mathbf{G}_m)$. We consider the problem of deforming $\alpha$ and $\alpha_{\Br}$ to $X'$. (We consider such questions in a slightly more general setting in Appendix \ref{appendix}, where we also consider the more refined question of deforming gerbes. Here we are only interested in the obstruction classes, for which a purely cohomological approach suffices). Consider the standard short exact sequence
\begin{equation}\label{eq:exp SES}
    0\to I\to\mathbf{G}_{m,X'}\to\mathbf{G}_{m,X}\to 1
\end{equation}
of \'{e}tale sheaves, where the left map is the truncated exponential $f\mapsto 1+f$. Taking cohomology, we find an exact sequence
\[
    \ldots\to\H^2(X, I)\to\H^2(X',\mathbf{G}_{m,X'})\to\H^2(X,\mathbf{G}_{m,X})\xrightarrow{\delta}\H^3(X, I)\to\ldots.
\]
We define $o(\alpha_{\Br}/X'):=\delta(\alpha_{\Br})\in\H^3(X,I)$. This class vanishes if and only if there exists a class $\alpha'_{\Br}\in\H^2(X',\mathbf{G}_{m,X'})$ such that $\alpha'_{\Br}|_X=\alpha_{\Br}$. We similarly define an obstruction class for $\alpha$ as follows. Multiplication by $n$ on~\eqref{eq:exp SES} gives a short exact sequence
\begin{equation}\label{eq:exp SES, mult by n}
    0\to I(n)\to\mathbf{G}_{m,X'}(n)\to\mathbf{G}_{m,X}(n)\to 1
\end{equation}
of complexes. We take cohomology and apply the identifications~\eqref{eq:isomorphism1} to obtain a long exact sequence
\[
    \ldots\to\H^2(X, I(n))\to\H^2(X',\mu_{n,X'})\to\H^2(X,\mu_{n,X})\xrightarrow{\delta'}\H^3(X, I(n))\to\ldots.
\]
We put $o(\alpha/X'):=\delta'(\alpha)\in\H^3(X,I(n))$. This class vanishes if and only if there exists a class $\alpha'\in\H^2(X',\mu_{n,X'})$ such that $\alpha'|_X=\alpha$. These obstructions are compatible: we have a commuting square
\[
    \begin{tikzcd}
        \H^2(X,\mu_{n,X})\arrow{r}{\delta'}\arrow{d}&\H^3(X,I(n))\arrow{d}\\
        \H^2(X,\mathbf{G}_{m,X})\arrow{r}{\delta}&\H^3(X,I)
    \end{tikzcd}
\]
where the right vertical arrow is induced by the projection $I(n)\to I$ onto the degree 0 term. It follows that $o(\alpha/X')\mapsto o(\alpha_{\Br}/X')$ under the right vertical arrow.

\begin{remark}
 The groups $\H^m(X, I(n))$ depend strongly on the behavior of multiplication by $n$ on $ I$. If multiplication by $n$ is invertible on $ I$, then the complex $ I(n)$ is quasi-isomorphic to 0. We therefore have $\H^m(X, I(n))=0$ for all $i$, and hence classes in $\H^2(X,\mu_n)$ deform uniquely over $X'$. On the other hand, suppose that multiplication by $n$ on $I$ is zero. We then have
    \begin{equation}\label{eq:isomorphism11}
         I(n)= I[-1]\oplus  I.
    \end{equation}
    Therefore $\H^3(X, I(n))=\H^2(X, I)\oplus \H^{3}(X, I)$, and classes in $\H^2(X,\mu_n)$ may (at least a-priori) be obstructed. On the other hand, the relative deformation problem with respect to the embedding $\mu_n\subset\mathbf{G}_m$ is more uniform in $n$.
\end{remark}

Suppose that $nI=0$, and let
    \[
        \pi_1:\H^2(X, I)\oplus\H^3(X, I)\to\H^2(X, I)
    \]
be the projection onto the first factor. If $\alpha\in\H^2(X,\mu_n)$, then we define
\begin{equation}\label{eq:refined obstruction class}
    \widetilde{o}(\alpha/X'):=\pi_1(o(\alpha/X'))=\pi_1(\delta'(\alpha)).
\end{equation}
\begin{remark}
  The class $\widetilde{o}(\alpha/X')$ has the following geometric interpretation. Let $\ms X$ be a $\mu_n$-gerbe on $X$ with cohomology class $\alpha$ and let $\ms X_{\Br}=\ms X\wedge_{\mu_n}\mathbf{G}_m$ be the induced $\mathbf{G}_m$-gerbe, which has class $\alpha_{\Br}$. Suppose that $o(\alpha_{\Br}/X')=0$, and fix a deformation $\ms X'_{\Br}$ of $\ms X_{\Br}$ over $X'$. We then have that $\widetilde{o}(\alpha/X')$ is equal to the obstruction class $o(\ms X/\ms X'_{\Br})\in\H^2(X,I)$ defined in Proposition \ref{prop:obstructions for relative gerbe change of group situation}. In particular, $o(\ms X/\ms X'_{\Br})$ depends only on $\alpha$ in this case. To see the equality, consider the short exact sequence
  \[
    0\to I[-1]\to I(n)\to I\to 0.
  \]
  Because $nI=0$, this sequence is split. Thus, in the long exact sequence on cohomology all boundary maps are zero, and we obtain a short exact sequence
  \[
        0\to\H^2(X, I)\to\H^2(X, I(n))\to\H^3(X, I)\to 0.
    \]
    By Remark \ref{rem:functoriality of obstruction class, take 2}, we have that $o(\ms X/\ms X'_{\Br})\mapsto o(\ms X/X')=o(\alpha/X')$ under the left arrow. The claim follows.
\end{remark}

\subsection{Kodaira--Spencer classes}\label{sec:Kodaira--Spencer}

Let $X$ be a reduced scheme over a field $k$ of characteristic $p>0$. Let $A$ be an Artinian local ring with maximal ideal $\mathfrak{m}$ satisfying $\mathfrak{m}^2=0$ and residue field identified with $k$. Let $X'$ be a flat scheme over $A$ together with a specified isomorphism $X'\otimes_{A}k\cong X$. Thus, $X\subset X'$ is an infinitesimal thickening defined by the square zero ideal sheaf $\mathfrak{m}\ms O_X=\mathfrak{m}\otimes_{A}\ms O_{X'}$. These data are summarized in the cartesian diagram
\begin{equation}\label{eq:deformation diagram with one square}
    \begin{tikzcd}
            X\arrow{d}\arrow[hook]{r}&X'\arrow{d}\\
            \Spec k\arrow[hook]{r}&\Spec A
    \end{tikzcd}
\end{equation}
where the horizontal arrows are closed immersions defined by the square-zero ideals $\mathfrak{m}$ and $\mathfrak{m}\ms O_X$.

Let $n$ be a positive integer. We consider the exact sequence
\begin{equation}\label{eq:def of nu}
    \mathbf{G}_{m,X}\xrightarrow{\cdot n}\mathbf{G}_{m,X}\to\mathbf{G}_{m,X}/\mathbf{G}_{m,X}^{\times n}\to 1
\end{equation}
of \'{e}tale sheaves on $X$. Here, the right term denotes the quotient sheaf in the \'{e}tale topology. The corresponding quotient in the flat topology vanishes. If $n$ is coprime to $p$, then the \'{e}tale quotient also vanishes. The sequence~\eqref{eq:def of nu} corresponds to a map
\begin{equation}\label{eq:map of complexes to the quotient}
    \mathbf{G}_{m,X}(n)\to\mathbf{G}_{m,X}/\mathbf{G}_{m,X}^{\times n}[-1]
\end{equation}
of complexes of \'{e}tale sheaves. Let
\begin{equation}\label{eq:Upsilon}
    \Upsilon:\H^2(X_{\fl},\mu_n)\to\H^1(X_{\et},\mathbf{G}_{m,X}/\mathbf{G}_{m,X}^{\times n})
\end{equation}
be the map obtained by taking cohomology of~\eqref{eq:map of complexes to the quotient} and using the identification~\eqref{eq:isomorphism1}. We remark that if $n$ is a power of $p$, then because $X$ is reduced, the left map of~\eqref{eq:def of nu} is injective,~\eqref{eq:map of complexes to the quotient} is a quasi-isomorphism, and~\eqref{eq:Upsilon} is an isomorphism. Let $p^r$ be the largest power of $p$ which divides $n$. We then have a commuting square
\begin{equation}\label{eq:Upsilon part 2 redux}
    \begin{tikzcd}
        \H^2(X,\mu_n)\arrow{d}[swap]{\cdot n/p^r}\arrow{r}{\Upsilon}&\H^1(X,\mathbf{G}_m/\mathbf{G}_m^{\times n})\isor{d}{}\\
        \H^2(X,\mu_{p^r})\arrow{r}{\sim}&\H^1(X,\mathbf{G}_{m,X}/\mathbf{G}_{m,X}^{\times p^r})
    \end{tikzcd}
\end{equation}
where the right vertical arrow is induced by the natural quotient map. This map is an isomorphism because, for any $m$ coprime to $p$, multiplication by $m$ on $\mathbf{G}_m$ is surjective in the \'{e}tale topology.

Assume that $n\mathfrak{m}=0$ (equivalently, $p^r\mathfrak{m}=0$). As $X$ is reduced, the restriction map $\mu_{n,X'}\to\mu_{n,X}$ of \'{e}tale sheaves is surjective. Applying the snake lemma to~\eqref{eq:exp SES, mult by n} yields an exact sequence
\begin{equation}\label{eq:deformation SES for nu}
    0\to\mathfrak{m}\ms O_X\to\mathbf{G}_{m,X'}/\mathbf{G}_{m,X'}^{\times n}\to\mathbf{G}_{m,X}/\mathbf{G}_{m,X}^{\times n}\to 1.
\end{equation}
Taking cohomology, we get an exact sequence
\begin{equation}\label{eq:coho sequence for nu}
    \H^1(X',\mathbf{G}_{m,X'}/\mathbf{G}_{m,X'}^{\times n})\to \H^1(X,\mathbf{G}_{m,X}/\mathbf{G}_{m,X}^{\times n})\xrightarrow{\delta''}\H^2(X,\mathfrak{m}\ms O_X).
\end{equation}

\begin{proposition}\label{prop:computation of relative obstruction class 2}
  Suppose that $n\mathfrak{m}=0$. If $\alpha\in\H^2(X,\mu_n)$ is a flat cohomology class, then $\widetilde{o}(\alpha/X')=\delta''\circ\Upsilon(\alpha)$.
\end{proposition}
\begin{proof}
    Because $n\mathfrak{m}=0$, there is a map of complexes $\mathfrak{m}\ms O_X(n)\to\mathfrak{m}\ms O_X[-1]$ given by the identity in degree 1. Combining this map with~\eqref{eq:exp SES, mult by n},~\eqref{eq:deformation SES for nu}, and the maps~\eqref{eq:map of complexes to the quotient} we find a commutative diagram
    \[
        \begin{tikzcd}
            0\arrow{r}&\mathfrak{m}\ms O_X(n)\arrow{d}\arrow{r}&\mathbf{G}_{m,X'}(n)\arrow{d}\arrow{r}&\mathbf{G}_{m,X}(n)\arrow{d}\arrow{r}&1\\
            0\arrow{r}&\mathfrak{m}\ms O_X[-1]\arrow{r}&\mathbf{G}_{m,X'}/\mathbf{G}_{m,X'}^{\times n}[-1]\arrow{r}&\mathbf{G}_{m,X}/\mathbf{G}_{m,X}^{\times n}[-1]\arrow{r}&1
        \end{tikzcd}
    \]
    with exact rows. Taking cohomology gives a commutative diagram
    \[
        \begin{tikzcd}
            \H^2(X,\mu_{n,X})\arrow{d}[swap]{\Upsilon}\arrow{r}{\delta'}&\H^3(X,\mathfrak{m}\ms O_X(n))\arrow{d}{\pi_1}\\
            \H^1(X,\mathbf{G}_{m,X}/\mathbf{G}_{m,X}^{\times n})\arrow{r}{\delta''}&\H^2(X,\mathfrak{m}\ms O_X).
        \end{tikzcd}
    \]
    By definition, $\widetilde{o}(\alpha/X')$ is equal to $\pi_1(\delta(\alpha))$, and we obtain the result.
\end{proof}

We now assume that $X$ is smooth over $k$ and that $A$ is the ring of dual numbers $k[\varepsilon]:=k[\varepsilon]/\varepsilon^2$. Thus,~\eqref{eq:deformation diagram with one square} becomes
\[
    \begin{tikzcd}
        X\arrow[hook]{r}\arrow{d}&X'\arrow{d}\\
        \Spec k\arrow[hook]{r}&\Spec k[\varepsilon].
    \end{tikzcd}
\]
Consider the exact sequence
\begin{equation}\label{eq:KS extension class}
    0\to \ms O_X\xrightarrow{d\varepsilon}\Omega^1_{X'/k[\varepsilon]}\to\Omega^1_{X/k}\to 0.
\end{equation}
The Kodaira--Spencer class $\tau_{X'}$ of the deformation $X'$ is the extension class of this sequence in $\Ext^1_X(\Omega^1_{X/k},\ms O_X)=\H^1(X,T_X)$.
There is a canonical cup product pairing
\begin{equation}\label{eq:cup product}
\_\cup\_:\H^1(X,\Omega^1_X)\otimes \H^1(X,T_X)\to\H^2(X,\ms O_X)
\end{equation}
and the map
\begin{equation}\label{eq:cup product part 2}
    \_\cup\tau_{X'}:\H^1(X,\Omega^1_X)\to\H^2(X,\ms O_X)
\end{equation}
is the boundary map coming from the long exact sequence on cohomology of~\eqref{eq:KS extension class}.
Consider the map
\begin{equation}\label{eq:map11}
	\dlog:\mathbf{G}_m\to\Omega^1_X
\end{equation}
of \'{e}tale sheaves on $X$ given on sections by $f\mapsto df/f$. Let $n$ be a positive integer which is divisible by $p$. Any $p$th power is killed by $\dlog$, so~\eqref{eq:map11} descends to a map
\begin{equation}\label{eq:map12}
	\dlog:\mathbf{G}_m/\mathbf{G}_m^{\times n}\to\Omega^1_X
\end{equation}
where $\mathbf{G}_m/\mathbf{G}_m^{\times n}$ is the quotient sheaf for the \'{e}tale topology. Composing~\eqref{eq:map of complexes to the quotient} with~\eqref{eq:map12} and taking cohomology we get a map
\begin{equation}\label{eq:map15}
	\dlog:\H^2(X,\mu_n)\to\H^1(X,\Omega^1_X).
\end{equation}

Fix a class $\alpha\in\H^2(X,\mu_n)$. The following result computes the class $\widetilde{o}(\alpha/X')\in\H^2(X,\varepsilon\ms O_X)$~\eqref{eq:refined obstruction class} in terms of the Kodaira--Spencer class $\tau_{X'}$ of the deformation $X'$. This computation also appears in a paper of Nygaard \cite[pg. 223]{MR723215}. The corresponding result for invertible sheaves is standard (see eg. \cite[Proposition 1.14]{Ogus78}). 
\begin{proposition}\label{prop:cup product}
    Suppose that $n$ is divisible by $p$. For any class $\alpha\in\H^2(X,\mu_n)$, the class $\widetilde{o}(\alpha/X')\in\H^2(X,\varepsilon\ms O_X)\cong\varepsilon\H^2(X,\ms O_X)$ is equal to $\varepsilon(\dlog(\alpha)\cup\tau_{X'})$.
\end{proposition}
\begin{proof}
    We have a commutative diagram
    \[
        \begin{tikzcd}
            0\arrow{r}&\varepsilon\ms O_X\isor{d}{\varepsilon^{-1}}\arrow{r}{\exp}&\mathbf{G}_{m,X'}/\mathbf{G}_{m,X'}^{\times n}\arrow{d}{\dlog}\arrow{r}&\mathbf{G}_{m,X}/\mathbf{G}_{m,X}^{\times n}\arrow{d}{\dlog}\arrow{r}&1\\
            0\arrow{r}&\ms O_X\arrow{r}{d\varepsilon}&\Omega^1_{X'/k[\varepsilon]}\arrow{r}&\Omega^1_{X/k}\arrow{r}&0
        \end{tikzcd}
    \]
    with exact rows. Taking cohomology we find a commutative diagram
	\[
	    \begin{tikzcd}
	        \H^1(X,\mathbf{G}_m/\mathbf{G}_{m}^{\times n})\arrow{d}[swap]{\dlog}\arrow{r}{\delta''}&\H^2(X,\varepsilon\ms O_X)\\
	        \H^1(X,\Omega^1_X)\arrow{r}{\_\cup\tau_{X'}}&\H^2(X,\ms O_X)\isor{u}{\varepsilon}.
	    \end{tikzcd}
	\]
	We conclude the result from Proposition \ref{prop:computation of relative obstruction class 2}.
\end{proof}

\section{K3 surfaces in positive characteristic}\label{sec:K3 surfaces in positive characteristic}

A \textit{K3 surface} is a smooth projective surface $X$ over such that $\omega_X\cong\ms O_X$ and $\H^1(X,\ms O_X)=0$. In this section we collect some facts about K3 surfaces in positive characteristic and their cohomology.

We assume that $k$ is algebraically closed of  characteristic $p>0$ and $X$ is a K3 surface over $k$. We recall the definition of the height of $X$ (see eg. \cite[\S18.3]{Huy06}). The \textit{formal Brauer group} of $X$ is the functor $\widehat{\Br}_X:=\Phi^2_{\mathbf{G}_m/X}$ on the category of Artinian local $k$-algebras, given by
\[
    A\mapsto\ker\left(\H^2(X_A,\mathbf{G}_m)\to\H^2(X,\mathbf{G}_m)\right)
\]
where $X_A=X\otimes_kA$ is the trivial deformation of $X$ over $A$. Due to the equalities $h^3(X,\ms O_X)=0$ and $h^2(X,\ms O_X)=1$, a result of Artin--Mazur \cite[Corollary 2.12]{MR0457458} implies that $\widehat{\Br}_X$ is prorepresentable by a smooth one dimensional commutative formal group over $k$. Such objects are classified up to isomorphism by their \textit{height}, which is a discrete invariant $h$, equal either to a positive integer or to $\infty$. The height is determined as follows. Fix an isomorphism $\widehat{\Br}_X\cong\Spf k[[s]]$. The multiplication by $p$ map $[p]:\widehat{\Br}_X\to\widehat{\Br}_X$ corresponds to a map $k[[s]]\to k[[s]]$, which we also denote by $[p]$. If $[p](s)\neq 0$, we define the height of $\widehat{\Br}_X$ to be the smallest integer $h$ such that
\[
    [p](s)=\lambda s^{p^h}+(\text{higher order terms})
\]
for some nonzero $\lambda\in k$. If $[p](s)=0$, we set $h=\infty$.
The height of the K3 surface $X$ is defined to be the height of the formal Brauer group $\widehat{\Br}_X$.

Let $\widehat{\Br}_X[n]$ denote the kernel of multiplication by $n$ on the formal Brauer group. The height may be equivalently described in terms of the formal scheme prorepresenting $\widehat{\Br}_X[p]$: $X$ has height $h<\infty$ if and only if
\[
    \widehat{\Br}_X[p]\cong\Spf k[[s]]/(s^{p^h})
\]
while $X$ has height $h=\infty$ if and only if
\[
    \widehat{\Br}_X[p]\cong\Spf k[[s]].
\]

We say that $X$ has \textit{finite height} if $h\neq\infty$. In this case, $h$ must lie in the range $1\leq h\leq 10$. If $h=1$, we say that $X$ is \textit{ordinary}. A K3 surface is ordinary if and only if the map $F:\H^2(X,\ms O_X)\to\H^2(X,\ms O_X)$ induced by the absolute Frobenius of $X$ is an isomorphism (see \cite[\S5]{MR1776939}). If $h=\infty$, then we say that $X$ is \textit{supersingular}. In the ordinary and supersingular cases, we have the following explicit descriptions of the formal Brauer group:
\[
    h(X)=1\Leftrightarrow\widehat{\Br}_X\cong\widehat{\mathbf{G}}_m\hspace{1cm}\text{and}\hspace{1cm}h(X)=\infty\Leftrightarrow\widehat{\Br}_X\cong\widehat{\mathbf{G}}_a.
\]
Accordingly, in these two cases the group structures on $\widehat{\Br}_X[p]$ are given by
\[
    h(X)=1\Leftrightarrow\widehat{\Br}_X[p]\cong\mu_p\hspace{1cm}\text{and}\hspace{1cm}h(X)=\infty\Leftrightarrow\widehat{\Br}_X[p]\cong\widehat{\mathbf{G}}_a.
\]


Suppose that $X$ is supersingular. In this case, there is a further discrete invariant of $X$, which may be characterized as follows. The flat cohomology group $\H^2(X,\mathbf{Z}_p(1)):=\varprojlim\H^2(X,\mu_{p^n})$ is a free $\mathbf{Z}_p$-module of rank 22, and is equipped with a natural $\mathbf{Z}_p$-valued pairing. The \textit{Artin invariant} of $X$ is the integer $\sigma_0$ such that
\[
    \mathrm{disc}\,\H^2(X,\mathbf{Z}_p(1))=-p^{2\sigma_0}.
\]
We have $1\leq\sigma_0\leq 10$. We say that $X$ is \textit{superspecial} if $\sigma_0=1$. The height may only rise upon specialization, and the Artin invariant can only fall upon specialization. Thus, the ordinary locus is open in moduli, and the superspecial locus is closed in moduli. In fact, there is up to isomorphism only one superspecial K3 surface \cite{Ogus78}. If $X$ has finite height, then we formally declare $\sigma_0(X)=\infty$.

\subsection{De Rham cohomology}
The second de Rham cohomology group $\H^2_{\dR}(X)$ has dimension 22, and is equipped with the \textit{Hodge} and \textit{conjugate} filtrations
\[
  0\subset F^2_H\subset F^1_H\subset F^0_H=\H^2_{\dR}(X)\hspace{1cm} 0\subset F^2_C\subset F^1_C\subset F^0_C=\H^2_{\dR}(X).
\]
Both $F^2_H$ and $F^2_C$ have dimension 1, and $F^1_H$ and $F^1_C$ have dimension 21. Under the cup product pairing on $\H^2_{\dR}(X)$, we have $(F^2_H)^{\perp}=F^1_H$ and $(F^2_C)^{\perp}=F^1_C$. The relative positions of $F^i_H$ and $F^i_C$ give some information on the invariants of $X$. We have that $X$ is ordinary if and only if $F^2_H\cap F^1_C=0$ if and only if $F^1_H\cap F^2_C=0$, and $X$ is superspecial if and only if $F^2_H=F^2_C$ if and only if $F^1_H=F^1_C$ (see \cite[\S8]{MR1776939}).

For future use, we list some explicit cohomology groups corresponding to the intersections of the various pieces of the Hodge and conjugate filtrations. Define
\[
    Z\Omega^i_X:=\ker(d:\Omega^i_X\to\Omega^{i+1}_X)\hspace{1cm}\text{\and}\hspace{1cm}B\Omega^i_X:=\mathrm{im}(d:\Omega^{i-1}_X\to\Omega^i_X).
\]
\begin{lemma}\label{lem:some natural identifications}
    We have natural identifications
    \begin{enumerate}
        \item $F^1_H\cap F^1_C=\H^1(X,Z\Omega^1_X)$,
        \item $F^1_H\cap F^2_C=\H^1(X,B\Omega^1_X)$,
        \item $F^2_H\cap F^1_C=\H^0(X,B\Omega^2_X)$, and
        \item $F^2_H\cap F^2_C=\H^0(X,\Omega^1_X/B\Omega^1_X)$.
    \end{enumerate}
\end{lemma}
\begin{proof}
    (1) is proven by Ogus in \cite[Proposition 1.2]{Ogus78}. For (2), consider the short exact sequence
    \[
        0\to\ms O_X^p\to\ms O_X\xrightarrow{d}B\Omega^1_X\to 0.
    \]
    Taking cohomology and using $\H^1(X,\ms O_X)=0$, we obtain a diagram
    \[
        \begin{tikzcd}
            0\arrow{r}&\H^1(X,B\Omega^1_X)\arrow{r}\arrow[hook]{d}&\H^2(X,\ms O_X^p)\arrow{r}\arrow[hook]{d}&\H^2(X,\ms O_X)\arrow{r}\arrow[equals]{d}&0\\
            0\arrow{r}&\H^2(X,\tau_{\geq 1}\Omega^{\bullet}_X)\arrow{r}&\H^2_{\dR}(X)\arrow{r}&\H^2(X,\ms O_X)\arrow{r}&0
        \end{tikzcd}
    \]
    where the vertical arrows are induced by the natural maps of complexes. Here, the middle verticle arrow is injective because of the degeneration of the conjugate spectral sequence, which implies the injectivity of the left arrow. The image of the middle arrow is $F^2_C$, and $\H^2(X,\tau_{\geq 1}\Omega^{\bullet}_X)=F^1_H$, so we conclude the result. For (3), we take cohomology of the exact sequence
    \begin{equation}\label{eq:SES with a B}
        0\to Z\Omega^1_X\to\Omega^1_X\xrightarrow{d}B\Omega^2_X\to 0
    \end{equation}
    which gives
    \begin{equation}\label{eq:SES with a B on coho}
        0\to\H^0(X,B\Omega^2_X)\to\H^1(X,Z\Omega^1_X)\to\H^1(X,\Omega^1_X).
    \end{equation}
    This identifies $\H^0(X,B\Omega^2_X)$ with the kernel of the map $F^1_H\cap F^1_C\to F^1_H/F^2_H$, which is exactly $F^2_H\cap F^1_C$. Finally, we show (4). Taking cohomology of the exact sequences
    \[
        0\to B\Omega^1_X\to\Omega^1_X\to\Omega^1_X/B\Omega^1_X\to 0
    \]
    and~\eqref{eq:SES with a B}, we find a diagram
    \[
        \begin{tikzcd}
            0\arrow{r}&\H^0(X,\Omega^1_X/B\Omega^1_X)\arrow[hook]{d}\arrow{r}&\H^1(X,B\Omega^1_X)\arrow[hook]{d}\arrow{r}&\H^1(X,\Omega^1_X)\arrow[equals]{d}\\
            0\arrow{r}&\H^0(X,B\Omega^2_X)\arrow{r}&\H^1(X,Z\Omega^1_X)\arrow{r}&\H^1(X,\Omega^1_X).
        \end{tikzcd}
    \]
    We conclude that $\H^0(X,\Omega^1_X/B\Omega^1_X)=(F^1_H\cap F^2_C)\cap (F^2_C\cap F^1_H)=F^2_H\cap F^2_C$. 
\end{proof}



\section{Formal deformation spaces for cohomology classes on K3 surfaces}\label{sec:formal deformations}

Let $k$ be an algebraically closed field of positive characteristic $p$ and let $W=W(k)$ be the ring of Witt vectors of $k$. Let $\cC_W$ be the category of artinian local $W$-algebras with residue field identified with $k$. Let $X$ be a K3 surface over $k$. A \textit{deformation} of $X$ over $A\in\cC_W$ is a pair $(X_A,\rho)$, where $X_A$ is a family of K3 surfaces over $A$ and $\rho:X_A\otimes_Ak\iso X$ is an isomorphism. We let 
\[
    \Def_X:=\Def_{X/W}
\]
be the functor on $\cC_W$ whose value on $A\in\cC_W$ is the set of isomorphism classes of deformations of $X$ over $A$.

Let $n$ be a positive integer, let $\alpha\in\H^2(X,\mu_n)$ be a cohomology class, and let $\alpha_{\Br}\in\H^2(X,\mathbf{G}_m)$ be the image of $\alpha$ in the Brauer group. Let $\ms X$ be a $\mu_n$-gerbe over $X$ with cohomology class $\alpha$, and let $\ms X_{\Br}:=\ms X\wedge_{\mu_n}\mathbf{G}_m$ be the associated $\mathbf{G}_m$-gerbe, which has class $\alpha_{\Br}$. We consider the deformation functors $\Def_{\ms X/W}$ and $\Def_{\ms X_{\Br}/W}$ on $\cC_{W}$ associated to the gerbes $\ms X$ and $\ms X_{\Br}$ (see Definition \ref{def:definition of def functor, most general version}). Up to isomorphism, these depend only on the cohomology classes $\alpha$ and $\alpha_{\Br}$ respectively. In an abuse of notation, we write
\[
    \Def_{(X,\alpha)}:=\Def_{\ms X/W}\hspace{1cm}\mbox{and}\hspace{1cm}\Def_{(X,\alpha_{\Br})}:=\Def_{\ms X_{\Br}/W}
\]
We have a commutative diagram of functors
	\begin{equation}\label{eq:diagram111}
		\begin{tikzcd}[column sep=tiny]
			\Def_{(X,\alpha)}\arrow{rr}{\iota}\arrow{dr}[swap]{\pi}&&\Def_{(X,\alpha_{\Br})}\arrow{dl}{\pi_{\Br}}\\
			&\Def_{X}&
		\end{tikzcd}
	\end{equation}
(see~\eqref{eq:commutative diagram of functors}). The map $\iota$ is induced by $\ms X_A\mapsto\ms X_A\wedge_{\mu_n}\mathbf{G}_m$, the map $\pi$ is induced by $\ms X_A\mapsto |\ms X_A|$, where $|\ms X_A|$ is the sheafification (or ``underlying K3 surface'') of $\ms X_A$, and $\pi_{\Br}$ is induced by $\ms X_{\Br,A}\mapsto |\ms X_{\Br,A}|$.

\begin{remark}\label{rem:naive functor is fine!}
    Let $\oDef_{(X,\alpha)}$ denote the functor on $\cC_W$ whose value on $A$ is the set of isomorphism classes of tuples $(X_A,\rho,\alpha_A)$, where $(X_A,\rho)$ is a deformation of $X$ over $A$ and $\alpha_A\in\H^2(X_A,\mu_n)$ is a class such that $\alpha_A|_X=\alpha$. There is a natural map
    \begin{equation}\label{eq:varpi for mu}
        \Def_{(X,\alpha)}\to\oDef_{(X,\alpha)}
    \end{equation}
    induced by the association $(\ms X_A,\varphi)\mapsto (X_A,\rho,[\ms X_A])$, where $X_A=|\ms X_A|$ is the sheafification of $\ms X_A$, $\rho=|\varphi|$, and $[\ms X_A]\in\H^2(X_A,\mu_n)$ is the cohomology class of $\ms X_A$. Because $\H^1(X,\mu_n)=0$, the map~\eqref{eq:varpi for mu} is an isomorphism (Lemma \ref{lem:hull lemma}), and we may without risk of confusion identify the two functors. In particular, the deformation functors $\Def_{\ms X/W}$ resulting from different choices of $\mu_n$-gerbe $\ms X$ with class $\alpha$ are \textit{canonically} isomorphic. With this identification, the map $\pi$~\eqref{eq:diagram111} is given by $(X_A,\rho,\alpha_A)\mapsto (X_A,\rho)$.

    On the other hand, the analogous map
    \begin{equation}\label{eq:varpi for br}
        \Def_{(X,\alpha_{\Br})}\to\oDef_{(X,\alpha_{\Br})}
    \end{equation}
    is \textit{not} an isomorphism. This is because for a general flat deformation $X_A$ of $X$, the restriction map $\Pic(X_A)\to\Pic(X)$ will not be surjective.
\end{remark}

  

\begin{proposition}\label{prop:pro-representable1}
  The functor $\Def_{(X,\alpha_{\Br})}$ is prorepresentable and formally smooth over $W$, and the map $\pi_{\Br}$~\eqref{eq:diagram111}
is formally smooth of relative dimension 1.
\end{proposition}
\begin{proof}
    We have $\H^1(X,\ms O_X)=0$ and $\H^0(X,T_X)=0$, so Corollary \ref{cor:prorep criterion the third} implies that $\Def_{(X,\alpha_{\Br})}$ is prorepresentable. As $\H^3(X,\mathscr{O}_X)=0$, the obstruction theory of Proposition \ref{prop:obstructions for gerbe deformation functor} implies that $\pi_{\Br}$ is formally smooth. Finally, by Lemma \ref{lem:LES on infinitesimal auts and such}, the map $T(\pi)$ on tangent spaces induced by $\pi$ fits into a short exact sequence
    \[
        0\to\H^2(X,\ms O_X)\to T(\Def_{(X,\alpha_{\Br})})\xrightarrow{T(\pi)} T(\Def_X)\to 0.
    \]
    As $\H^2(X,\ms O_X)$ is one dimensional, we conclude that $\pi_{\Br}$ has relative dimension 1.
\end{proof}

\begin{remark}
  By Lemma \ref{lem:hull lemma}, the map~\eqref{eq:varpi for br} exhibits the functor $\Def_{(X,\alpha_{\Br})}$ as a hull for the naive deformation functor $\oDef_{(X,\alpha_{\Br})}$, which is not prorepresentable (see Remark \ref{rem:not prorepresentable example}).
\end{remark}

\begin{proposition}\label{prop:prorep for mu n}
  The deformation functor $\Def_{(X,\alpha)}$ is prorepresentable.
\end{proposition}
\begin{proof}
    We will verify the conditions of Theorem \ref{thm:prorep criterion the second}. Let $X_A$ be a flat deformation of $X$ over $A\in\cC_W$. We will show that the functor $\Phi^1=\Phi^1_{\mu_n/X_A}$~\eqref{eq:def of Phi} on $\cC_A$ is formally smooth. As $X$ is K3, we have $\H^0(X,\mu_n)=\H^1(X,\mu_n)=0$. It follows that for any $B\in\cC_A$ we have $\Phi^1(B)=\H^1(X_B,\mu_n)$. The Leray spectral sequence gives an exact sequence
    \[
        0\to\H^1(\Spec B,\mu_n)\to\H^1(X_B,\mu_n)\to\H^0(\Spec B,R^1f_{B*}\mu_n)
    \]
    where $f_B:X_B\to\Spec B$ is the structural morphism. We have $R^1f_{B*}\mu_n=\uPic_{X_B/B}[n]=0$. It follows that the left inclusion is an isomorphism, and so $\Phi^1(B)=\H^1(\Spec B,\mu_n)$. Using the Kummer sequence and the vanishing of $\H^1(\Spec B,\mathbf{G}_m)$, we have $\H^1(\Spec B,\mu_n)\cong B^{\times}/B^{\times n}$. (We remark that if $n$ is coprime to $p$, then this quotient is zero, and hence $\Phi^1(B)=0$ for all $B$. This is not the case however if $p$ divides $n$.) Consider a surjection $B'\twoheadrightarrow B$ in $\cC_W$. By the snake lemma, the map $\Phi^1(B')\to\Phi^1(B)$ is isomorphic to the map
    \[
        B^{'\times}/B^{'\times p}\to B^{\times}/B^{\times p}.
    \]
    The map $B^{'\times}\to B^{\times}$ is surjective, so this map is surjective as well. It follows that $\Phi^1$ is formally smooth, as desired.
\end{proof}




\begin{proposition}\label{prop:one equation}
  The map $\iota$~\eqref{eq:diagram111} is a closed formal subscheme defined by one equation.
\end{proposition}
\begin{proof}
    
    To show that $\iota$ is a closed immersion, it suffices to verify that the induced map
  \[
    T(\iota):T(\Def_{(X,\alpha)})\to T(\Def_{(X,\alpha_{\Br})})
  \]
  on tangent spaces is injective. This follows from Proposition \ref{prop:obstructions for relative gerbe change of group situation} and the vanishing $\H^1(X,\mathfrak{n}_{\mu_n/\mathbf{G}_m})=\H^1(X,\ms O_X)=0$. We now show that $\iota$ is defined by one equation.
  This is a formal consequence of Proposition \ref{prop:obstructions for relative gerbe change of group situation}, which gives an obstruction theory for the morphism $\iota$ with values in the one-dimensional $k$-vector space $\H^2(X,\ms O_X)$. For lack of an exact reference, we give the proof.
  Let $(R,\mathfrak{m})$ be a complete local ring prorepresenting $\Def_{(X,\alpha_{\Br})}$. The map $\iota$ corresponds to a surjection $R\to R/J$ for some ideal $J\subset R$. Let $i\geq 1$ be an integer. Consider the square zero extension
  \[
    R/(\mathfrak{m}J+\mathfrak{m}^i)\twoheadrightarrow R/(J+\mathfrak{m}^i)
  \]
  of Artinian $R$-algebras, which has kernel $I_i=(J+\mathfrak{m}^i)/(\mathfrak{m}J+\mathfrak{m}^i)$. As $I_i$ is killed by $\mathfrak{m}$, it has a natural $k$-vector space structure. We will show that $I_i$ has dimension 1 over $k$. Consider the diagram
  \[
    \begin{tikzcd}
        \Spec R/(J+\mathfrak{m}^i)\arrow[hook]{d}\arrow{r}&\Def_{(X,\alpha)}\arrow[hook]{d}\\
        \Spec R/(\mathfrak{m}J+\mathfrak{m}^i)\arrow[hook]{r}\arrow[dashed]{ur}&\Def_{(X,\alpha_{\Br})}.
    \end{tikzcd}  
  \]
  By Proposition \ref{prop:obstructions for relative gerbe change of group situation}, there is a functorial obstruction class $o\in\H^2(X,f^*I_i)=\H^2(X,\ms O_X)\otimes_k I_i$ whose vanishing is equivalent to the existence of the dashed arrow. Let $\tau\in\H^2(X,\ms O_X)$ be a generator. The obstruction class $o$ is then equal to $v\otimes \overline{f}_i$, where $\overline{f}_i\in I_i$ is the image of some element $f_i\in J+\mathfrak{m}^i$. Consider the square zero extension
  \[
    R/(\mathfrak{m}J+(f_i)+\mathfrak{m}^i)\twoheadrightarrow R/(J+\mathfrak{m}^i)
  \]
  which has kernel $(J+\mathfrak{m}^i)/(\mathfrak{m}J+(\overline{f}_i)+\mathfrak{m}^i)=I_i/(\overline{f}_i)$. We consider the diagram
  \[
    \begin{tikzcd}
        \Spec R/(J+\mathfrak{m}^i)\arrow[hook]{d}\arrow[hook]{r}&\Def_{(X,\alpha)}\arrow[hook]{d}\\
        \Spec R/(\mathfrak{m}J+(f_i)+\mathfrak{m}^i)\arrow[hook]{r}\arrow[dashed]{ur}&\Def_{(X,\alpha_{\Br})}.
    \end{tikzcd}  
  \]
  By functoriality, the class $o'\in\H^2(X,\ms O_X)\otimes I_i/(\overline{f}_i)$ which obstructs the existence of the dashed arrow vanishes. It follows that $J\subset\mathfrak{m}J+(f_i)+\mathfrak{m}^i$, and therefore $I_i/(\overline{f}_i)=0$ and the left verical arrow is an isomorphism. In particular, $I_i$ is generated by $\overline{f}_i$, and hence has dimension 1 over $k$.
  
  By the Artin--Rees lemma \cite[00IN]{stacks-project}, we have $\mathfrak{m}^i\cap J\subset\mathfrak{m}J$ for $i$ sufficiently large. We therefore have that
  \[
    I_i=(J+\mathfrak{m}^i)/(\mathfrak{m}J+\mathfrak{m}^i)=J/(\mathfrak{m}J+\mathfrak{m}^i\cap J)=J/\mathfrak{m}J
  \]
  for $i$ sufficiently large. We conclude that $J/\mathfrak{m}J$ has dimension one over $k$. If $f\in J$ is any element whose image in $J/\mathfrak{m}J$ is nonzero, then by Nakayama's lemma we have $J=(f)$.
\end{proof}

We will describe the deformation spaces~\eqref{eq:diagram111} in explicit coordinates. By Proposition \ref{prop:pro-representable1}, the map $\pi_{\Br}$ may be represented in suitable coordinates by the projection
\[
    \Spf W[[t_1,\dots,t_{20},s]]\to\Spf W[[t_1,\dots,t_{20}]].
\]
By Proposition \ref{prop:one equation}, the diagram~\eqref{eq:diagram111} may then be represented by
	\begin{equation}\label{eq:diagram1111}
	\begin{tikzcd}[column sep=-1.2cm]
			\Spf W[[t_1,\dots,t_{20},s]]/(g)\arrow[hook]{rr}{\iota}\arrow{dr}[swap]{\pi}&&\Spf W[[t_1,\dots,t_{20},s]]\arrow{dl}{\pi_{\Br}}\\
			&\Spf W[[t_1,\dots,t_{20}]]&
		\end{tikzcd}
	\end{equation}
for some function $g\in W[[t_1,\dots,t_{20},s]]$.

\begin{proposition}\label{prop:nice coordinates}
	We may choose $g$ so that $g$ is congruent modulo $(p,t_1,\dots,t_{20})$ to either $s^{p^{rh}}$ (if $X$ has finite height $h$) or to $0$ (if $X$ is supersingular), where $p^r$ is the largest power of $p$ dividing $n$. If $p$ does not divide $n$, then we may even take $g=s$. 
\end{proposition}
\begin{proof}
    Restricting~\eqref{eq:diagram111} to the closed point $0\in\Def_X$ we get a map
	\begin{equation}\label{eq:diagram1121}
		\iota_0:\Def_{(X,\alpha)}|_0\hookrightarrow\Def_{(X,\alpha_{\Br})}|_0
	\end{equation}
	which is represented by
	\begin{equation}\label{eq:diagram1131}
	    \Spf k[[s]]/(\overline{g})\hookrightarrow\Spf k[[s]]
	\end{equation}
	where $\overline{g}$ is the image of $g$ modulo $(p,t_1,\dots,t_{20})$. Write $\widehat{\H}^2(X,\mu_n)$ for the functor
	\[
		A\mapsto\ker(\H^2(X_A,\mu_n)\to\H^2(X,\mu_n))  
	\]
	on $\cC_k$. For any $A\in\cC_k$, the inclusion $i_A:X\hookrightarrow X_A$ is split. It follows that we have $\widehat{\Br}_X=\Def_{(X,0_{\Br})}|_0$ and $\widehat{\H}^2(X,\mu_n)=\Def_{(X,0)}|_0$. There is an isomorphism
	\[
	    \Def_{(X,\alpha)}|_0\iso\widehat{\H}^2(X,\mu_n)
	\]
	defined by $\alpha_A\mapsto \alpha_A-\rho_A^*(\alpha)$, where $\rho_A:X_A\to X$ is the projection. We similarly define an isomorphism $\Def_{(X,\alpha_{\Br})}|_0\iso\widehat{\Br}_X$. Thus, we may assume without loss of generality that both $\alpha$ and $\alpha_{\Br}$ are zero, in which case
    the map~\eqref{eq:diagram1121} is identified with the natural map
	\begin{equation}\label{eq:formal Kummer map}
	  \widehat{\H}^2(X,\mu_n)\to\widehat{\Br}_{X}
	\end{equation}
	induced by the inclusion $\mu_n\subset\mathbf{G}_m$. Consider the commutative diagram
	\[
	    \begin{tikzcd}
	        \widehat{\H}^2(X,\mu_{p^r})\isor{d}{}\arrow{r}{\sim}&\widehat{\Br}_X[p^r]\arrow[equals]{d}\\
	        \widehat{\H}^2(X,\mu_{n})\arrow{r}{\sim}&\widehat{\Br}_X[n].
	    \end{tikzcd}
	\]
	Here, the horizontal arrows are induced by the natural maps~\eqref{eq:formal Kummer map}, which are injective because the Picard scheme of $X$ is discrete. The left vertical arrow is induced by the inclusion $\mu_{p^r}\subset\mu_n$, and is an isomorphism because $\widehat{\H}^m(X,\mu_{n/p^r})=0$ for all $m$ as $n/p^r$ is coprime to $p$.
	
	We conclude that the inclusion~\eqref{eq:diagram1121} is isomorphic to the inclusion $\widehat{\Br}_X[p^r]\subset\widehat{\Br}_X$.
	If $X$ has finite height $h$, then this map is represented by the closed immersion
	\[
        \Spf k[[s]]/(s^{p^{hr}})\hookrightarrow\Spf k[[s]].
	\]
	If $X$ is supersingular, then $\widehat{\Br}_X=\widehat{\mathbf{G}}_a$, so $\widehat{\Br}_X[p^r]=\widehat{\Br}_X$ as long as $r\geq 1$. Finally, suppose moreover that $n$ is coprime to $p$. In this case, $\pi$ is an isomorphism, and so $g$ is equal to a unit times $s$. Hence, we may take $g=s$. 
\end{proof}

\begin{remark}
  Proposition \ref{prop:nice coordinates} implies that the forgetful map $\Def_{(X,\alpha)}\to\Def_X$ will frequently be non-flat. Nevertheless, we will see that $\Def_{(X,\alpha)}$ itself is always flat over $W$.
\end{remark}

	

We consider the maps on tangent spaces induced by~\eqref{eq:diagram111}. We have a canonical identification $T(\Def_X)=\H^1(X,T_X)$. Comparing the exact sequences of Lemma \ref{lem:LES on infinitesimal auts and such} and using the vanishing $\H^0(X,T_X)=\H^1(X,\ms O_X)=\H^3(X,\ms O_X)=0$, we obtain a diagram
\begin{equation}\label{eq:tangent space diagram, first version}
    \begin{tikzcd}
        0\arrow{r}&\H^2(X,\ms O_X(n))\arrow[hook]{d}\arrow{r}&T(\Def_{(X,\alpha)})\arrow[hook]{d}[swap]{T(\iota)}\arrow{r}{T(\pi)}&\H^1(X,T_X)\arrow[equals]{d}\arrow{r}{\varepsilon^{-1}\mathrm{ob}}&\H^3(X,\ms O_X(n))\arrow{d}\\
        0\arrow{r}&\H^2(X,\ms O_X)\arrow{r}&T(\Def_{(X,\alpha_{\Br})})\arrow{r}{T(\pi_{\Br})}&\H^1(X,T_X)\arrow{r}&0
    \end{tikzcd}
\end{equation}
with exact rows, where $\mathrm{ob}$ is the map which sends $\tau_{X'}$ to the obstruction class $o(\alpha/X')$.

We can be more explicit. If $n$ is coprime to $p$, then the complex $\ms O_X(n)$ is quasi-isomorphic to 0. The diagram~\eqref{eq:tangent space diagram, first version} becomes
\begin{equation}\label{eq:summary1}
	\begin{tikzcd}
		&&T(\Def_{(X,\alpha)})\arrow{r}{T(\pi)}[swap]{\sim}\arrow[hook]{d}&\H^1(X,T_X)\arrow[equal]{d}&\\
		0\arrow{r}&\H^2(X,\mathscr{O}_X)\arrow{r}&T(\Def_{(X,\alpha_{\Br})})\arrow{r}&\H^1(X,T_X)\arrow{r}&0.
	\end{tikzcd}
\end{equation}
In particular, $T(\Def_{(X,\alpha)})$ has dimension $20$.

On the other hand, suppose that $p$ divides $n$. We then have $\ms O_X(n)=\ms O_X\oplus\ms O_X[-1]$. As $\H^1(X,\ms O_X)=0$, the left vertical arrow of~\eqref{eq:tangent space diagram, first version} is an isomorphism. Given $v\in\H^1(X,\Omega^1_X)$ write $\Ann(v)\subset\H^1(X,T_X)$ for the subspace of elements $\tau$ such that $v\cup\tau=0$. By Proposition \ref{prop:cup product}, we have a commutative diagram
\begin{equation}\label{eq:triangle with ob}
    \begin{tikzcd}
        &\H^3(X,\ms O_X(n))\arrow{dr}{\pi_1}&\\
        \H^1(X,T_X)\arrow{ur}{\varepsilon^{-1}\mathrm{ob}}\arrow{rr}{\dlog(\alpha)\cup\_}&&\H^2(X,\ms O_X).
    \end{tikzcd}
\end{equation}
The map $\pi_1$ is an isomorphism, and it follows that the kernel of $\varepsilon^{-1}\mathrm{ob}$ is equal to $\Ann(\dlog(\alpha))$. The diagram~\eqref{eq:tangent space diagram, first version} becomes
\begin{equation}\label{eq:summary2}
  \begin{tikzcd}
    0\arrow{r}&\H^2(X,\mathscr{O}_X)\arrow{r}\arrow[equal]{d}&T(\Def_{(X,\alpha)})\arrow{r}\arrow[hook]{d}&\Ann(\dlog \alpha)\arrow{r}\arrow[hook]{d}&0\\
    0\arrow{r}&\H^2(X,\mathscr{O}_X)\arrow{r}&T(\Def_{(X,\alpha_{\Br})})\arrow{r}&\H^1(X,T_X)\arrow{r}&0.
  \end{tikzcd}
\end{equation}
Because $X$ is K3, the cup product pairing~\eqref{eq:cup product} is perfect, and therefore the map
\begin{equation}\label{eq:canonical KS maps}
    \H^1(X,\Omega^1_X)\iso\Hom(\H^1(X,T_X),\H^2(X,\mathscr O_X))
\end{equation}
is an isomorphism. Thus, the above diagram shows in particular that $T(\Def_{(X,\alpha)})$ has dimension $20$ if $\dlog(\alpha)\neq 0$ and has dimension $21$ otherwise.

\begin{remark}
  Note that when $p$ divides $n$, the group $\H^2(X,\ms O_X)$ plays two distinct roles: it appears as both the relative tangent space to $\pi$ (in the top row of~\eqref{eq:summary2}), and as the obstruction group for the morphism $\iota$ (in~\eqref{eq:triangle with ob}).
\end{remark}

We deduce some consequences for universal deformation spaces.
\begin{proposition}\label{prop:smooth1}
	Consider a class $\alpha\in\H^2(X,\mu_n)$. If $n$ is coprime to $p$, then $\Def_{(X,\alpha)}$ is formally smooth over $W$. If $p$ divides $n$ and $\dlog(\alpha)\neq 0$, then $\Def_{(X,\alpha)}$ is formally smooth over $W$.
\end{proposition}
\begin{proof}
	Consider the Jacobian ideal
	\[
	    J:=\left(\dfrac{\partial g}{\partial t_1},\dots,\dfrac{\partial g}{\partial t_{20}},\dfrac{\partial g}{\partial s}\right)\subset W[[t_1,\dots,t_{20},s]]
	\]
	of the formal subscheme $\Def_{(X,\alpha)}\subset\Def_{(X,\alpha_{\Br})}$. Under the given conditions, the tangent space to $\Def_{(X,\alpha)}\otimes k$ at the closed point has dimension 20. It follows that $J/pJ$ is the unit ideal. By Nakayama's lemma, $J$ is the unit ideal, and hence $\Def_{(X,\alpha)}$ is formally smooth over $W$.
\end{proof}

We now incorporate a line bundle on $X$.

\begin{notation}
Let $L$ be a line bundle on $X$. We let $\Def_{(X,L)}$ denote the functor on $\cC_W$ sending $A$ to the set of isomorphism classes of tuples $(X_A,\rho,L_A)$, where $(X_A,\rho)$ is a deformation of $X$ over $A$ and $L_A$ is a line bundle on $X_A$ whose restriction to $X$ is isomorphic to $L$. We put 
\[
    \Def_{(X,\alpha,L)}:=\Def_{(X,\alpha)}\times_{\Def_X}\Def_{(X,L)}.
\]
By Remark \ref{rem:naive functor is fine!}, we may equivalently define $\Def_{(X,\alpha,L)}$ to be the functor sending $A\in\cC_W$ to the set of isomorphism classes of tuples $(X_A,\rho,\alpha_A,L_A)$, where $(X_A,\rho)$ is a deformation of $X$ over $A$, $\alpha_A\in\H^2(X_A,\mu_n)$ is a class such that $\alpha_A|_X=\alpha$, and $L_A$ is a line bundle such that $L_A|_X$ is isomorphic to $L$.


Given a collection of line bundles $L_1,\dots,L_m$, we similarly define $\Def_{(X,L_1,\dots,L_m)}$ and $\Def_{(X,\alpha,L_1,\dots,L_m)}$.
\end{notation}


We write
\begin{align}\label{eq:maps32}
	\Pic(X)=\H^1(X,\mathbf{G}_m)\xrightarrow{c_1}\H^1(X,\Omega^1_X)
\end{align}
for the map induced by $\dlog$~\eqref{eq:map11}.
\begin{proposition}\label{prop:smooth2}
	Consider a class $\alpha\in\H^2(X,\mu_n)$ and a line bundle $L$ on $X$. If $n$ is coprime to $p$ and $c_1(L)$ is nonzero, then $\Def_{(X,\alpha,L)}$ is formally smooth over $W$. If $p$ divides $n$ and $c_1(L)$ and $\dlog(\alpha)$ are linearly independent in $\H^1(X,\Omega^1_X)$, then $\Def_{(X,\alpha,L)}$ is formally smooth over $W$.
\end{proposition}
\begin{proof}
	By Proposition \ref{prop:one equation}, the inclusion $\Def_{(X,\alpha,L)}\subset\Def_{(X,\alpha_{\Br})}$ is a closed formal subscheme defined by two equations. Under the assumed conditions, the tangent space to $\Def_{(X,\alpha,L)}\otimes k$ at the closed point has dimension 19. As in Proposition \ref{prop:smooth1}, we conclude that $\Def_{(X,\alpha,L)}$ is formally smooth over $W$.
\end{proof}

\begin{remark}
  If $n$ is coprime to $p$, Proposition \ref{prop:smooth2} follows from a result of Ogus \cite[Proposition 2.2]{Ogus78} and Proposition \ref{prop:smooth1}. 
\end{remark}

\section{The dlog map and de Rham cohomology}\label{sec:dlog map}

Motivated by Proposition \ref{prop:smooth2}, we seek conditions under which the classes $\dlog(\alpha)$ and $c_1(L)$ are linearly independent. We will study the interaction between the images of the various $\dlog$ maps in de Rham cohomology. We consider the $\dlog$ map~\eqref{eq:map11}
\[
    \dlog:\mathbf{G}_m\to\Omega^1_X.
\]
As the target is $p$-torsion, $\dlog$ kills the subsheaf $\mathbf{G}_m^{\times p}$ of $p$th powers. Furthermore, the image of $\dlog$ is contained in the subsheaf $Z\Omega^1_X\subset\Omega^1_X$. To distinguish between the resulting maps on cohomology, we will use the notation in the following commutative diagram.
\begin{equation}\label{eq:all the dlog maps}
    \begin{tikzcd}
        \Pic(X)\arrow{dd}\arrow{dr}[swap]{c_1^{\dR}}\arrow[bend left=15]{drr}{c_1}\\
        &\H^1(X,Z\Omega^1_X)\arrow{r}&\H^1(X,\Omega^1_X).\\
        \H^1(X,\mathbf{G}_m/\mathbf{G}_m^{\times p})\arrow{ur}[swap]{\dlog^{\dR}}\arrow[bend right=15]{urr}[swap]{\dlog}
    \end{tikzcd}
\end{equation}
Here, the vertical map is induced by the quotient and the horizontal map is induced by the inclusion.
We will also use $c_1^{\dR}$ and $\dlog^{\dR}$ to denote the respective compositions of these maps with the inclusion $\H^1(X,Z\Omega^1_X)\subset\H^2_{\dR}(X)$. Let
\[
    C:Z\Omega^1_X\to\Omega^1_X
\]
denote the Cartier operator, which satisfies $C(f^p\omega)=fC(\omega)$ and $C(f^{p-1}df)=df$ for any local sections $f\in\ms O_X$ and $\omega\in Z\Omega^1_X$. As a consequence, if $f$ is invertible, then $C(df/f)=df/f$. By \cite[Corollaire 0.2.1.18]{Ill79}, we have a short exact sequence
\begin{equation}\label{eq:diagram31}
	1\to\mathbf{G}_m/\mathbf{G}_m^{\times p}\xrightarrow{\dlog}Z\Omega^1_X\xrightarrow{1-C}\Omega^1_X\to 0
\end{equation}
where $1$ denotes the inclusion. Taking cohomology, we find an exact sequence
\begin{equation}\label{eq:diagram33}
	0\to\H^1(X,\mathbf{G}_m/\mathbf{G}_m^{\times p})\xrightarrow{\dlog^{\dR}} \H^1(X,Z\Omega^1_X)\xrightarrow{1-C}\H^1(X,\Omega^1_X)
\end{equation}
where the injectivity on the left follows from the vanishing of $\H^0(X,\Omega^1_X)$. Under the identifications of Lemma \ref{lem:some natural identifications}, the sequence~\eqref{eq:diagram33} becomes
\[
  0\to\H^1(X,\mathbf{G}_m/\mathbf{G}_m^{\times p})\xrightarrow{\dlog^{\dR}} F^1_H\cap F^1_C\xrightarrow{1-C}F^1_H/F^2_H
\]
where the right hand map is given by the difference of the map
\[
  1:F^1_H\cap F^1_C\subset F^1_H\twoheadrightarrow F^1_H/F^2_H,
\]
and the Cartier operator $C$, which factors as the composition
\[
  F^1_H\cap F^1_C\subset F^1_C\twoheadrightarrow F^1_C/F^2_C\xrightarrow{\sim} F^1_H/F^2_H.
\]
Thus, the kernel of the map $1$ is $F^2_H\cap F^1_C$, and the kernel of $C$ is $F^1_H\cap F^2_C$. 

We are interested in the injectivity of the various maps induced by $\dlog$~\eqref{eq:all the dlog maps}. Note that the maps $c_1^{\dR}$ and $c_1$~\eqref{eq:all the dlog maps} have $p$-torsion codomain, and hence kill $p\Pic(X)$. They therefore descend to maps on $\Pic(X)/p=\Pic(X)\otimes_{\mathbf{Z}}\mathbf{F}_p$. The following result is due to Ogus \cite[Corollary 1.3, Proposition 1.4]{Ogus78}. We include the proof.

\begin{proposition}\label{prop:dlog is nonzero sometimes}
    If $X$ is any K3 surface, then the maps
    \[
        \Pic(X)\otimes_{\mathbf{Z}}\mathbf{F}_p\xrightarrow{c_1^{\dR}\otimes\mathbf{F}_p}\H^1(X,Z\Omega^1_X)\hspace{1cm}\text{and}\hspace{1cm}\H^1(X,\mathbf{G}_m/\mathbf{G}_m^{\times p})\xrightarrow{\dlog^{\dR}}\H^1(X,Z\Omega^1_X),
    \]
    are injective. If $X$ is not superspecial, then also the maps
    \[
        \Pic(X)\otimes_{\mathbf{Z}}\mathbf{F}_p\xrightarrow{c_1\otimes\mathbf{F}_p}\H^1(X,\Omega^1_X)\hspace{1cm}\text{and}\hspace{1cm}\H^1(X,\mathbf{G}_m/\mathbf{G}_m^{\times p})\xrightarrow{\dlog}\H^1(X,\Omega^1_X),
    \]
    are injective.
\end{proposition}
\begin{proof}
    Consider the short exact sequence
    \[
        1\to\mathbf{G}_m\xrightarrow{\cdot p}\mathbf{G}_m\to\mathbf{G}_m/\mathbf{G}_m^{\times p}\to 1.
    \]
    Taking cohomology, we deduce that the map
    \[
        \Pic(X)\otimes_{\mathbf{Z}}\mathbf{F}_p\to\H^1(X,\mathbf{G}_m/\mathbf{G}_m^{\times p})
    \]
    is injective. By the exactness of~\eqref{eq:diagram33}, $\dlog^{\dR}$ is injective. This proves the first two claims. For the second two, suppose that $\dlog$ is not injective. We then have a nonzero element $\sigma\in F^2_H$ which is killed by $1-C$. It follows that $\sigma\in F^2_C$, and therefore $F^2_H=F^2_C$. We conclude that $X$ is superspecial.
\end{proof}

We now consider the maps $c_1^{\dR}\otimes k$ and $c_1\otimes k$ obtained by tensoring with $k$. We record the following result.

\begin{proposition}\label{prop:inj for c 1 tensor k}
  If $X$ has finite height, then the maps
  \[
        \Pic(X)\otimes_{\mathbf{Z}}k\xrightarrow{c_1^{\dR}\otimes k}\H^1(X,Z\Omega^1_X)\hspace{1cm}\text{and}\hspace{1cm}\Pic(X)\otimes_{\mathbf{Z}}k\xrightarrow{c_1\otimes k}\H^1(X,\Omega^1_X)
    \]
    are injective.
\end{proposition}
\begin{proof}
    This follows from the Newton--Hodge decomposition on the second crystalline cohomology of $X$. See \cite[Remark 1.9]{Ogus78}. A different proof is given in \cite[Proposition 10.3]{MR1776939}. 
\end{proof}

We strengthen this slightly in the following.

\begin{proposition}\label{prop:trivial intersection finite height case}
  If $X$ has finite height, then the map
  \begin{equation}\label{eq:c_1 times k}
    c_1^{\dR}\otimes k:\Pic(X)\otimes_{\mathbf{Z}} k\to\H^2_{\dR}(X)
  \end{equation}
  is injective, and its image has trivial intersection with the subspace $F^2_H+F^2_C\subset \H^2_{\dR}(X)$.
\end{proposition}
\begin{proof}
    Consider the commutative diagram
  \[
    \begin{tikzcd}
    &&\Pic(X)\otimes_{\mathbf{Z}} k\arrow{d}[swap]{c_1^{\dR}\otimes k}\arrow{dr}{c_1\otimes k}&\\
    0\arrow{r}&\H^0(X,B\Omega^2_X)\arrow{r}{a}&\H^1(X,Z\Omega^1_X)\arrow{r}&\H^1(X,\Omega^1_X)\\
    &&\H^1(X,B\Omega^1_X)\arrow[hook]{u}{b}\arrow{ur}[swap]{c}
    \end{tikzcd}
  \]
  where the row is exact and $b$ and $c$ are induced by the natural inclusions of sheaves. Under the identifications of Lemma \ref{lem:some natural identifications}, we have $\H^1(X,Z\Omega^1_X)=F^1_H\cap F^1_C$, the image of $a$ is $F^2_H\cap F^1_C$, and the image of $b$ is $F^1_H\cap F^2_C$.
  
  We now use the assumption that $h<\infty$. By \ref{prop:inj for c 1 tensor k}, $c_1^{\dR}\otimes k$ and $c_1\otimes k$ are injective. Furthermore, by \cite[Proposition 10.2]{MR1776939}, the image of $c_1\otimes k$ has trivial intersection with the image of $c$. We conclude that the image of $c_1^{\dR}\otimes k$ has trivial intersection with $F^2_H+F^2_C$.
\end{proof}

\begin{remark}
  If $X$ has finite height, then one can strengthen \ref{prop:inj for c 1 tensor k} to show that the maps $\dlog^{\dR}\otimes k$ and $\dlog\otimes k$ are injective. It is also true that the image of $\dlog^{\dR}\otimes k$ has trivial intersection with $F^2_H+F^2_C$.
\end{remark}

\subsection{The supersingular case}
The preceding results \ref{prop:inj for c 1 tensor k}, \ref{prop:trivial intersection finite height case} are false if $X$ is supersingular. In fact, in this case the map $c_1^{\dR}\otimes k$ is never injective, and furthermore the subspace $F^2_H+F^2_C$ is even contained in the image of $c_1^{\dR}\otimes k$. To explain this situation, we recall some results of Ogus \cite{Ogus78}.

Suppose that $X$ is supersingular with Artin invariant $\sigma_0$. By the Tate conjecture for supersingular K3 surfaces, $\Pic(X)$ is a $\mathbf{Z}$-lattice of rank 22.\footnote{If we wish to avoid the use of the Tate conjecture, we may replace $\Pic(X)$ with $\H^2(X,\mathbf{Z}_p(1))$ and replace $c_1^{\dR}\otimes k$ with the natural map $\H^2(X,\mathbf{Z}_p(1))\otimes k\to\H^2_{\dR}(X)$.} Write $\varphi:\Pic(X)\otimes k\to\Pic(X)\otimes k$ for the bijective map given by $v\otimes\lambda\mapsto v\otimes\lambda^p$. The map $c_1^{\dR}\otimes k$ factors through $F^1_H\cap F^1_C$, and thus for dimension reasons cannot be injective. Its kernel is equal to $\varphi(K)$ for some subspace $K\subset\Pic(X)\otimes k$. We have an exact sequence
\[
    0\to\varphi(K)\to\Pic(X)\otimes k\xrightarrow{c_1^{\dR}\otimes k}\H^2_{\dR}(X/k).
\]
The subspace $K$ is the \textit{characteristic subspace} associated to $X$, and plays a central role in the theory of supersingular K3 surfaces. The following result is due to Ogus.
\begin{lemma}\label{lem:char sub properties}
    The subspace $K\subset\Pic(X)\otimes k$ has the following properties.
  \begin{enumerate}
      \item[{\rm (1)}] $\dim_k K=\sigma_0$
      \item[{\rm (2)}] $\dim_k \,K+\varphi(K)=\sigma_0+1$
      \item[{\rm (3)}] $\dim_k \sum_{i\geq 0} \varphi^i(K)=2\sigma_0$
  \end{enumerate}
\end{lemma}
\begin{proof}
    This follows from \cite[Proposition 3.12.2, 3.12.3]{Ogus78}. We note that, while loc. cit. has a standing assumption that $p\neq 2$, this is not used in the proof of the cited result.
\end{proof}

\begin{remark}
  We mention two other approaches to the characteristic subspace $K$, complementing Ogus's crystalline methods. Nygaard \cite{MR596878} has given an interpretation for $K$ using de Rham-Witt cohomology. Katsura-Van der Geer give an elementary proof of the above properties for $K$ \cite[\S 11]{MR1776939} (in the notation of loc. cit., the subspace $U_i\subset\Pic(X)\otimes k$ is equal to $\varphi(K)\cap\dots\cap\varphi^i(K)$ if $i\geq 1$, and $U_0=K+\varphi(K)$).
\end{remark}

\begin{lemma}\label{lem:char subspace lemma}
    For each $0\leq i\leq\sigma_0$, we have
    \[
        \dim_k(K+\varphi(K)+\dots+\varphi^i(K))=\sigma_0+i
    \]
    and
    \[
        \dim_k(K\cap\varphi(K)\cap\dots\cap\varphi^{i}(K))=\sigma_0-i.
    \]
\end{lemma}
\begin{proof}
    We prove the first claim. Write $Z_i=K+\varphi(K)+\dots+\varphi^i(K)$. We induct on $i$. The case $i=0$ is true by assumption. For the induction step, consider the quotient $Z_{i+1}/Z_{i}$. We will show that if $i<\sigma_0$ then $\dim_k(Z_{i+1}/Z_i)=1$. We have $\dim(\varphi^{i}(K)+\varphi^{i+1}(K))=\sigma_0+1$, so $\dim_k(Z_{i+1}/Z_{i})$ is either 0 or 1. In the former case, we have $Z_{i}=Z_{i+1}$, so $\varphi^{i+1}(K)\subset Z_{i}$, and therefore $V\otimes k=\sum_{j\geq 0}\varphi^j(K)=Z_{i}$. By induction, $Z_i$ has dimension $\sigma_0+i$, so $i=\sigma_0$. Thus, if $i<\sigma_0$, we have $\dim_k(Z_{i+1}/Z_i)=1$. The second claim is similar.
\end{proof}

The subspace $\sum_i\varphi^i(K)$ is fixed by $\varphi$, and hence is equal to $M\otimes k$ for some $\mathbf{F}_p$-subspace $M\subset\Pic(X)\otimes\mathbf{F}_p$. In making computations, it is helpful to choose a basis of $M\otimes k$ which is adapted to $K$. By Lemma \ref{lem:char subspace lemma}, the subspace $\varphi^{-\sigma_0+1}(K)\cap\dots\cap K$ has dimension 1. Let $e$ be a generator, and set $e_i=\varphi^i(e)$. It follows that, for each $0\leq b\leq \sigma_0$, the vectors 
\[
    \left\{e_0,\dots,e_{\sigma_0+b-1}\right\}
\]
are linearly independent, and form a basis for $K+\varphi(K)+\dots+\varphi^b(K)$. In particular, $\left\{e_0,e_1,\dots,e_{\sigma_0-1}\right\}$ is a basis for $K$, and $\left\{e_0,e_1,\dots,e_{2\sigma_0-1}\right\}$ is a basis for $M\otimes k$.
We refer to such a vector $e$ as a \textit{characteristic vector} for $K$. This construction is due to Ogus; see \cite[pg. 33]{Ogus78}.

We define a sequence of subspaces
\[
    0= V_0\subset V_1\subset V_2\subset\dots\subset\H^2_{\dR}(X)
\]
by setting $V_0=0$ and
\[
    V_i:=\text{Im}(K+\varphi(K)+\dots+\varphi^i(K)\xrightarrow{c_1^{\dR}\otimes k}\H^2_{\dR}(X/k))
\]
for $i\geq 1$. Thus, $c_1^{\dR}\otimes k$ induces an isomorphism
\[
    (K+\varphi(K)+\dots+\varphi^i(K))/\varphi(K)\iso V_i.
\]
By Lemma \ref{lem:char subspace lemma}, we have $\dim_kV_i=i$ for $1\leq i\leq \sigma_0$, and $V_{\sigma_0}=V_{\sigma_0+j}$ for all $j\geq 0$. 

The following result gives a cohomological interpretation for $V_1$ and $V_2$. We will use the commuting squares
\[  	
    \begin{tikzcd}
  		\Pic(X)\otimes k\arrow{r}{\id}\arrow{d}[swap]{c_1^{\dR}\otimes k}&\Pic(X)\otimes k\arrow{d}{c_1\otimes k}\\
 		 \H^1(X,Z\Omega^1_X)\arrow{r}{1}&\H^1(X,\Omega^1_X)
  	\end{tikzcd}
  	\hspace{1cm}
  	\begin{tikzcd}
  	  \Pic(X)\otimes k\arrow{r}{\varphi^{-1}}\arrow{d}[swap]{c_1^{\dR}\otimes k}&\Pic(X)\otimes k\arrow{d}{c_1\otimes k}\\
  	  \H^1(X,Z\Omega^1_X)\arrow{r}{C}&\H^1(X,\Omega^1_X).
  	\end{tikzcd}
\]

\begin{lemma}\label{lem:cohomological interpretation}
    We have $V_1=F^2_H$ and $V_2=F^2_H+F^2_C$.
\end{lemma}
\begin{proof}
    Let $e$ be an element of $\varphi(K)$ which is not in $\varphi^2(K)$. We then have $c_1^{\dR}\otimes k(e)=0$. Using the above commuting square involving $C$, we have $c_1\otimes k(\varphi^{-1}(e))=0$. Because $\varphi^{-1}(e)\notin\varphi(K)$, $c_1^{\dR}\otimes k(\varphi^{-1}(e))$ is a nonzero element of the kernel of the projection $\H^1(X,Z\Omega^1_X)\to\H^1(X,\Omega^1_X)$, which is $F^2_H$. It follows that $V_1=F^2_H$. Similarly, let $f$ be an element of $\varphi^2(K)$ which is not in $\varphi(K)$. We have that $c_1^{\dR}\otimes k(f)$ is a nonzero element of the kernel of $C:\H^1(X,Z\Omega^1_X)\to\H^1(X,\Omega^1_X)$, which is $F^2_C$. We conclude that $c_1\otimes k$ induces an isomorphism
    \[
        (\varphi(K)+\varphi^2(K))/\varphi(K)\iso F^2_C.
    \]
    Thus, we have $V_2=F^2_H+F^2_C$.
\end{proof}



We have the following result.

\begin{proposition}\label{prop:trivial intersection supersingular case}
  Let $X$ be a supersingular K3 surface with $\sigma_0(X)\geq 3$.
    Let $L$ be a line bundle on $X$. If $c_1^{\dR}(L)$ is contained in $F^2_H+F^2_C$, then $L$ is a $p$th power.
\end{proposition}
\begin{proof}
    If $c_1^{\dR}(L)$ is in $F^2_H+F^2_C$, then using Lemma \ref{lem:cohomological interpretation}, we deduce that $[L]\in K+\varphi(K)+\varphi^2(K)$. Let $e$ be a characteristic vector for $K$.
Write $[L]$ as a linear combination of $e_0,\dots,e_{\sigma_0+1}$. Applying $\varphi$ to both sides, we find a linear relation between the vectors $e_0,\dots,e_{\sigma_0+2}$. As $\sigma_0\geq 3$, this relation must be trivial, which implies $[L]=0$ as an element of $\Pic(X)\otimes k$.
\end{proof}


We now incorporate a flat cohomology class. Let $n$ be a positive integer. 
In~\eqref{eq:map15} we defined a map
\begin{equation}\label{eq:map15, yet again}
    \dlog:\H^2(X,\mu_n)\to\H^1(X,\Omega^1_X).
\end{equation}
If $n$ is coprime to $p$, this map is zero. If $n$ is divisible by $p$, then it fits into the commuting diagram
\begin{equation}\label{eq:a square for some flat coho dlogs}
    \begin{tikzcd}
        \H^2(X,\mu_n)\arrow{d}{\cdot n/p}\arrow{r}{\Upsilon}\arrow[bend left=25]{rr}{\dlog}&\H^1(X,\mathbf{G}_m/\mathbf{G}_m^{\times n})\arrow{d}\arrow{r}{\dlog}&\H^1(X,\Omega^1_X)\\
        \H^2(X,\mu_p)\arrow{r}{\sim}&\H^1(X,\mathbf{G}_m/\mathbf{G}_m^{\times p})\arrow{ur}[swap]{\dlog}
    \end{tikzcd}
\end{equation}
where the right vertical arrow is induced by the natural quotient map.



\begin{proposition}\label{prop:linearly indep}
  Suppose that $X$ has finite height, or that $X$ is supersingular with Artin invariant $\sigma_0\geq 3$. Let $n$ be a positive integer which is divisible by $p$. If $\alpha\in\H^2(X,\mu_n)$ is a class such that $p$ does not divide $\frac{n}{\ord(\alpha_{\Br})}$
  and $L$ is a line bundle on $X$ that is not a $p$th power, then $\dlog(\alpha)$ and $c_1(L)$ are nonzero and linearly independent in $\H^1(X,\Omega^1_X)$.
\end{proposition}
\begin{proof}
    Here, as usual, $\alpha_{\Br}$ denotes the image of $\alpha$ in $\Br(X)=\H^2(X,\mathbf{G}_m)$, and $\ord(\alpha_{\Br})$ is the order of $\alpha_{\Br}$. Suppose that $\dlog(\alpha)=\lambda c_1(L)\in\H^1(X,\Omega^1_X)$ for some scalar $\lambda$. We have
  \[
    \dlog^{\dR}(\alpha)=\lambda c_1^{\dR}(L)+\sigma
  \]
  as elements of $\H^1(X,Z\Omega^1_X)=F^1_H\cap F^1_C\subset\H^2_{\dR}(X)$, for some $\sigma\in F^2_H\cap F^1_C$ (here, $\dlog^{\dR}$ is the evident lift of~\eqref{eq:map15, yet again} to a map with target $\H^1(X,Z\Omega^1_X)$). We have a commuting diagram
  \[
    \begin{tikzcd}
      0\arrow{r}&\Pic(X)\otimes\mathbf{F}_p\arrow{r}\arrow[hook]{d}&\Pic(X)\otimes k\arrow{r}{\id\otimes(1-\varphi^{-1})}\arrow{d}{c_1^{\dR}\otimes k}&\Pic(X)\otimes k\arrow{d}{c_1\otimes k}\arrow{r}&0\\
      0\arrow{r}&\H^1(X,\mathbf{G}_m/\mathbf{G}_m^{\times p})\arrow{r}{\dlog^{\dR}}&\H^1(X,Z\Omega^1_X)\arrow{r}{1-C}&\H^1(X,\Omega^1_X)&
    \end{tikzcd}
  \]
  with exact rows. By the commutativity of the right hand square, we deduce that $(\lambda-\lambda^{1/p})c_1(L)=C(\sigma)$, and thus $C((\lambda^p-\lambda)c^{\dR}_1(L)-\sigma)=0$. The kernel of $C$ is $F^1_H\cap F^2_C$, so this implies that 
  \[
  (\lambda^p-\lambda)c^{\dR}_1(L)\in F^2_H+F^2_C.
  \]
  Applying \ref{prop:trivial intersection finite height case} or \ref{prop:trivial intersection supersingular case}, we have $(\lambda^p-\lambda)c^{\dR}_1(L)=0$. As $L$ is not a $p$th power, \ref{prop:dlog is nonzero sometimes} implies that $c_1^{\dR}(L)$ is nonzero. Hence, $\lambda^p-\lambda=0$, and therefore $\lambda\in\mathbf{F}_p$. Consider the commutative diagram
  \begin{equation}\label{eq:the diagram that makes the proof work}
  	\begin{tikzcd}
		\Pic(X)\arrow[equal]{r}\arrow{d}&\Pic(X)\arrow[two heads]{r}\arrow{d}[swap]{\delta}&\Pic(X)\otimes_{\mathbf{Z}}\mathbf{F}_p\arrow[hook]{d}\arrow{dr}{c_1\otimes\mathbf{F}_p}\\
		\H^2(X,\mu_n)\arrow{r}{\cdot\frac{n}{p}}\arrow{d}&
		\H^2(X,\mu_p)\arrow{r}{\sim}\arrow{d}&\H^1(X,\mathbf{G}_m/\mathbf{G}_m^{\times p})\arrow{r}[swap]{\dlog}&\H^1(X,\Omega^1_X)\\
		\Br(X)\arrow{r}{\cdot\frac{n}{p}}&\Br(X)&&
	\end{tikzcd}
  \end{equation}
  where the left columns are fragments of the long exact sequences induced by the Kummer sequences for $n$ and $p$. Using the commutative diagram~\eqref{eq:a square for some flat coho dlogs}, we see that the horizontal composition $\H^2(X,\mu_n)\to\H^1(X,\Omega^1_X)$ is the $\dlog$ map~\eqref{eq:map15}. Let $s\in\mathbf{Z}$ be a lift of $\lambda$. We have $\dlog(\alpha)=c_1([L^{\otimes s}])$. Because $X$ is not superspecial, the map $\dlog$ in~\eqref{eq:the diagram that makes the proof work} is injective (Proposition \ref{prop:dlog is nonzero sometimes}). It follows that $\frac{n}{p}\alpha=\delta([L^{\otimes s}])$, and therefore the image of $\frac{n}{p}\alpha$ in the Brauer group vanishes. This implies that $p$ divides $\frac{n}{\ord(\alpha_{\Br})}$, contrary to our assumption.
\end{proof}

 Combined with Proposition \ref{prop:smooth2}, we obtain the following result on the smoothness of universal deformation spaces.

\begin{theorem}\label{thm:formally smooth def space, one line bundle}
  Let $X$ be a K3 surface over $k$. Let $\alpha\in\H^2(X,\mu_n)$ be a class and let $L$ be a line bundle on $X$ which is not a $p$th power. Assume that one of the following holds.
  \begin{enumerate}
      \item $n$ is coprime to $p$ and either $h<\infty$ or $h=\infty$ and $\sigma_0\geq 2$.
      \item $p$ divides $n$, $p$ does not divide $\frac{n}{\ord(\alpha_{\Br})}$, and either $h<\infty$ or $h=\infty$ and $\sigma_0\geq 3$.
  \end{enumerate}
  The formal deformation space $\Def_{(X,\alpha,L)}$ is smooth over $W$.
\end{theorem}

Combined with the algebraization result of Proposition \ref{prop:02}, Theorem \ref{thm:formally smooth def space, one line bundle} implies the existence of lifts (even over $W$) outside of a small locus of exceptional cases. We will treat the general case when $\Def_{(X,\alpha,L)}$ is not smooth using global methods in \S\ref{sec:moduli of twisted K3 surfaces}.

\subsection{Generalization to the case of multiple line bundles}\label{ssec:multiple line bundles}


We generalize the preceding results to the case of multiple line bundles. Let $X$ be a supersingular K3 surface.

\begin{proposition}\label{prop:multiple line bundles, 1}
  Let $L_1,\dots,L_m$ be line bundles on $X$ whose classes in $\Pic(X)\otimes\mathbf{F}_p$ generate a subspace of dimension $m$. Let $Q\subset \Pic(X)\otimes k$ be the $k$-vector space generated by the classes $[L_1],\dots,[L_m]$. If $i$ is a non-negative integer such that $\sigma_0\geq m+i$, then the subspace $Q$ has trivial intersection with $K+\varphi(K)+\dots+\varphi^i(K)$.
\end{proposition}
\begin{proof}
Assume that $\sigma_0\geq m+i$ and that the intersection of $Q$ and $K+\varphi(K)+\dots+\varphi^i(K)$ is nonzero. Let $e$ be a characteristic vector for $K$. We then find a relation of the form
    \begin{equation}\label{eq:relation 1}
        \sum_{j=1}^m\lambda_j[L_j]=\sum_{j=1}^{\sigma_0+i-1}\mu_je_j.
    \end{equation}
    Let $N$ be the number of nonzero $\lambda_j$. We have $1\leq N\leq m$. By dividing, we may assume that $\lambda_j=1$ for some $j$. Applying $\varphi-1$ to both sides of~\eqref{eq:relation 1}, we find a relation of the form
    \begin{equation}\label{eq:relation 2}
        \sum_{j=1}^m\lambda'_j[L_j]=\sum_{j=1}^{\sigma_0+i}\mu'_je_j.
    \end{equation}
    We have $N'<N$, where $N'$ is the number of nonzero $\lambda'_j$. Moreover, as we assume $\sigma_0\geq m+i$, the vectors $e_0,\dots,e_{\sigma_0+i}$ are linearly independent. Thus, the sum on the right hand side of~\eqref{eq:relation 2} is necessarily nonzero, because if $\mu_je_j$ is a nonzero term in the right hand side of~\eqref{eq:relation 1} with the largest index, then $\mu'_{j+1}=\mu_j^p$. We continue in this manner until all of the $\lambda_j$ are zero. We are then left with a nontrivial linear relation between the vectors $e_0,\dots,e_{\sigma_0+m+i-1}$. But $\sigma_0\geq m+i$ implies that these vectors are linearly independent, so this is a contradiction. 
    
    
\end{proof}

\begin{proposition}\label{prop:omnibus linearly indep}
  Let $L_1,\dots,L_m$ be line bundles on $X$ whose classes in $\Pic(X)\otimes\mathbf{F}_p$ generate a subspace of dimension $m$. Let $P\subset\H^2_{\dR}(X)$ be the $k$-vector space generated by the classes $c_1^{\dR}(L_1),\dots,c_1^{\dR}(L_m)$. If $\sigma_0\geq m+i$ for some $i\geq 0$, then $P$ has dimension $m$ and has trivial intersection with the subspace $V_i$. In particular,
  \begin{enumerate}
      \item if $\sigma_0\geq m$, then $P$ has dimension $m$,
      \item if $\sigma_0\geq m+1$, then $P$ has dimension $m$, and has trivial intersection with $F^2_H$, and
      \item if $\sigma_0\geq m+2$, then $P$ has dimension $m$, and has trivial intersection with $F^2_H+F^2_C$.
  \end{enumerate}
\end{proposition}
\begin{proof}
    By Proposition \ref{prop:multiple line bundles, 1}, the $k$-subspace $Q\subset\Pic(X)\otimes k$ generated by the classes $[L_1],\dots,[L_m]$ has trivial intersection with the subspace $K+\varphi(K)+\dots+\varphi^i(K)$. As $Q=\varphi(Q)$, this implies that $Q$ has trivial intersection with $\varphi(K)$ (even when $i=0$). The kernel of the map
    \[
        c_1^{\dR}\otimes k:\Pic(X)\otimes k\to\H^2_{\dR}(X)
    \]
    is $\varphi(K)$, and so $c_1^{\dR}\otimes k$ maps $Q$ isomorphically to $P$. We conclude that $P$ has dimension $m$, and has trivial intersection with $V_i$.
\end{proof}

In particular, we obtain the following result.

 \begin{proposition}\label{prop:omnibus linearly indep 2}
 With the assumptions of Proposition \ref{prop:omnibus linearly indep}, let $\overline{P}\subset\H^1(X,\Omega^1_X)$ be the $k$-vector space generated by the classes $c_1(L_1),\dots,c_1(L_m)$. If $\sigma_0\geq m+1$, then $\overline{P}$ has dimension $m$.
\end{proposition}

The following result generalizes Proposition \ref{prop:linearly indep}.

\begin{proposition}
    Let $X$ be a K3 surface. Let $L_1,\dots,L_m$ be line bundles on $X$ whose classes in $\Pic(X)\otimes\mathbf{F}_p$ generate a subspace of dimension $m$. Let $n$ be an integer which is divisible by $p$ and let $\alpha\in\H^2(X,\mu_n)$ be a class such that $p$ does not divide $\frac{n}{\ord(\alpha_{\Br})}$. If $X$ has finite height or is supersingular with Artin invariant $\sigma_0\geq m+2$, then the classes $\dlog(\alpha),c_1(L_1),\dots,c_1(L_m)$ in $\H^1(X,\Omega^1_X)$ are linearly independent, and generate a subspace of dimension $m+1$.
\end{proposition}
\begin{proof}
    By Proposition \ref{prop:omnibus linearly indep 2}, the classes $c_1(L_1),\dots,c_1(L_m)$ are linearly independent. We now reason as in Proposition \ref{prop:linearly indep}: suppose that there is a relation
    \[
        \dlog(\alpha)=\sum_i\lambda_ic_1(L_i)
    \]
    for some $\lambda_i\in k$. We then have
    \[
        \dlog(\alpha)=\sum_i\lambda_ic_1^{\dR}(L_i)+\sigma
    \]
    for some $\sigma\in F^2_H\cap F^1_C$. As before, we deduce that
    \[
        \sum_i(\lambda_i^p-\lambda_i)c_1^{\dR}(L_i)\in F^2_H+F^2_C.
    \]
    Using \ref{prop:trivial intersection finite height case} (if $h<\infty$) or \ref{prop:omnibus linearly indep} (if $h=\infty$), we conclude that $\lambda_i^p-\lambda_i=0$ for all $i$. Hence, $\lambda_i\in\mathbf{F}_p$ for all $i$. Choose lifts $s_i\in\mathbf{Z}$ of the $\lambda_i$. We obtain
    \[
        \frac{n}{p}\alpha=\sum_i\delta([L_i^{\otimes s_i}])
    \]
    and therefore the image of $\frac{n}{p}\alpha$ in the Brauer group is trivial, a contradiction.
\end{proof}

We record the following consequences for formal deformation spaces.

\begin{corollary}\label{cor:formally smooth def space}
    Let $X$ be a K3 surface over $k$. Let $L_1,\dots,L_m$ be a collection of line bundles on $X$ whose classes in $\Pic(X)\otimes\mathbf{F}_p$ generate a subspace of dimension $m$. If $X$ has finite height or is supersingular with Artin invariant $\sigma_0\geq m+1$, then $\Def_{(X,L_1,\dots,L_m)}$ is formally smooth over $W$.
\end{corollary}

\begin{corollary}\label{cor:formally smooth def space, twisted version}
  Let $X$ be a K3 surface over $k$. Let $L_1,\dots,L_m$ be a collection of line bundles on $X$ whose classes in $\Pic(X)\otimes\mathbf{F}_p$ generate a subspace of dimension $m$. Let $n$ be a positive integer and let $\alpha\in\H^2(X,\mu_n)$ be a flat cohomology class. Assume one of the following holds.
  \begin{enumerate}
      \item $n$ is coprime to $p$ and either $h<\infty$ or $h=\infty$ and $\sigma_0\geq m+1$.
      \item $p$ divides $n$, $p$ does not divide $\frac{n}{\ord(\alpha_{\Br})}$, and either $h<\infty$ or $h=\infty$ and $\sigma_0\geq m+2$.
  \end{enumerate}
  The universal deformation space $\Def_{(X,\alpha,L_1,\dots,L_m)}$ is formally smooth over $W$.
\end{corollary}

    

\section{Arithmetic moduli of twisted K3 surfaces}\label{sec:moduli of twisted K3 surfaces}

In this section we introduce some global moduli spaces of twisted polarized K3 surfaces over $\Spec \mathbf{Z}$ and describe some of their basic geometric properties.

\begin{definition}
	Fix positive integers $n$ and $d$. Define $\ms M^n_d$ to be the stack over $\Spec\mathbf{Z}$ whose objects over a scheme $S$ are tuples $(X,\alpha,L)$, where $f:X\to S$ is a family of K3 surfaces, $\alpha\in\H^0(S,R^2f_*\mu_n)$, and $L\in\H^0(S,\uPic_{X/S})$ is a section whose restriction to every geometric fiber of $X\to S$ is a primitive ample class of degree $2d$.
\end{definition}

If $k$ is an algebraically closed field and $f:X\to\Spec k$ is a K3 surface, we have
\[
    \H^2(X,\mu_n)=\H^0(\Spec k,R^2f_*\mu_n).
\]
Thus, the $k$-points of $\ms M^n_d$ are tuples $(X,\alpha,L)$ where $X$ is a K3 surface over $k$, $\alpha\in\H^2(X,\mu_n)$, and $L\in\Pic(X)$ is an ample class of degree $2d$.

Let $\ms M_d$ denote the usual moduli stack of polarized K3 surfaces of degree $2d$. There is a morphism
\begin{equation}\label{eq:the forgetful map pi}
  \pi:\ms M^{n}_d\to\ms M_d
\end{equation}
given by forgetting the class $\alpha$. If $n=1$ this map is an isomorphism.
\begin{proposition}\label{prop:its DM}
  The moduli stack $\ms M^{n}_d$ is Deligne--Mumford.
\end{proposition}
\begin{proof}
  It is well known that $\ms M_d$ is Deligne--Mumford (see eg. \cite[Ch. 5, Proposition 4.10]{Huy06} or \cite[4.3.3]{MR2263236}). Let $f:\ms X\to\ms M_d$ be the universal polarized K3 surface. We have $\ms M^n_d=R^2f_*\mu_n$ as functors on the category of schemes over $\ms M_d$. By Theorem \ref{thm:representability for R2}, the map $\ms M^n_d\to\ms M_d$ is representable by algebraic spaces. We conclude that $\ms M^n_d$ is a Deligne--Mumford stack.
\end{proof}



If $f:X\to S$ is a morphism, then $R^2f_*\mu_n$ may be computed as the flat sheafification of the functor $T\mapsto \H^2(X\times_ST,\mu_n)$ on the category of $S$-schemes. The following result shows that if $f$ is a family of K3 surfaces then this sheafification may be taken instead in the \'{e}tale topology. This simplification will be important in our discussion of the geometry of $\ms M^n_d$.
\begin{lemma}\label{lem:flat and etale covers killing classes}
    Let $f:X\to S$ be a family of K3 surfaces.
    \begin{enumerate}
        \item If $\alpha\in\H^0(S,R^2f_*\mu_n)$ is any class, then there exists an \'{e}tale cover $S'\to S$ such that $\alpha$ is in the image of the map $\H^2(X\times_SS',\mu_n)\to\H^0(S',R^2f_{S'*}\mu_n)$.
        \item If $\alpha\in\H^2(X,\mu_n)$ is a class and there there exists an fppf cover $S'\to S$ such that $\alpha|_{X\times_SS'}=0$, then there exists an \'{e}tale cover $S'\to S$ such that $\alpha_{X\times_SS'}=0$.
    \end{enumerate}
\end{lemma}
\begin{proof}
    We have $R^0f_*\mu_n=\mu_n$ and $R^1f_*\mu_n=\uPic_{X/S}[n]=0$. The Leray spectral sequence therefore gives an exact sequence
    \[
        0\to\H^2(S,\mu_n)\to\H^2(X,\mu_n)\to\H^0(X,R^2f_*\mu_n)\to\H^3(S,\mu_n).
    \]
    It follows from the Kummer sequence that if $m\geq 2$ then any class in $\H^m(S,\mu_n)$ may by killed by an \'{e}tale cover of $S$. This implies (1). For (2), we note that a class $\alpha\in\H^2(X,\mu_n)$ is killed by an fppf cover of $S$ if and only if $\alpha$ maps to 0 in $\H^0(X,R^2f_*\mu_n)$.
\end{proof}


The following gives some basic geometric properties of $\ms M^{n}_d\to\Spec\mathbf{Z}$, and is the main result of this section. As an immediate consequence of this result, we obtain Theorem \ref{thm:basic lifting}.
\begin{theorem}\label{thm:irreducible components of fibers}
	The morphism $\ms M^{n}_d\to\Spec\mathbf{Z}$ is flat and is a local complete intersection of relative dimension 19.
\end{theorem}
\begin{proof}
	Let $k$ be an algebraically closed field of characteristic $p$, and consider a $k$-point $x\in\ms M^{n}_d(k)$ corresponding to a K3 surface $X$ with an ample class $L$ of degree $2d$ and a class $\alpha\in\H^2(X,\mu_n)$. Let $\widehat{\ms M}_{x}$ be the category cofibered in groupoids over $\cC_W$ whose fiber over $A\in\cC_W$ is the groupoid of 2-commutative diagrams
	\[
	    \begin{tikzcd}
	           \Spec k\arrow{r}\arrow[bend left=25]{rr}{x}&\Spec A\arrow{r}&\ms M^n_d.
	    \end{tikzcd}
	\]
	By Lemma \ref{lem:flat and etale covers killing classes}, if $A$ is an Artinian local ring with residue field $k$ and $f_A:X_A\to\Spec A$ is a relative K3 surface, then the map $\H^2(X_A,\mu_n)\to\H^0(\Spec A,R^2f_{A*}\mu_n)$ is an isomorphism. It follows that the natural map
	\begin{equation}\label{eq:natural map of def spaces}
	    \Def_{(X,\alpha,L)}\to\widehat{\ms M}_{x}
	\end{equation}
	of categories cofibered in groupoids over $\cC_W$ is an isomorphism.

	To show the result, it will therefore suffice to check that $\Def_{(X,\alpha,L)}$ is flat and lci of relative dimension 19 over $\Spf W$. By definition, we have a Cartesian square
	\begin{equation}\label{eq:cart diagram}
		\begin{tikzcd}
			\Def_{(X,\alpha,L)}\arrow[hook]{r}\arrow{d}&\Def_{(X,\alpha)}\arrow{d}\\
			\Def_{(X,L)}\arrow[hook]{r}&\Def_X.
		\end{tikzcd}
	\end{equation}
	By Proposition \ref{prop:one equation} and \cite[1.6]{MR638598}, the inclusion $\Def_{(X,\alpha,L)}\subset\Def_{(X,\alpha_{\Br})}$ is a closed immersion cut out by two equations. As in~\eqref{eq:diagram1111}, we choose coordinates so that the diagram~\eqref{eq:cart diagram} is represented by
	\begin{equation}\label{eq:cart diagram coordinates}
		\begin{tikzcd}
			\Spf W[[t_1,\dots,t_{20},s]]/(f,g)\arrow[hook]{r}\arrow{d}&\Spf W[[t_1,\dots,t_{20},s]]/(g)\arrow{d}\\
			\Spf W[[t_1,\dots,t_{20}]]/(f)\arrow[hook]{r}&\Spf W[[t_1,\dots,t_{20}]]
		\end{tikzcd}
	\end{equation}
	for some functions $f\in W[[t_1,\dots,t_{20}]]$ and $g\in W[[t_1,\dots,t_{20},s]]$. Let $f_0,g_0$ denote the images of $f$ and $g$ modulo $p$. It will suffice to show that $f_0,g_0\in R_0=k[[t_1,\dots,t_{20},s]]$ is a regular sequence.
	
	If $n$ is coprime to $p$, then by Proposition \ref{prop:nice coordinates} we may assume $g=s$. On the other hand, if $p$ divides $n$, then it follows from \cite[Proposition 14]{MR1827026} that the closed formal subscheme of $\Def_{(X,\alpha,L)}\otimes k$ parametrizing deformations whose underlying K3 has infinite height has dimension at most $10$. Therefore the generic point of any irreducible component of $\ms M^n_d\otimes\overline{\mathbf{F}}_p$ has finite height. So, in this case it will suffice to show the result when $X$ has finite height $h$. By Proposition \ref{prop:nice coordinates}, we may assume that $g$ is congruent to $s^k$ modulo $(p,t_1,\dots,t_{20})$ for some positive integer $k$. We conclude that, in either case, it suffices to prove that $f_0,g_0$ is a regular sequence under the additional assumption that $g_0=g'_0+s^k$ for some positive integer $k$ and some $g_0'\in (t_1,\dots,t_{20})\subset R_0$.
	
	To prove this, we first recall that by \cite[1.6]{MR638598} $f_0$ is not a zero divisor in $R_0$. It remains to show that the image of $g_0$ in $R_0/(f_0)$ is not a zero divisor. Suppose that $g_0h_0\in (f_0)$ for some $h_0\in R_0$. Then $f_0$ divides $g_0h_0=(g'_0+s^k)h_0$. Note that $f_0$ is contained in the subring $k[[t_1,\dots,t_{20}]]$, and also in the ideal $(t_1,\dots,t_{20})$ of $R_0$. The same is true for any irreducible factor of $f_0$. But no such element can divide $g'_0+s^k$. Hence, every irreducible factor of $f_0$ divides $h_0$, so $f_0$ divides $h_0$. This completes the proof.
\end{proof}

We record a few remarks regarding the forgetful morphism $\pi:\ms M^{n}_d\to\ms M_d$~\eqref{eq:the forgetful map pi}. The restriction of $\pi$ to $\Spec\mathbf{Z}\left[\frac{1}{n}\right]$ is \'{e}tale. Over geometric points whose residue characteristics divide $n$, we can describe the fibers of $\pi$ as follows. Given a K3 surface $X$ over an algebraically closed field $k$, write $\underline{\H}^2(X,\mu_n)$ for the functor $R^2f_*\mu_n$, where $f:X\to\Spec k$ is the structural morphism. Thus, $\underline{\H}^2(X,\mu_n)$ is a group scheme over $k$ whose group of $k$-points is $\H^2(X,\mu_n)$, and the fiber of $\pi$ over a geometric point $[(X,L)]\in\ms M_d(k)$ is exactly $\underline{\H}^2(X,\mu_n)$. Let $\underline{\U}^2(X,\mu_n)\subset\underline{\H}^2(X,\mu_n)$ denote the connected component of the identity, and let $\underline{\D}^2(X,\mu_n)$ be the quotient, so that we have a short exact sequence
\[
    0\to\underline{\U}^2(X,\mu_n)\to\underline{\H}^2(X,\mu_n)\to\underline{\D}^2(X,\mu_n)\to 0.
\]
If $n$ is invertible in $k$, then $\underline{\U}^2(X,\mu_n)$ is trivial, and $\underline{\D}^2(X,\mu_n)\cong(\underline{\mathbf{Z}}/n\underline{\mathbf{Z}})^{\oplus 22}$. 
Suppose that $k$ has characteristic $p>0$. Let $p^r$ be the largest power of $p$ dividing $n$ and set $m=n/p^r$. The completion of $\underline{\U}^2(X,\mu_n)$ at the identity is isomorphic to the $p^r$ torsion in the formal Brauer group $\widehat{\Br}_X$. This determines $\underline{\U}^2(X,\mu_n)$ up to isomorphism. In particular, if $X$ has finite height, then $\underline{\U}^2(X,\mu_n)$ is a purely infinitesimal group scheme of length $p^{rh}$, and if $h=1$, then
\[
    \underline{\U}^2(X,\mu_n)\cong\mu_{p^r}.
\]
If $h=\infty$, then as long as $r\geq 1$ we have
\[
    \underline{\U}^2(X,\mu_n)\cong\mathbf{G}_a.
\]
The \'{e}tale quotient $\underline{\D}^2(X,\mu_n)$ can also be computed explicitly. If $h<\infty$, then
\[
    \underline{\D}^2(X,\mu_n)\cong (\underline{\mathbf{Z}}/m\underline{\mathbf{Z}})^{\oplus 22}\oplus(\underline{\mathbf{Z}}/p^r\underline{\mathbf{Z}})^{\oplus 22-2h}
\]
and if $h=\infty$ then
\[
    \underline{\D}^2(X,\mu_n)\cong (\underline{\mathbf{Z}}/m\underline{\mathbf{Z}})^{\oplus 22}\oplus(\underline{\mathbf{Z}}/p^r\underline{\mathbf{Z}})^{\oplus 22-2\sigma_0}.
\]
In particular, if $p$ divides $n$ then the forgetful morphism
    \[
        \pi_p:\ms M^n_d\otimes \mathbf{F}_p\to\ms M_d\otimes \mathbf{F}_p
    \]
is not flat, and the height--Artin invariant stratification provides a flattening stratification. Furthermore, the generic geometric fiber of $\pi_p$ is nonreduced, and so the forgetful map is inseparable.

\begin{corollary}\label{cor:non reduced components}
  If $p$ divides $n$, then the stack $\ms M^n_d\otimes\mathbf{F}_p$ has an irreducible component which is everywhere nonreduced.
\end{corollary}
\begin{proof}
    The morphism $\pi_p$ has a section
    \[
        \sigma:\ms M_d\otimes\mathbf{F}_p\to \ms M_d^n\otimes\mathbf{F}_p
    \]
    defined on $S$-points by $(X,L)\mapsto (X,\delta(L),L)$, where
    \[
        \delta:\H^0(S,R^1f_*\mathbf{G}_m)\to\H^0(S,R^2f_*\mu_n)
    \]
    is the boundary map coming from the Kummer sequence. Let $W$ be an irreducible component of $\ms M_d\otimes\mathbf{F}_p$, and let $Z$ be the irreducible component of $\ms M^n_d\otimes\mathbf{F}_p$ which contains $\sigma(W)$. We claim that $Z$ is everywhere nonreduced. To see this, let $W_1\subset W$ be the open dense subset parametrizing ordinary K3 surfaces. Let $Z_1=\pi_p^{-1}(W_1)\cap Z\subset Z$ be its preimage in $Z$. Every geometric fiber of $Z_1\to W_1$ is a disjoint union of copies of $\mu_p$. The subscheme $\sigma(W_1)\subset Z_1$ is reduced, and is not equal to $Z_1$. It follows that $Z_1$, and hence $Z$, is everywhere nonreduced.
\end{proof}

\subsection{Multiple line bundles}\label{sec:multiple line bundles, global results}

We indicate the extension of the preceding results to the case of multiple line bundles. The proofs are essentially the same, so we shall be brief.

\begin{definition}
    Let $\Lambda$ be a lattice. Let $\ms M_{\Lambda}$ be the moduli stack over $\Spec\mathbf{Z}$ whose objects over a scheme $S$ are pairs $(X,\iota)$, where $f:X\to S$ is a family of K3 surfaces and $\iota:\underline{\Lambda}_S\hookrightarrow\uPic_{X/S}$ is an isometric embedding whose image contains a primitive ample class.
    
    Let $n$ be a positive integer. We let $\ms M^n_{\Lambda}$ be the moduli stack parametrizing tuples $(X,\iota,\alpha)$, where $X$ and $\iota$ are as before, and $\alpha\in\H^0(S,R^2f_*\mu_n)$.
\end{definition}

Both $\ms M_{\Lambda}$ and $\ms M^n_{\Lambda}$ are Deligne-Mumford stacks over $\Spec\mathbf{Z}$. Write $m=\rk(\Lambda)$.

\begin{proposition}\label{prop:flatness for moduli space, many line bundles}
  Suppose that $m\leq 10$. The map $\ms M_{\Lambda}\to\Spec\mathbf{Z}$ is a flat local complete intersection of relative dimension $20-m$. Every irreducible component of every geometric fiber of $\ms M_{\Lambda}\to\Spec\mathbf{Z}$ is generically smooth of dimension $20-m$.  
\end{proposition}
\begin{proof}
    Let $p$ be a prime and consider an irreducible component $Z\subset\ms M_{\Lambda}\otimes\overline{\mathbf{F}}_p$. We have $\dim(Z)\geq 20-m$. By \cite[Proposition 14]{MR1827026}, the supersingular locus in $Z$ has dimension at most $9$, so $Z$ contains a geometric point $x$ parametrizing a K3 surface say $X$ of finite height. By Proposition \ref{prop:inj for c 1 tensor k}, the image of $\Lambda$ in $\H^1(X,\Omega^1_X)$ has dimension $m$, and hence the tangent space to $\ms M_{\Lambda}$ at $x$ has dimension $20-m$. It follows that $Z$ is smooth at $x$ of dimension $20-m$. Hence, $Z$ is generically smooth of dimension $20-m$.
    
    We know that the local deformation space to $\ms M_{\Lambda}$ at any geometric point is a subscheme of $\Spf W[[t_1,\dots,t_{20}]]$ cut out by $m$ equations. Our computation of the dimension of $Z$ therefore implies that $\ms M_{\Lambda}$ is flat and lci over $\Spec\mathbf{Z}$ of relative dimension $20-m$.
\end{proof}

\begin{proposition}\label{prop:flatness for twisted moduli space, many line bundles}
  Let $n$ be a positive integer. If $m\leq 9$, then the map $\ms M^n_{\Lambda}\to\Spec\mathbf{Z}$ is a flat local complete intersection of relative dimension $20-m$. If $m\leq 10$, then the same conclusion holds for the restriction $\ms M^n_{\Lambda}\otimes\mathbf{Z}\left[\frac{1}{n}\right]\to\Spec\mathbf{Z}\left[\frac{1}{n}\right]$.
\end{proposition}
\begin{proof}
    The local deformation space of any geometric point of $\ms M_{\Lambda}^n$ is a closed subscheme of $\Spf W[[t_1,\dots,t_{20},s]]$ cut out by $m+1$ equations. As in the proof of \ref{prop:flatness for moduli space, many line bundles}, to show the first claim it will suffice to show that if $p$ is a prime and $\ms Z$ is an irreducible component of $\ms M^n_{\Lambda}\otimes\overline{\mathbf{F}}_p$ then $\dim(\ms Z)=20-m$. To show this, consider the forgetful map $\pi:\ms M^n_{\Lambda}\to\ms M_{\Lambda}$. The fibers of this map have dimension at most one. Let $Z\subset \ms M_{\Lambda}\otimes\overline{\mathbf{F}}_p$ be an irreducible component containing the image of $\ms Z$. By \cite[Proposition 14]{MR1827026} the supersingular locus in $Z$ has dimension at most $9$, and therefore the supersingular locus in $\ms Z$ has dimension at most $10$. Because $m\leq 9$, we deduce that $\ms Z$ contains a geometric point $x$ parametrizing a K3 surface of finite height. By \ref{prop:flatness for moduli space, many line bundles} $Z$ has dimension $20-m$, and by \ref{prop:nice coordinates}, the fiber of $\ms Z\to Z$ containing $x$ is zero-dimensional. We conclude that $\ms Z$ has dimension $20-m$. This gives the first claim.
    
    For the second, we note that the map $\ms M^n_{\Lambda}\otimes\mathbf{Z}\left[\frac{1}{n}\right]\to\Spec\mathbf{Z}\left[\frac{1}{n}\right]$ is \'{e}tale. The result therefore follows from \ref{prop:flatness for moduli space, many line bundles}.
\end{proof}

We highlight the following consequence for the existence of liftings of twisted K3 surfaces together with a collection of line bundles. In the non-twisted case, this problem has been considered by Lieblich--Olsson \cite{LO15} and Lieblich--Maulik \cite{MR3934849}.

\begin{theorem}\label{thm:mega lifting}
	Let $X$ be a K3 surface over an algebraically closed field $k$ of characteristic $p>0$ and let $\alpha_{\Br}\in\Br(X)$ be a Brauer class. Let $\alpha\in\H^2(X,\mu_n)$ be a class whose image in the Brauer group is $\alpha_{\Br}$. Let $V\subset\Pic(X)$ be a saturated sublattice of rank $m$ containing an ample class. Suppose that at least one of the following holds.
	\begin{enumerate}
		\item[{\rm (A)}] $X$ has finite height.
		\item[{\rm (B)}] $m\leq 9$.
		\item[{\rm (C)}] $n$ is coprime to $p$ and $m\leq 10$.
	\end{enumerate}
	There exists
	\begin{enumerate}
		\item[{\rm (1)}] a DVR $R$ with fraction field $K$ of characteristic 0 and residue field $k$,
		\item[{\rm (2)}] a K3 surface $\widetilde{X}$ over $R$ and an isomorphism $\ms X\otimes_R k\cong X$,
		\item[{\rm (3)}] a class $\widetilde{\alpha}\in\H^2(\widetilde{X},\mu_n)$ such that $\widetilde{\alpha}|_X=\alpha$, and
		\item[{\rm (4)}] a sublattice $\ms V\subset\Pic(\widetilde{X})$ which over $k$ specializes to the inclusion $V\subset\Pic(X)$ 
		and which for every algebraically closed field $L$ containing $K$ induces an isomorphism $\ms V|_{\widetilde{X}_L}=\Pic(\widetilde{X}_L)$.
	\end{enumerate}
 \end{theorem}

\subsection{Moduli of primitive twisted K3 surfaces}\label{sec:moduli of primitive twisted K3 surfaces}

Note that in the definition of $\ms M^n_d$ we allow the class $\alpha$ to have order smaller than $n$, and in particular to vanish. Furthermore, we have imposed no restriction on the relationship between $L$ and $\alpha$. As a consequence, the stack $\ms M^n_d$ has some undesirable behavior (eg. Corollary \ref{cor:non reduced components}). Inspired by Brakkee \cite{MR4126898}, we will consider a variant of the stack $\ms M^n_d$ in which we require the class $\alpha$ to be primitive (in a certain sense) with respect to $L$ (for the precise relation with Brakkee's definitions, we refer to Remark \ref{rem:Brakkee}). This stack will turn out to have some better properties.

Let $(X,L)$ be a polarized K3 surface over an algebraically closed field. We set
\[
    \H^2(X,\mu_n)_{\prim}=\H^2(X,\mu_n)/\langle\delta(L)\rangle
\]
where $\delta:\Pic(X)\to\H^2(X,\mu_n)$ is the boundary map from the Kummer sequence, and $\langle\delta(L)\rangle=\mathbf{Z}/n\mathbf{Z}\cdot\delta(L)$ is the cyclic subgroup generated by $\delta(L)$. By the exactness of the Kummer sequence, the map $\H^2(X,\mu_n)\to\Br(X)$ descends to a map
\[
    \H^2(X,\mu_n)_{\prim}\to\Br(X)
\]
which we will denote by $\alpha\mapsto\alpha_{\Br}$, as before.

We make a similar definition in families. Consider a pair $(X,L)$ where $f:X\to S$ is a family of K3 surfaces and $L\in\H^0(S,\uPic_{X/S})$ is a class whose restriction to every geometric fiber is primitive. Consider the boundary map $\delta:R^1f_*\mathbf{G}_m\to R^2f_*\mu_n$ coming from the Kummer sequence. The global section $L$ induces a map of group schemes $\underline{\mathbf{Z}}_S\to R^1f_*\mathbf{G}_m$. We define $(R^2f_*\mu_n)_{\prim}$ to be the quotient of the composition of this map with $\delta$. Thus, we have a short exact sequence
\begin{equation}\label{eq:primitive SES}
    0\to\underline{\mathbf{Z}}_S/n\underline{\mathbf{Z}}_S\xrightarrow{1\mapsto\delta(L)}R^2f_*\mu_n\to (R^2f_*\mu_n)_{\prim}\to 0.
\end{equation}
For the motivation behind this notation we refer to Remark \ref{rem:Brakkee}.
\begin{definition}\label{def:some more open substacks}
	Let 
	$\mc M_d[n]$ be the stack over $\Spec\mathbf{Z}$ whose objects over a scheme $S$ are triples $(X,\alpha,L)$ where $f:X\to S$ is a family of K3 surfaces, $L\in\H^0(S,\uPic_{X/S})$ is a section whose restriction to each geometric fiber of $X\to S$ is a primitive ample class of degree $2d$, and $\alpha\in\H^0(S,(R^2f_*\mu_n)_{\prim})$.
	
	Let 
	$\mc M_d^n\subset\mc M_d[n]$ be the substack such that for all geometric points $s\in S$ the class $\alpha_s$ has order $n$.
\end{definition}

\begin{proposition}\label{prop:the stacks are DM}
    The stacks $\mc M_d[n]$ and $\mc M_d^n$ are Deligne--Mumford.
\end{proposition}
\begin{proof}
    Let $f:X\to\ms M_d$ be the universal polarized K3 surface. The map $\underline{\mathbf{Z}}/n\underline{\mathbf{Z}}\to R^2f_*\mu_n$ is injective. As $\underline{\mathbf{Z}}/n\underline{\mathbf{Z}}$ is flat over $\ms M_d$, the quotient sheaf $(R^2f_*\mu_n)_{\prim}=\mc M_d[n]$ is representable over $\ms M_d$. As $\ms M_d$ is Deligne--Mumford, we conclude the same for $\mc M_d[n]$. The inclusion $\mc M_d^n\subset\mc M_d[n]$ is open, so $\mc M_d^n$ is Deligne--Mumford as well.
\end{proof}

The fiber of $\mc M_d^n$ over the complex numbers has been studied by Brakkee \cite{MR4126898} (see Remark \ref{rem:Brakkee}). Note that the stack $\mc M_d^n$ is large enough to still allow for interesting variation in the Brauer class $\alpha_{\Br}$. In particular, it admits a reasonable notion of Noether--Lefschetz loci.

The stacks we have defined are related by maps
\begin{equation}\label{eq:map of stacks}
    \ms M^n_d\to\mc M_d[n]\supset\mc M_d^n
\end{equation}
where the left arrow is an $n$-fold cyclic \'{e}tale cover (corresponding to the short exact sequence~\eqref{eq:primitive SES} of sheaves) and the right map is an open inclusion.

\begin{remark}\label{rem:Brakkee}
Brakkee \cite[Definition 2.1]{MR4126898} studies a functor on schemes over the complex numbers whose $\mathbf{C}$-points are isomorphism classes of tuples $(X,\alpha,L)$, where $(X,L)$ is a primitively polarized K3 surface of degree $2d$ and $\alpha\in\Hom(\H^2(X,\mathbf{Z})_{\prim},\mathbf{Z}/n\mathbf{Z})$, as well as the subfunctor of tuples such that $\alpha$ has order $n$. Brakkee shows that these functors admit coarse moduli spaces \cite[Theorem 1]{MR4126898}, which are moreover constructed explicitly in terms of the period domain for complex K3 surfaces.

As explained in \cite[\S 2.1]{MR4126898}, there is a canonical isomorphism 
\[
    \Hom(\H^2(X,\mathbf{Z})_{\prim},\mathbf{Z}/n\mathbf{Z})\cong\H^2(X,\mu_n)/\langle\delta(L)\rangle.
\]
Thus, Brakkee's functors are exactly the functors of isomorphism classes associated to the fibers $\mc M_d[n]\otimes\mathbf{C}$ and $\mc M^n_d\otimes\mathbf{C}$ of our moduli stacks over the complex numbers. Our results in this section therefore give a natural extension of Brakkee's moduli spaces to spaces defined over the integers. In particular, Proposition \ref{prop:the stacks are DM} gives a purely algebraic proof of Theorem 1 of \cite[Theorem 1]{MR4126898}.
\end{remark}

We consider the singular locus of the fiber $\mc M_d^n\otimes\mathbf{F}_p$. It is convenient to make the following definition. Suppose that $(X,\alpha_{\Br})$ is a twisted K3 surface over an algebraically closed field $k$ of characteristic $p>0$. If $X$ is supersingular of Artin invariant $\sigma_0$, we define the \textit{Artin invariant} of $(X,\alpha_{\Br})$ by
\[
  \sigma_0(X,\alpha_{\Br})=\begin{cases}
  \sigma_0(X)+1,&\mbox{ if }\alpha_{\Br}\neq 0\\
  \sigma_0(X),&\mbox{ if }\alpha_{\Br}=0.
  \end{cases}
\]
For a more motivated approach to this definition we refer to \cite[Section 3.4]{BL17}. One consequence of this convention is that if $p$ divides $n$ then for any $1\leq\sigma\leq 11$ the locus in $\ms M^{n}_d\otimes\mathbf{F}_p$ or in $\mc M_d^n\otimes\mathbf{F}_p$ parametrizing tuples $(X,\alpha,L)$ such that $(X,\alpha_{\Br})$ is supersingular of Artin invariant $\leq \sigma$ has dimension $\sigma-1$, as in the untwisted case.

\begin{proposition}\label{prop:singularities}
    The fiber $\mc M_d^n\otimes\mathbf{Q}$ is regular. Furthermore, if $p$ is a prime, then we have the following descriptions of the singular loci of the fiber $\mc M_d^n\otimes\mathbf{F}_p$.
  \begin{enumerate}
      \item If $p$ does not divide $2dn$, then $\mc M_d^n\otimes\mathbf{F}_p$ is regular.
      \item If $p$ divides $2d$ but not $n$, then $\mc M_d^n\otimes\mathbf{F}_p$ is non-singular away from the locus of supersingular points with Artin invariant $\sigma_0=1$.
      \item If $p$ divides $n$, then $\mc M_d^n\otimes\mathbf{F}_p$ is non-singular away from the locus of supersingular points with Artin invariant $\sigma_0\leq 3$.
  \end{enumerate}  
\end{proposition}
\begin{proof}
  Let $k$ be an algebraically closed field, and consider a $k$-point $x\in\mc M_d^n(k)$ corresponding to a K3 surface $X$ with an ample class $L$ of degree $2d$ and a class $\alpha\in\H^2(X,\mu_n)_{\prim}$ of order $n$. Let $\alpha'\in\H^2(X,\mu_n)$ be a lift of $\alpha$, and let $x'=(X,\alpha',L)$ be the resulting point of $\ms M_d^n$. As the quotient map $\ms M_d^n\to\mc M_d[n]$ is \'{e}tale and $\mc M_d^n$ is open in $\mc M_d[n]$, we have that $\mc M_d^n$ is regular at $x$ if and only if $\ms M_d^n$ is regular at $x'$. In particular, this shows that $\mc M_d^n\otimes\mathbf{Q}$ is regular.
 
  Suppose that $k$ has characteristic $p$. Arguing as in Theorem \ref{thm:irreducible components of fibers}, we have that $\ms M_d^n$ is regular at $x'$ if and only if the universal deformation space $\Def_{(X,\alpha',L)}$ is formally smooth over $W$.
  
  Suppose that $p$ does not divide $2d$. Then $L^2$ is nonzero modulo $p$, so $c_1^{\dR}(L)^2\neq 0$. As $F^2_H$ is isotropic, we have $c_1^{\dR}(L)\notin F^2_H$, and so $c_1(L)$ is nonzero. If also $p$ does not divide $n$, then Proposition \ref{prop:smooth2} implies that $\Def_{(X,\alpha',L)}$ is formally smooth, which gives (1) (see also \cite[Lemma 4.1.3]{MR2263236}).
  
  Now, suppose that $p$ does not divide $n$, but possibly $p$ divides $2d$. By Theorem \ref{thm:formally smooth def space, one line bundle}, if $\ms M_d^n$ is singular at $x'$ then $X$ is superspecial.
  The Brauer group of a supersingular K3 surface is $p$-torsion \cite[Theorem 4.3]{Artin74}, so $\alpha_{\Br}=0$, and hence $\sigma_0(X,\alpha_{\Br})=1$. This gives (2).
  

  Finally, suppose that $p$ divides $n$. We consider two cases. Suppose first that $p$ does not divide $\frac{n}{\mathrm{ord}(\alpha'_{\Br})}$. By Theorem \ref{thm:formally smooth def space, one line bundle}, if $\ms M^n_d$ is singular at $x'$ then $X$ is supersingular and $\sigma(X)\leq 2$, which implies $\sigma_0(X,\alpha'_{\Br})\leq 3$. Suppose that $p$ divides $\frac{n}{\mathrm{ord}(\alpha'_{\Br})}$, or equivalently that $\frac{n}{p}\alpha'\in\H^2(X,\mu_p)$ has trivial Brauer class. By Proposition \ref{prop:smooth2}, $\ms M_d^n$ is nonsingular at $x'$ if the classes $\dlog(\alpha')$ and $c_1(L)$ are nonzero and linearly independent in $\H^1(X,\Omega^1_X)$ (in fact, as $\ms M^d_n$ has relative dimension 19, this is an if and only if). These classes are the images of $\frac{n}{p}\cdot\alpha'$ and $\delta(L)$ under the map
   \[
    \dlog:\H^2(X,\mu_p)\to\H^1(X,\Omega^1_X).
   \]
  As $\frac{n}{p}\cdot\alpha'$ has trivial Brauer class, we have $\alpha'=\delta(M)$ for some $M\in\Pic(X)$. By assumption, $\alpha'$ has order $n$ modulo $\langle\delta(L)\rangle$. It follows that $\frac{n}{p}\alpha'$ has order $p$ modulo $\langle \delta(L)\rangle$. Hence, the images $\overline{L}$ and $\overline{M}$ of $L$ and $M$ in $\Pic(X)\otimes\mathbf{F}_p$ are linearly independent. By Proposition \ref{prop:omnibus linearly indep 2}, we conclude that if the classes $c_1(L)$ and $c_1(M)=\dlog(\alpha')$ are not linearly independent in $\H^1(X,\Omega^1_X)$, then $X$ is supersingular with $\sigma_0(X)\leq 2$. As before, this implies $\sigma_0(X,\alpha_{\Br})\leq 3$, and we obtain (3).
  
\end{proof}

\begin{theorem}
  The morphism $\mc M_d^n\to\Spec\mathbf{Z}$ is flat and a local complete intersection of relative dimension 19. For each prime $p$, every connected component of $\mc M_d^n\otimes\overline{\mathbf{F}}_p$ is reduced and irreducible and is generically smooth of dimension 19.
\end{theorem}
\begin{proof}
    By Theorem \ref{thm:irreducible components of fibers}, we have an \'{e}tale cover $\ms M_d^n\to\mc M_d[n]$ of $\mc M_d[n]$ by a flat local complete intersection of relative dimension 19. These properties are \'{e}tale local on the source (see \cite[069P]{stacks-project}) and so descend to $\mc M_d[n]\to\Spec\mathbf{Z}$. The inclusion $\mc M_d^n\subset\mc M_d[n]$ is open, and hence these properties also hold for $\mc M_d^n\to\Spec\mathbf{Z}$.
   
   Fix a prime $p$ and let $Z\subset\mc M_d^n\otimes\overline{\mathbf{F}}_p$ be an irreducible component. By \cite[Proposition 14]{MR1827026} there is a dense open subset of $Z$ parametrizing K3 surfaces of finite height. Proposition \ref{prop:singularities} implies in particular that $Z$ is generically smooth, and hence reduced. The same is true for the intersection of any two irreducible components. We conclude that every irreducible component is reduced and generically smooth of dimension 19, and that no two irreducible components intersect.
\end{proof}

\section{An application to twisted derived equivalences}\label{sec:application to twisted derived equivalences}

In this section we give an application of our results to a problem concerning derived equivalences of twisted K3 surfaces. Suppose that $(X,\alpha_{\Br})$ and $(Y,\beta_{\Br})$ are twisted K3 surfaces over an algebraically closed field $k$. Given a Fourier--Mukai equivalence $\Phi_P:D(X,\alpha_{\Br})\xrightarrow{\sim} D(Y,\beta_{\Br})$ one can ask if the induced cohomological transform is  orientation preserving (sometimes also called ``signed'').\footnote{Over the complex numbers, this notion is usually phrased in terms of the Hodge structure on singular cohomology (see eg. \cite{MR2553878}). However, it extends without difficulty to a field of arbitrary characteristic by instead using the extended N\'{e}ron-Severi groups, as recorded in \cite[Definition 3.4.6]{MR4184293}.} It is conjectured that this should always be the case. If $k=\mathbf{C}$ (or more generally if the characteristic of $k$ is zero) then this was shown in the untwisted case by Huybrechts--Macr\`{i}--Stellari \cite{MR2553878}. An alternative proof, which extends to the twisted case, was given by Reinecke \cite{MR3946279}. If $k$ has positive characteristic, various special cases were treated in \cite[Appendix B]{MR4184293}. Using a combination of standard techniques and Theorem \ref{thm:basic lifting}, we can complete the proof of this conjecture in arbitrary characteristic.

\begin{theorem}
  Let $(X,\alpha_{\Br})$ and $(Y,\beta_{\Br})$ be twisted K3 surfaces over an algebraically closed field $k$.
  If $\Phi_P:D(X,\alpha_{\Br})\xrightarrow{\sim} D(Y,\beta_{\Br})$ is a Fourier--Mukai equivalence, then the induced cohomological transform $\Phi_{v(P)}$ is orientation preserving. 
\end{theorem}
\begin{proof}
	As discussed above, if the characteristic of $k$ is 0 this is shown in \cite{MR2553878,MR3946279}. Suppose that the characteristic of $k$ is positive. Using now standard techniques introduced in \cite{LO15}, to prove the result it suffices to show that every twisted K3 surface admits a lift to characteristic 0, which reduces the problem to the case considered in \cite{MR2553878,MR3946279} (this strategy is outlined for instance in \cite[Appendix B]{MR4184293}). Thus, the result follows from Theorem \ref{thm:basic lifting}.
\end{proof}

\appendix

\section{Deformations of gerbes and flat cohomology classes}\label{appendix}


In this appendix we record some results on deformations of gerbes, particularly those banded by a possibly non-smooth group scheme.

Let $S$ be a scheme. Let $\mathbf{G}$ be a flat commutative group scheme over $S$ which is locally of finite presentation. The \textit{co-Lie complex} of $\mathbf{G}$ is
    \[
        \colie_{\mathbf{G}/S}:=\mathbf{L}e_{\mathbf{G}}^*L_{\mathbf{G}/S}\in D(\ms O_S)
    \]
    where $e_{\mathbf{G}}:S\to\mathbf{G}$ is the identity section and $L_{\mathbf{G}/S}$ is the cotangent complex of the morphism $\mathbf{G}\to S$. The \textit{Lie complex} of $\mathbf{G}$ is its derived dual
    \[
        \colie_{\mathbf{G}/S}^{\vee}:=R\sHom_{\ms O_S}(\colie_{\mathbf{G}/S},\ms O_S).
    \]
    The complex $\colie_{\mathbf{G}/S}$ is supported in degrees $[-1,0]$ and $\colie_{\mathbf{G}/S}^{\vee}$ is supported in degrees $[0,1]$ \cite[3.1.1.3]{MR0491681}. We set
    \[
        \mathfrak{t}_{\mathbf{G}/S}:=\ms H^0(\colie^{\vee}_{\mathbf{G}/S})\hspace{1cm}\mbox{and}\hspace{1cm}\mathfrak{n}_{\mathbf{G}/S}:=\ms H^1(\colie^{\vee}_{\mathbf{G}/S}).
    \]
    We may omit the base scheme $S$ from the notation if it is clear from context. We also will use a version of these definitions relative to a homomorphism $\mathbf{G}\to\mathbf{H}$ of flat commutative lfp group schemes: we define $\colie_{\mathbf{G}/\mathbf{H}}:=\mathbf{L}e_{\mathbf{G}}^*L_{\mathbf{G}/\mathbf{H}}$, we let $\colie^{\vee}_{\mathbf{G}/\mathbf{H}}$ be its derived dual, and we let $\mathfrak{t}_{\mathbf{G}/\mathbf{H}}$ and $\mathfrak{n}_{\mathbf{G}/\mathbf{H}}$ be the 0th and 1st cohomology sheaves of $\colie^{\vee}_{\mathbf{G}/\mathbf{H}}$.
    
    The co-Lie complex is contravariantly functorial with respect to maps of group schemes over $S$, and the Lie complex is covariantly functorial. We refer to Illusie \cite[3.1.1]{MR0491681} for further discussion of these objects and some of their basic properties (see also \cite[2.5.1]{MR0491682}).
    
    We mention two situations in which we can explicitly compute the Lie complex of $\mathbf{G}$. If $\mathbf{G}$ is smooth, then there is a canonical identification $\colie_{\mathbf{G}}=\omega_{\mathbf{G}}:=e^*_{\mathbf{G}}\Omega^1_{\mathbf{G}}$, and hence $\colie_{\mathbf{G}}^{\vee}=\omega^{\vee}_{\mathbf{G}}=\mathfrak{t}_{\mathbf{G}}$ is the Lie algebra of $\mathbf{G}$. Suppose that $\mathbf{G}$ is not necessarily smooth, and that we are given a closed immersion $\mathbf{G}\hookrightarrow\mathbf{H}$ where $\mathbf{H}$ is a smooth commutative group scheme over $S$. Let $J$ be the ideal sheaf of $\mathbf{G}$ in $\mathbf{H}$. The cotangent complex of $\mathbf{G}$ is then given by
  \[
    L_{\mathbf{G}/S}=[J/J^2\to\Omega^1_{\mathbf{H}/S}]
  \]
    where the terms on the right hand side are in degrees -1 and 0. We have $L_{\mathbf{G}/\mathbf{H}}=(J/J^2)[1]=N^{\vee}_{\mathbf{G}/\mathbf{H}}[1]$, where $N$ is the normal bundle of $\mathbf{G}$ in $\mathbf{H}$. It follows that $\colie^{\vee}_{\mathbf{G}/\mathbf{H}}=\mathfrak{n}_{\mathbf{G}/\mathbf{H}}=e_{\mathbf{G}}^*N_{\mathbf{G}/\mathbf{H}}$ and the Lie complex of $\mathbf{G}$ is given by
    \begin{equation}\label{eq:computation of lie complex}
        \colie_{\mathbf{G}}^{\vee}=[\mathfrak{t}_{\mathbf{H}}\to \mathfrak{n}_{\mathbf{G}/\mathbf{H}}]
    \end{equation}
    (terms in degrees 0 and 1).
    \begin{example}
        Let $n$ be a positive integer. The co-Lie complex of the group scheme $\mathbf{G}=\mu_n$ can be computed using the Kummer sequence
        \[
            1\to\mu_n\to\mathbf{G}_m\xrightarrow{\cdot n}\mathbf{G}_m\to 1.
        \]
        We have $\mathfrak{t}_{\mathbf{G}_m}\cong\ms O_X$, $\mathfrak{n}_{\mu_n/\mathbf{G}_m}\cong\ms O_X$, and $\colie^{\vee}_{\mu_n}\cong [\ms O_X\xrightarrow{\cdot n}\ms O_X]=\ms O_X(n)$.
    \end{example}

We now consider a closed immersion $i:S\hookrightarrow S'$ which is defined by a sheaf of ideals $I\subset\ms O_{S'}$ which satisfies $I^2=0$. We may regard $I$ also as a module over $\ms O_S$. Let $\mathbf{G}'$ be a flat commutative lfp group scheme on $S'$. Let $\mathbf{G}'\hookrightarrow\mathbf{H}'$ be an embedding into a smooth commutative group scheme $\mathbf{H}'$ on $S'$. We let $\mathbf{Q}'$ denote the quotient, so that we have a short exact sequence
\begin{equation}\label{eq:smooth resolution of G}
    1\to\mathbf{G}'\to\mathbf{H}'\to\mathbf{Q}'\to 1
\end{equation}
of flat commutative group schemes on $S'$. We write $\mathbf{G}=\mathbf{G}'|_S$, $\mathbf{H}=\mathbf{H}'|_S$, and $\mathbf{Q}=\mathbf{Q}'|_S$. By flat descent, $\mathbf{Q}'$ is also smooth. Consider the commutative diagram
\begin{equation}\label{eq:big diagram, which will give a qis}
    \begin{tikzcd}
        &&0\arrow{d}&0\arrow{d}&\\
        &&\mathfrak{t}_{\mathbf{H}}\otimes I\arrow{r}\arrow{d}&\mathfrak{t}_{\mathbf{Q}}\otimes I\arrow{d}&\\
        1\arrow{r}&\mathbf{G}'\arrow{r}\arrow{d}&\mathbf{H}'\arrow{r}\arrow{d}&\mathbf{Q}'\arrow{r}\arrow{d}&1\\
        1\arrow{r}&i_*\mathbf{G}\arrow{r}&i_*\mathbf{H}\arrow{r}\arrow{d}&i_*\mathbf{Q}\arrow{r}\arrow{d}&1\\
        &&1&1&
    \end{tikzcd}
\end{equation}
of sheaves on the big fppf site of $S'$, which has exact rows and columns. Incorperating the identification $\mathfrak{t}_{\mathbf{Q}}=\mathfrak{n}_{\mathbf{G}/\mathbf{H}}$, the diagram~\eqref{eq:big diagram, which will give a qis} gives a quasi-isomorphism
\begin{equation}\label{eq:Deligne's quasi-isomorphism}
    [\mathbf{G}'\to i_*\mathbf{G}]\cong\colie^{\vee}_{\mathbf{G}}\otimes^{\mathbf{L}}_{\ms O_X}I
\end{equation}
of complexes of fppf sheaves.

\subsection{Gerbes}\label{ssec:gerbes}

Let $X$ be an algebraic space. We will use the notation of a \textit{gerbe} over $X$ \cite[06QC]{stacks-project}. If $\mathbf{G}$ is a commutative group scheme over $X$, then a $\mathbf{G}$-gerbe is a gerbe $\ms X\to X$ equipped with an isomorphism $\mathbf{G}|_{\ms X}\iso\ms I_{\ms X}$. A morphism of $\mathbf{G}$-gerbes on $X$ is a morphism of algebraic stacks over $X$ which is compatible with these isomorphisms. The set of $\mathbf{G}$-gerbes over $X$ forms a 2-groupoid $\sGerb_X(\mathbf{G})$, and there is a natural bijection between the set of isomorphism classes of $\mathbf{G}$-gerbes over $X$ and the cohomology group $\H^2(X,\mathbf{G})$ \cite[Theorem 12.2.8]{MR3495343}.

An \textit{absolute gerbe} is an algebraic stack $\ms X$ which is locally nonempty and locally connected. By \cite[06QJ]{stacks-project}, $\ms X$ is an absolute gerbe if and only if the inertia stack $\ms I_{\ms X}\to \ms X$ is flat. If $\ms X$ is an absolute gerbe, then by \cite[06QD]{stacks-project} the sheafification $X=|\ms X|$ of $\ms X$ is an algebraic space, and the map $\ms X\to X$ makes $\ms X$ into a gerbe over $X$, in the sense of \cite[06QC]{stacks-project}.

Let $S$ be an algebraic space and let $\mathbf{G}$ be a commutative group scheme on $S$.  An \textit{absolute }$\mathbf{G}$\textit{-gerbe over} $S$ is an absolute gerbe $\ms X$ equipped with a morphism $\ms X\to S$ and an isomorphism $\mathbf{G}|_{\ms X}\iso \ms I_{\ms X}$. There is an induced factorization
\[
    \ms X\to X\to S
\]
where $X=|\ms X|$, and the map $\ms X\to X$ is a $\mathbf{G}|_{X}$-gerbe.
A morphism of absolute $\mathbf{G}$-gerbes over $S$ is a map of algebraic stacks over $S$ compatible with the given isomorphisms.

\subsection{Gerbes and torsors for two-term complexes}\label{ssec:gerbes and torsors for two-terms complexes}

Let $\mathbf{G}\to\mathbf{H}$ be a homomorphism of commutative group schemes on $X$, which we regard as a complex supported in degrees $[0,1]$. A \textit{torsor} for $[\mathbf{G}\to\mathbf{H}]$ is a $\mathbf{G}$-torsor $\ms T$ on $X$ equipped with a $\mathbf{G}\to\mathbf{H}$-equivariant map $\ms T\to\mathbf{H}$. Equivalently, a torsor for $[\mathbf{G}\to\mathbf{H}]$ consists of a $\mathbf{G}$-torsor $\ms T$ equipped with a trivialization of the induced $\mathbf{H}$-torsor $\ms T\wedge_{\mathbf{G}}\mathbf{H}$. The collection of torsors for $[\mathbf{G}\to\mathbf{H}]$ forms a groupoid which we denote by $\sTors_X([\mathbf{G}\to\mathbf{H}])$, and there is a natural bijection between the set of isomorphism classes of torsors for $[\mathbf{G}\to\mathbf{H}]$ and the cohomology group $\H^1(X,[\mathbf{G}\to\mathbf{H}])$.

There is a similar notion for gerbes. A \textit{gerbe} for $[\mathbf{G}\to\mathbf{H}]$ consists of a $\mathbf{G}$-gerbe $\ms X$ over $X$ and a $\mathbf{G}\to\mathbf{H}$-equivariant map $\ms X\to\mathrm{B}\mathbf{H}$ of gerbes, or equivalently a $\mathbf{G}$-gerbe $\ms X$ equipped with a trivialization of the induced $\mathbf{H}$-gerbe $\ms X\wedge_{\mathbf{G}}\mathbf{H}$. The collection of gerbes for $[\mathbf{G}\to\mathbf{H}]$ forms a 2-groupoid which we denote by $\sGerb_X([\mathbf{G}\to\mathbf{H}])$, and there is a natural bijection between the set of isomorphism classes of gerbes for $[\mathbf{G}\to\mathbf{H}]$ and the cohomology group $\H^2(X,[\mathbf{G}\to\mathbf{H}])$.

\subsection{Deformations of gerbes}\label{ssec:deformations of gerbes}

We consider the following deformation situation. Let $S$ be an algebraic space and let $i:S\hookrightarrow S'$ be a closed immersion whose defining sheaf of ideals $I$ is locally nilpotent. Let $\mathbf{G}'$ be a flat commutative lfp group scheme on $S'$ and write $\mathbf{G}=\mathbf{G}'|_S$. We assume that $\mathbf{G}'$ embeds in a smooth commutative group scheme over $S'$, so that we have a short exact sequence \eqref{eq:smooth resolution of G}. Let $\ms X$ be an absolute $\mathbf{G}$-gerbe over $S$. A \textit{deformation} of $\ms X$ over $S'$ is a 2-cartesian diagram
\[
    \begin{tikzcd}
        \ms X\arrow{d}\arrow[hook]{r}{\iota}&\ms X'\arrow{d}\\
        S\arrow[hook]{r}{i}&S'
    \end{tikzcd}
\]
where $\ms X'$ is an absolute $\mathbf{G}'$-gerbe flat over $S'$ and $\iota$ is a map of absolute $\mathbf{G}'$-gerbes. Equivalently, a deformation is a pair $(\ms X',\varphi)$ where $\ms X'$ is an absolute $\mathbf{G}'$-gerbe flat over $S'$ and $\varphi:\ms X'\times_{S'}S\iso \ms X$ is an isomorphism of absolute $\mathbf{G}$-gerbes over $S$. The collection of deformations of $\ms X$ over $S'$ has a natural structure of a 2-groupoid.

\begin{definition}
    We let $\sDef(\ms X/S')$ denote the 2-groupoid of deformations of $\ms X$ over $S'$. We let $\Def(\ms X/S')$ be the set of isomorphism classes of objects of $\sDef(\ms X/S')$.
\end{definition}

Let $X:=|\ms X|$ be the sheafification of $\ms X$, and consider the factorization
\[
    \ms X\to X\to S.
\]
Let $\ms X'$ be a deformation of $\ms X$ over $S'$, with sheafification $X'=|\ms X'|$. 
The map $\ms X'\to X'$ is faithfully flat. As $\ms X'\to S'$ is flat, so is $X'\to S'$. Thus, the association $(\ms X',\varphi)\mapsto(|\ms X'|,|\varphi|)$ defines a functor
\begin{equation}\label{eq:coarse space morphism of groupoids}
    \sDef(\ms X/S')\to\sDef(X/S').
\end{equation}
The homotopy fiber of this map over a deformation $(X',\rho)$ of $X$ is $\sDef(\ms X/X')$. We consider this latter 2-groupoid in more detail. To simplify notation, we omit pullbacks along $X\to S$. We wish to interpret the groupoid $\sDef(\ms X/X')$ cohomologically. To do this, we need to prove that deformations of $\ms X$ exist fppf locally on $X'$. This is implied by the following result. We define
\[
    \Phi^m(\mathbf{G}/X'):=\H^m(X',[\mathbf{G}'\to i_*\mathbf{G}])
\]

\begin{proposition}\label{prop:acyclicity for i, first statement} \
    \begin{enumerate}
        \item If $\alpha\in\H^m(X,\mathbf{G})$ is a flat cohomology class and $m\geq 1$, there exists an fppf cover $V\to X'$ such that the cover $V\times_{X'}X\to X$ kills $\alpha$.
        \item The natural map $\H^m(X',i_*\mathbf{G})\to\H^m(X,\mathbf{G})$ is an isomorphism for all $m\geq 0$.
    \end{enumerate}
\end{proposition}
\begin{proof}
    Consider a class $\alpha\in\H^m(X,\mathbf{G})$. Suppose that $m=1$. Let $T\to X$ be a $\mathbf{G}$-torsor with class $\alpha$. As $\mathbf{G}\to X$ is a syntomic cover, so is $T\to X$, and moreover we have $\alpha|_T=0$. By \cite[04E3]{stacks-project} there exists a syntomic cover $V\to X'$ such that $V\times_{X'}X\to X$ factors through $T\to X$. In particular, $V\times_{X'}X\to X$ kills $\alpha$. Suppose $m\geq 2$. We assume $\mathbf{G}$ embeds in a smooth group scheme, so we may find an \'{e}tale cover of $X$ which kills $\alpha$. We then conclude as before. This proves (1). It follows that the higher direct image $R^mi_*\mathbf{G}$ vanishes for all $m\geq 1$, which implies (2).
\end{proof}

Let $\ms T_{\mathbf{G}}$ denote the trivial $\mathbf{G}$-torsor over $X$.
\begin{proposition}\label{prop:acyclicity for i}
    We have canonical isomorphisms
    \[
        \sGerb_{X'}([\mathbf{G}'\to i_*\mathbf{G}])\iso\sDef(\mathrm{B}\mathbf{G}/X')
    \]
    and
    \[
        \sTors_{X'}([\mathbf{G}'\to i_*\mathbf{G}])\iso\sDef(\ms T_{\mathbf{G}}/X').
    \]
\end{proposition}
\begin{proof}
    We construct the first isomorphism. We have 2-functors
    \[
        \sGerb_{X'}(i_*\mathbf{G})\to\sGerb_X(i^{-1}i_*\mathbf{G})\to\sGerb_X(\mathbf{G})
    \]
    the first being pullback along $i$ and the second induced by the canonical map $i^{-1}i_*\mathbf{G}\to\mathbf{G}$. On isomorphism classes of objects (resp. isomorphism classes of 1-automorphisms, resp. isomorphism classes of 2-automorphisms) this composition corresponds to the natural map $\H^m(X',i_*\mathbf{G})\to\H^m(X,\mathbf{G})$ for $m=2$ (resp. $m=1$, resp. $m=0$). By Proposition \ref{prop:acyclicity for i, first statement}, these maps are isomorphisms for all $m\geq 0$, and so the composition is an equivalence of 2-groupoids. The second equivalence is constructed similarly.
\end{proof}

As a consequence of Proposition \ref{prop:acyclicity for i}, we have natural identifications $\Phi^2(\mathbf{G}/X')=\Def(\mathrm{B}\mathbf{G}/X')$ and $\Phi^1(\mathbf{G}/X')=\Def(\ms T_{\mathbf{G}}/X')$.

\begin{proposition}\label{prop:auts in terms of Phi} \
    \begin{enumerate}
        \item  The set of isomorphism classes of deformations of $\ms X$ over $X'$ is a pseudo-torsor under $\Phi^2(\mathbf{G}/X')$.
        \item The group of isomorphism classes of automorphisms of any deformation of $\ms X$ over $X'$ is $\Phi^1(\mathbf{G}/X')$.
        \item The group of 2-automorphisms of any 1-morphism of deformations of $\ms X$ over $X'$ is $\Phi^0(\mathbf{G}/X')$.
    \end{enumerate}
\end{proposition}
\begin{proof}
    Taking tensor products of gerbes defines an action of $\Def(\mathrm{B}\mathbf{G}/X')=\Phi^2(\mathbf{G}/X')$ on $\Def(\ms X/\ms Y')$. If the latter is nonempty, this action is simply transitive, giving (1). The groupoid of invertible self 1-morphisms of any object of $\sDef(\ms X/X')$ is equivalent $\sTors_{X'}([\mathbf{G}'\to i_*\mathbf{G}])$, which by Proposition \ref{prop:acyclicity for i} is equivalent to the groupoid of torsors for the 2-term complex $[\mathbf{G}'\to i_*\mathbf{G}]$, which implies $(2)$ and $(3)$.
\end{proof}

Suppose now that $I^2=0$. We describe a tangent--obstruction theory for the groupoid $\sDef(\ms X/X')$.

\begin{proposition}\label{prop:obstructions for gerbe deformation functor}
    Suppose that $I^2=0$, and let $\ms X$ be a $\mathbf{G}$-gerbe over $X$.
    \begin{enumerate}
        \item\label{item:property1} There exists a functorial class $o(\ms X/X')\in\H^3(X,\colie_{\mathbf{G}}^{\vee}\otimes_{\ms O_X}^{\mathbf{L}}I)$ whose vanishing is necessary and sufficient for the existence of a deformation of $\ms X$ over $X'$.
        \item\label{item:property2}  If $o(\ms X/X')=0$, then the set of isomorphism classes of deformations of $\ms X$ over $X'$ is a torsor under $\H^2(X,\colie_{\mathbf{G}}^{\vee}\otimes_{\ms O_X}^{\mathbf{L}}I)$.
        \item\label{item:property3}  The group of isomorphism classes of automorphisms of any deformation of $\ms X$ over $X'$ is $\H^1(X,\colie_{\mathbf{G}}^{\vee}\otimes_{\ms O_X}^{\mathbf{L}}I)$.
        \item\label{item:property4} The group of 2-automorphisms of any 1-morphism of deformations of $\ms X$ over $X'$ is $\H^0(X,\colie_{\mathbf{G}}^{\vee}\otimes_{\ms O_X}^{\mathbf{L}}I)$.
    \end{enumerate}
\end{proposition}
\begin{proof}
The quasi-isomorphism~\eqref{eq:Deligne's quasi-isomorphism} gives isomorphisms
\[
    \Phi^m(\mathbf{G}/X')=\H^m(X,\colie_{\mathbf{G}}^{\vee}\otimes_{\ms O_X}^{\mathbf{L}}I).
\]
Taking cohomology of the short exact sequence
\begin{equation}\label{eq:canonical SES of the def complex}
    1\to i_*\mathbf{G}[-1]\to [\mathbf{G}'\to i_*\mathbf{G}]\to\mathbf{G}'\to 1
\end{equation}
and using Proposition \ref{prop:acyclicity for i, first statement} we obtain a long exact sequence
\begin{equation}\label{eq:Long ES square zero situation}
    \ldots\to\H^1(X,\mathbf{G})\to\H^2(X,\colie_{\mathbf{G}}^{\vee}\otimes_{\ms O_X}^{\mathbf{L}}I)\to\H^2(X',\mathbf{G}')\to\H^2(X,\mathbf{G})\xrightarrow{\delta}\H^3(X,\colie_{\mathbf{G}}^{\vee}\otimes_{\ms O_X}^{\mathbf{L}}I)\to\ldots.
\end{equation}
We put $o(\ms X/X'):=\delta(\alpha)$. It is immediate that this class has the properties claimed in~\eqref{item:property1}. The remaining claims follow from Proposition \ref{prop:auts in terms of Phi}.
\end{proof}

Note that the obstruction class $o(\ms X/X')$ depends only on the cohomology class $\alpha=[\ms X]\in\H^2(X,\mathbf{G})$ of $\ms X$. We may sometimes write $o(\alpha/X')=o(\ms X/X')$.

\begin{remark}\label{rem:its a 2-gerbe}
  More conceptually, the assertions of Proposition \ref{prop:obstructions for gerbe deformation functor} may be summarized in the statement that the 2-stack $\suDef(\ms X/X')$ on the small \'{e}tale site of $X$ defined by $U\mapsto \sDef(\ms X_U/U')$ (where $U'$ is the unique \'{e}tale $X'$-scheme such that $U'\times_{X'}X=U$) is a 2-gerbe banded by the 2-term complex $[\mathfrak{t}_{\mathbf{H}}\to\mathfrak{n}_{\mathbf{G}/\mathbf{H}}]$.
\end{remark}

We compare deformations of $\ms X$ with deformations of its cohomology class $\alpha=[\ms X]\in\H^2(X,\mathbf{G})$. There is a map $\Def(\ms X/X')\to\H^2(X',\mathbf{G}')$ induced by $(\ms X',\varphi)\mapsto [\ms X']$. The image of this map is the set of classes $\alpha'\in\H^2(X',\mathbf{G}')$ such that $\alpha'|_X=\alpha$, which we denote by $\oDef(\alpha/X')$.
\begin{lemma}\label{lem:properties of varpi, non functor version}
  The map 
    \begin{equation}\label{eq:definition of varpi, non functor version}
        \Def(\ms X/X')\to\oDef(\alpha/X').
    \end{equation}
  is surjective. It is bijective if and only if the map $\H^1(X',\mathbf{G}')\to\H^1(X,\mathbf{G})$ is surjective.
\end{lemma}
\begin{proof}
There is a natural action of the group $\H^1(X,\mathbf{G})$ on the set $\Def(\ms X/X')$, which descends to a free action of the quotient $\H^1(X',\mathbf{G}')/\H^1(X,\mathbf{G})$. Furthermore, the quotient of $\Def(\ms X/X')$ by this action is exactly $\oDef(\alpha/X')$.
\end{proof}

We visualize this situation as an ``exact sequence''
\[
    \H^1(X',\mathbf{G}')\to\H^1(X,\mathbf{G})\rightsquigarrow\Def(\ms X/X')\to\H^2(X',\mathbf{G}')\xrightarrow{r-\alpha}\H^2(X,\mathbf{G})
\]
where the squiggly arrow denotes a group action and $r$ is the restriction map. If $\alpha=0$, this is just the long exact sequence on cohomology coming from~\eqref{eq:canonical SES of the def complex}.

\begin{remark}\label{rem:choice of gerbe}
    Given an isomorphism $\ms X\iso\ms Z$ of $\mathbf{G}$-gerbes on $X$, there is an induced equivalence $\sDef(\ms X/X')\iso\sDef(\ms Z/X')$. In particular, up to noncanonical isomorphism the 2-groupoid $\sDef(\ms X/X')$ and hence the set $\Def(\ms X/X')$ depends only on the cohomology class $[\ms X]\in\H^2(X,\mathbf{G})$ of the gerbe $\ms X$.
    
    If the map $\H^1(X',\mathbf{G}')\to\H^1(X,\mathbf{G})$ is surjective (eg. if $\H^1(X,\mathbf{G})=0$), then~\eqref{eq:definition of varpi, non functor version} is an isomorphism. Thus, in this case, there is a canonical isomorphism $\Def(\ms X/X')\cong\Def(\ms Z/X')$ for any $\ms X$ and $\ms Z$ with class $\alpha$.
\end{remark}

\subsection{Deformations relative to an embedding into a smooth group scheme}\label{ssec:change of group}

We will consider a relative deformation problem with respect to the embedding $\mathbf{G}\hookrightarrow\mathbf{H}$. Let $\ms Y=\ms X\wedge_{\mathbf{G}}\mathbf{H}$ be the $\mathbf{H}$-gerbe associated to $\ms X$. Fix a deformation $(\ms Y',\psi)$ of $\ms Y$ over $X'$. We let $\sDef(\ms X/\ms Y')$ be the groupoid whose objects are 2-cartesian diagrams
\begin{equation}\label{eq:cartesian diagram for relative def problem}
    \begin{tikzcd}
        \ms X\arrow{d}[swap]{g}\arrow[hook]{r}{\iota_{\ms X}}&\ms X'\arrow{d}{g'}\\
        \ms Y\arrow[hook]{r}{\iota_{\ms Y}}&\ms Y'
    \end{tikzcd}
\end{equation}
where $\ms X'$ is an absolute $\mathbf{G}'$-gerbe, $\iota_{\ms X}$ is a map of absolute $\mathbf{G}'$-gerbes, and $g'$ is a $\mathbf{G}'\to\mathbf{H}'$-equivariant map of gerbes. Because $\mathbf{G}'\to\mathbf{H}'$ is a monomorphism, the maps $g$ and $g'$ are representable. It follows that $\sDef(\ms X/\ms Y')$ has a natural structure of a groupoid. We let $\suDef(\ms X/\ms Y')$ denote the stack on the small \'{e}tale site of $X$ defined by
\[
    U\mapsto \sDef(\ms X_U/\ms Y'_{U'})
\]
where $U'\to X'$ is the unique \'{e}tale morphism such that $U'\times_{X'}X=U$.

\begin{lemma}\label{lem:reduction of structure group}
    There is a canonical equivalence of stacks
    \[
        \suDef(\ms T_\mathbf{Q}/X')\cong\suDef(\mathrm{B}\mathbf{G}/\mathrm{B}\mathbf{H}').
    \]
\end{lemma}
\begin{proof}
    There is a canonical equivalence between the groupoid of $\mathbf{Q}'$-torsors on $X'$ and the groupoid whose objects are $\mathbf{G}'$-gerbes on $X'$ equipped with a map of gerbes to $\mathrm{B}\mathbf{H}'$. Explicitly, given a $\mathbf{Q}'$-torsor $\ms T'$, we may consider the associated $\mathbf{G}'$-gerbe of trivializations $\ms X(\ms T')$, which is equipped with a map $\ms X(\ms T')\to\mathrm{B}\mathbf{H}'$ of gerbes. Conversely, given a $\mathbf{G}'$-gerbe $\ms X'$ and a map $\ms X'\to\mathrm{B}\mathbf{H}'$ of gerbes, the restriction of $\ms X'$ along the canonical section of $\mathrm{B}\mathbf{H}'\to X'$ is a $\mathbf{Q}'$-torsor on $X'$. These equivalences are compatible with restriction to $X$, and induce the desired equivalence of groupoids. They are furthermore compatible with \'{e}tale localization, and we obtain the claimed equivalence of stacks.
\end{proof}

\begin{proposition}\label{prop:obstructions for change of group map, normal bundle version}
    The stack $\suDef(\ms X/\ms Y')$ is a gerbe on the small \'{e}tale site of $X$ which is canonically banded by $\mathfrak{n}_{\mathbf{G}/\mathbf{H}}\otimes I$.
\end{proposition}
\begin{proof}
    We first show that $\suDef(\ms X/\ms Y')$ is locally nonempty and locally connected. As $\mathbf{H}'$ and $\mathbf{Q}'$ are smooth, we may find an \'{e}tale cover of $X'$ which trivializes both $\ms X$ and $\ms Y'$. It therefore suffices to show the claim for $\suDef(\mathrm{B}\mathbf{G}/\mathrm{B}\mathbf{H}')$, which by Lemma \ref{lem:reduction of structure group} is isomorphic to the stack $\suDef(\ms T_\mathbf{Q}/X')$. The group $\Def(\ms T_\mathbf{Q}/X')$ is nonempty, because it contains the trivial deformation, and is isomorphic to $\Phi^1(\mathbf{Q}/X')$. We have an exact sequence
    \[
        \H^0(X',\mathbf{Q}')\to\H^0(X,\mathbf{Q})\to\Phi^1(\mathbf{Q}/X')\to\H^1(X',\mathbf{Q}').
    \]
    Because $\mathbf{Q}'$ is smooth, the map $\mathbf{Q}'\to i_*\mathbf{Q}$ is surjective in the \'{e}tale topology. Furthermore, we may trivialize classes in $\H^1(X',\mathbf{Q}')$ by \'{e}tale covers. It follows that any two elements of $\Phi^1$ are locally equal to 0.

    We now construct the banding. Given a 2-groupoid $\ms G$ and an object $x\in\ms G$, we let $\sAut(x)$ denote the groupoid whose objects are 1-morphisms in $\ms G$ from $x$ to itself and whose morphisms are 2-morphisms in $\ms G$. Fix an object $(\ms X',g',\iota_{\ms X})$ of $\sDef(\ms X/\ms Y')$. The map $g'$ factors uniquely through an isomorphism $\ms X'\wedge_{\mathbf{G}'}\mathbf{H}'\iso\ms Y'$. Conjugating by this isomorphism, we obtain a map $\sAut(\ms X',\iota_{\ms X})\to\sAut(\ms Y',\iota_{\ms Y})$. The homotopy fiber of this map over the identity is equivalent to the group $\Aut(\ms X',g',\iota_{\ms X})$.
    

    The groupoid $\sAut(\ms X',\iota_{\ms X})$ may be realized as the groupoid $\sDef(\ms T_\mathbf{G}/X')$ of deformations of the trivial $\mathbf{G}$-torsor over $X'$, which in turn is isomorphic to the groupoid of torsors for the 2-term complex $[\mathbf{G}'\to i_*\mathbf{G}]$. Via the quasi-isomorphism~\eqref{eq:Deligne's quasi-isomorphism} induced by the diagram~\eqref{eq:big diagram, which will give a qis}, this is equivalent to the groupoid of torsors for $[\mathfrak{t}_{\mathbf{H}}\otimes I\to\mathfrak{n}_{\mathbf{G}/\mathbf{H}}\otimes I]$. As $\mathbf{H}$ is smooth, the groupoid $\Aut(\ms Y',\iota_{\ms Y})$ is equivalent to the groupoid of torsors for $\mathfrak{t}_{\mathbf{H}}\otimes I$. These maps fit into a commutative diagram
    \[
        \begin{tikzcd}
            \sAut(\ms X',\iota_{\ms X})\arrow{r}{\sim}\arrow{d}&\sTors_{X'}([\mathbf{G}'\to i_*\mathbf{G}])\arrow{r}{\sim}\arrow{d}&\sTors_X([\mathfrak{t}_{\mathbf{H}}\otimes I\to\mathfrak{n}_{\mathbf{G}/\mathbf{H}}\otimes I])\arrow{d}\\
            \sAut(\ms Y',\iota_{\ms Y})\arrow{r}{\sim}&\sTors_{X'}([\mathbf{H}'\to i_*\mathbf{H}])\arrow{r}{\sim}&\sTors_X(\mathfrak{t}_{\mathbf{H}}\otimes I)
        \end{tikzcd}
    \]
    in the homotopy category.
    Under the bottom composition, the identity map is sent to the trivial $\mathfrak{t}_{\mathbf{H}}\otimes I$-torsor. The homotopy fiber of the right vertical arrow over the trivial $\mathfrak{t}_{\mathbf{H}}\otimes I$-torsor is equivalent to the category of trivializations of the trivial $\mathfrak{n}_{\mathbf{G}/\mathbf{H}}\otimes I$-torsor, which is exactly $\Gamma(X,\mathfrak{n}_{\mathbf{G}/\mathbf{H}}\otimes I)$.
\end{proof}

As an immediate consequence, we obtain the following result.

\begin{proposition}\label{prop:obstructions for relative gerbe change of group situation}
Let $\ms X$ be a $\mathbf{G}$-gerbe over $X$, let $\ms Y=\ms X\wedge_{\mathbf{G}}\mathbf{H}$ be the associated $\mathbf{H}$-gerbe, and let $(\ms Y',\psi)$ be a deformation of $\ms Y$ over $X'$.
\begin{enumerate}
        \item There exists a functorial class $o(\ms X/\ms Y')\in\H^2(X,\mathfrak{n}_{\mathbf{G}/\mathbf{H}}\otimes I)$ whose vanishing is necessary and sufficient for the existence of a deformation of $\ms X$ over $X'$ which fits into a 2-cartesian diagram~\eqref{eq:cartesian diagram for relative def problem}.
        \item If $o(\ms X/\ms Y')=0$, then the set of isomorphism classes of such diagrams is a torsor under $\H^1(X,\mathfrak{n}_{\mathbf{G}/\mathbf{H}}\otimes I)$.
        \item The group of automorphisms of any such diagram is $\H^0(X,\mathfrak{n}_{\mathbf{G}/\mathbf{H}}\otimes I)$.
    \end{enumerate}
\end{proposition}
\begin{proof}
    We define $o(\ms X/\ms Y')$ to be the class $[\suDef(\ms X/\ms Y')]\in\H^2(X,\mathfrak{n}_{\mathbf{G}/\mathbf{H}}\otimes I)$ of the gerbe $\suDef(\ms X/\ms Y')$. This class vanishes if and only if there exists a global section of $\suDef(\ms X/\ms Y')$, or equivalently if and only if $\sDef(\ms X/\ms Y')$ is nonempty. As $\suDef(\ms X/\ms Y')$ is a $\mathfrak{n}_{\mathbf{G}/\mathbf{H}}\otimes I$-gerbe, if $o(\ms X/\ms Y')=0$ then the set of its global sections is a torsor under $\H^1(X,\mathfrak{n}\otimes I)$, and the automorphism group of any section is identified with $\H^0(X,\mathfrak{n}\otimes I)$.
\end{proof}


\begin{remark}\label{rem:functoriality of obstruction class, take 2}
Consider the short exact sequence
\[
    0\to \mathfrak{n}_{\mathbf{G}/\mathbf{H}}[-1]\to [\mathfrak{t}_{\mathbf{H}}\to\mathfrak{n}_{\mathbf{G}/\mathbf{H}}]\to\mathfrak{t}_{\mathbf{H}}\to 0
  \]
of complexes. We tensor with $I$ and take $\H^3$ to obtain an exact sequence
  \[
    \H^2(X,\mathfrak{n}_{\mathbf{G}/\mathbf{H}}\otimes I)\to\H^3(X,[\mathfrak{t}_{\mathbf{H}}\otimes I\to\mathfrak{n}_{\mathbf{G}/\mathbf{H}}\otimes I])\to\H^3(X,\mathfrak{t}_{\mathbf{H}}\otimes I).
  \]
If $\ms X$ is a $\mathbf{G}$-gerbe over $X$ and $\ms Y$ is the induced $\mathbf{H}$-gerbe, then the right map satisfies $o(\ms X/X')\mapsto o(\ms Y/X')$. If $o(\ms X/X')=0$ and $\ms Y'$ is a fixed deformation of $\ms Y$ over $X'$, then the left hand map satisfies $o(\ms X/\ms Y')\mapsto o(\ms X/X')$.
\end{remark}


\subsection{Deformation functors and prorepresentability}

We now consider deformations of a cohomology class or a gerbe simultaneously over various infinitesimal thickenings. We will follow as much as possible the terminology of the Stacks Project \cite[06G7]{stacks-project}. Let $(\Lambda,\mathfrak{m}_{\Lambda})$ be a noetherian local ring with residue field $k$. Let $\cC_{\Lambda}$ be the category whose objects are Artinian local $\Lambda$-algebras $A$ such that the map $\Lambda\to A$ is local and induces an isomorphism on residue fields. Via this isomorphism, we identify the residue field of any object $A\in\cC_\Lambda$ with $k$. A morphism in $\cC_\Lambda$ is a homomorphism of $\Lambda$-algebras.



Let $S$ be a $k$-scheme. Let $S_\Lambda$ be a flat formal $\Lambda$-scheme equipped with an isomorphism $S_\Lambda\otimes_\Lambda k\cong S$ of $k$-schemes. Let $\mathbf{G}_{\Lambda}$ be a flat commutative formal group scheme on $S_\Lambda$. We assume that $\mathbf{G}_{\Lambda}$ admits an embedding into a smooth formal commutative group scheme $\mathbf{H}_{\Lambda}$ on $S_{\Lambda}$. Given $A\in\cC_{\Lambda}$, we write $S_A:=S_\Lambda\otimes_\Lambda A$ for the base change of $S_\Lambda$ to $A$, and set $\mathbf{G}_{A}:=\mathbf{G}_{\Lambda}|_{S_A}$. We also write $\mathbf{G}:=\mathbf{G}_k$ for the restriction of $\mathbf{G}_{\Lambda}$ to $S$. Let $\ms X$ be an absolute $\mathbf{G}$-gerbe flat over $S$ and let $X=|\ms X|$ be its sheafification. We have a factorization
\[
    \ms X\to X\to S
\]
where the first map is a $\mathbf{G}_X$-gerbe over $X$ and the second is flat. We will assume that $X\to S$ is moreover smooth.


\begin{definition}\label{def:definition of def functor, most general version}
    We let $\sDef_{\ms X/S_\Lambda}$ be the 2-category cofibered in 2-groupoids over $\cC_{\Lambda}$ whose fiber over $A\in\cC_{\Lambda}$ is the 2-groupoid $\sDef(\ms X/S_A)$.
    We let $\Def_{\ms X/S_\Lambda}$ be the functor on $\cC_\Lambda$ whose value on $A\in\cC_{\Lambda}$ is the set $\Def(\ms X/S_A)$ of isomorphism classes of objects of $\sDef(\ms X/S_A)$.
\end{definition}


Let $\ms Y:=\ms X\wedge_{\mathbf{G}}\mathbf{H}$ be the induced absolute $\mathbf{H}$-gerbe. We have a diagram
\begin{equation}\label{eq:commutative diagram of functors}
    \begin{tikzcd}[column sep=small]
        \sDef_{\ms X/S_{\Lambda}}\arrow{rr}{\iota}\arrow{dr}[swap]{\pi_{\mathbf{G}}}&&\sDef_{\ms Y/S_{\Lambda}}\arrow{dl}{\pi_{\mathbf{H}}}\\
    &\sDef_{X/S_{\Lambda}}.&
    \end{tikzcd}
\end{equation}
Here, the map $\pi_{\mathbf{G}}$ is given by $(\ms X_A,\varphi)\mapsto (|\ms X_A|,|\varphi|)$~\eqref{eq:coarse space morphism of groupoids} (and similarly for $\pi_\mathbf{H}$), and $\iota$ is given by $(\ms X_A,\varphi)\mapsto (\ms X_A\wedge_{\mathbf{G}_A}\mathbf{H}_A,\varphi\wedge_{\mathbf{G}}\mathbf{H})$. 
The homotopy fiber of $\pi_{\mathbf{G}}(A)$ over a deformation $(X_A,\rho)$ of $X$ over $S_A$ is the 2-groupoid $\sDef(\ms X/X_A)$. The homotopy fiber of $\iota(A)$ over a deformation $(\ms Y_A,\psi)$ of $\ms Y$ over $S_A$ is the groupoid $\sDef(\ms X/\ms Y_A)$.

We consider deformations over the dual numbers $k[\varepsilon]$. The \textit{tangent space} to $\sDef_{\ms X/S_\Lambda}$ is $T(\sDef_{\ms X/S_\Lambda}):=\Def_{\ms X/S_{\Lambda}}(k[\varepsilon])=\Def(\ms X/S[\varepsilon])$. We let $\Inf^{-1}(\sDef_{\ms X/S_\Lambda})$ be the group of isomorphism classes of 1-automorphisms of the trivial deformation of $\ms X$ over $S[\varepsilon]$, and we let $\Inf^{-2}(\sDef_{\ms X/S_\Lambda})$ be the group of 2-automorphisms of the identity map of the trivial deformation. The following result might be compared to \cite[06L5]{stacks-project}.

\begin{lemma}\label{lem:LES on infinitesimal auts and such}
    There is a canonical isomorphism
    \[
        \H^0(X,\colie^{\vee}_{\mathbf{G}})\iso\Inf^{-2}(\sDef_{\ms X/S_\Lambda})
    \]
    and exact sequence
    \[
    \begin{tikzcd}
  0 \rar & \H^1(X,\colie^{\vee}_{\mathbf{G}}) \rar & \Inf^{-1}(\sDef_{\ms X/S_\Lambda}) \rar
             \ar[draw=none]{d}[name=X, anchor=center]{}
    & \H^0(X,T_{X/S}) \ar[rounded corners,
            to path={ -- ([xshift=2ex]\tikztostart.east)
                      |- (X.center) \tikztonodes
                      -| ([xshift=-2ex]\tikztotarget.west)
                      -- (\tikztotarget)}]{dll} &\\      
  &\H^2(X,\colie^{\vee}_{\mathbf{G}}) \rar & T(\sDef_{\ms X/S_\Lambda}) \rar & \H^1(X,T_{X/S}) \arrow{r}{\varepsilon^{-1}\mathrm{ob}} & \H^3(X,\colie^{\vee}_{\mathbf{G}}).
\end{tikzcd}
    \]
\end{lemma}
\begin{proof}
    The homotopy fiber of the map
    \[
        \sDef(\ms X/S[\varepsilon])\to\sDef(X/S[\varepsilon])
    \]
    over the trivial deformation $X[\varepsilon]$ of $X$ over $S[\varepsilon]$ is exactly $\sDef(\ms X/X[\varepsilon])$. By Proposition \ref{prop:obstructions for gerbe deformation functor}, the tangent space, group of infinitesimal 1-automorphisms, and group of infinitesimal 2-automorphisms of $\sDef_{\ms X/X[\varepsilon]}$ are, respectively, $\H^2(X,\colie^{\vee}_{\mathbf{G}})$, $\H^1(X,\colie^{\vee}_{\mathbf{G}})$, and $\H^0(X,\colie^{\vee}_{\mathbf{G}})$. The tangent space and group of infinitesimal 1-automorphisms of $\sDef_{X/S[\varepsilon]}$ are respectively $\H^1(X,T_{X/S})$ and $\H^0(X,T_{X/S})$. With these identifications, the construction of the exact sequence is exactly as in \cite[06L5]{stacks-project}. The isomorphism follows from the fact that $\sDef(X/S[\varepsilon])$ is 1-truncated.
\end{proof}

We prove the following result about the existence of certain pushouts of gerbes.

\begin{proposition}\label{prop:pushouts of gerbes}
    Let $T$ be a scheme. Let $\mathbf{G}$ be a flat commutative lfp group scheme over $T$. Given maps $\ms Z\xleftarrow{j}\ms X\xrightarrow{i}\ms X$ of absolute $\mathbf{G}$-gerbes over $T$ where $i$ is a nilpotent closed immersion and $j$ is affine, there exists a 2-commutative diagram
    \begin{equation}\label{eq:square on the gerbes}
        \begin{tikzcd}
            \ms X\arrow[hook]{r}{i}\arrow{d}[swap]{j}&\ms Y\arrow{d}{j'}\\
            \ms Z\arrow[hook]{r}{i'}&\ms W
        \end{tikzcd}
     \end{equation}
     of absolute $\mathbf{G}$-gerbes over $T$ which is both a 2-pullback and a 2-pushout diagram in the 2-category of $\mathbf{G}$-gerbes over $T$.
\end{proposition}
\begin{proof}
    By \cite[Proposition A.2]{MR3589351}, we find an algebraic stack $\ms W$ and a 2-commutative diagram~\eqref{eq:square on the gerbes} of algebraic stacks over $T$ which is 2-cartesian and 2-cocartesian. We will show that $\ms W$ has a canonical structure of a $\mathbf{G}$-gerbe. Let $X,Y,Z$ be the sheafifications of $\ms X$,$\ms Y$, and $\ms Z$. The induced map $X\to Y$ is a nilpotent closed immersion, and the induced map $X\to Z$ is affine. Let $W$ be the pushout of $Z\leftarrow X\to Y$. We have a 2-commutative diagram
    \[
        \begin{tikzcd}[row sep=1.5em, column sep = 1.5em]
            \ms X \arrow[rr,hook,"i"] \arrow[dr,"j"] \arrow[dd,swap] &&
            \ms Y \arrow[dd] \arrow[dr,"j'"] \\
            & \ms Z \arrow[hook]{rr}[xshift=-2ex]{i'} \arrow[dd] &&
            \ms W \arrow[dd,dashed] \\
            X\arrow[rr,hook] \arrow[dr] && Y \arrow[dr] \\
            & Z \arrow[rr,hook] && W.
    \end{tikzcd}
    \]
    As the top face is 2-cocartesian, there is an induced map $\ms W\to W$. Using that the bottom face is a pushout, it follows that $\ms W\to W$ is initial with respect to maps from $\ms W$ to sheaves. That is, $W$ is the sheafification of $\ms W$. As $W$ is an algebraic space, $\ms W$ is a gerbe.
    
    Consider the 2-commutative diagram
    \begin{equation}\label{eq:square on inertia}
        \begin{tikzcd}
            \ms I_{\ms X}\arrow{r}\arrow{d}&\ms I_{\ms Y}\arrow{d}\\
            \ms I_{\ms Z}\arrow{r}&\ms I_{\ms W}.
        \end{tikzcd}
    \end{equation}
    It follows from \cite[06R5]{stacks-project} that~\eqref{eq:square on inertia} is 2-cartesian. As $\ms W$ is a gerbe, the map $\ms I_{\ms W}\to \ms W$ is flat \cite[06QJ]{stacks-project}. It follows from \cite[07W3]{stacks-project} that~\eqref{eq:square on inertia} is also 2-cocartesian.
    (this reference refers only to the categories of algebraic spaces over a pushout diagram of algebraic spaces. The result extends immediately to the categories of spaces over a pushout diagram of algebraic stacks).
    The given identifications of the inertia of $\ms X$,$\ms Y$, and $\ms Z$ with the respective pullbacks of $\mathbf{G}$ induce an isomorphism between $\ms I_{\ms W}$ and the 2-pushout of $\mathbf{G}|_{\ms Z}\leftarrow\mathbf{G}|_{\ms X}\to\mathbf{G}|_{\ms Y}$, which is $\mathbf{G}|_{\ms W}$. This gives $\ms W$ the structure of an absolute $\mathbf{G}$-gerbe over $T$. Moreover, the maps $i'$ and $j'$ are maps of absolute $\mathbf{G}$-gerbes. By construction, the diagram~\eqref{eq:square on the gerbes} is a 2-pullback and a 2-pushout in the 2-category of algebraic stacks over $T$. Because~\eqref{eq:square on inertia} is also a 2-pullback and a 2-pushout, it follows that~\eqref{eq:square on the gerbes} is a 2-pullback and a 2-pushout in the category of absolute $\mathbf{G}$-gerbes over $T$.
\end{proof}


We recall the statement of Schlessinger's theorem, as formulated in the stacks project \cite[06JM]{stacks-project} (the original reference being \cite[Theorem 2.11]{MR217093}). Let $F$ be a covariant functor on $\cC_{\Lambda}$. Following the terminology of the Stack project \cite[06G7]{stacks-project}, we say that $F$ is a \textit{predeformation functor} if $F(k)$ is a singleton. We say that $F$ \textit{satisfies (RS)} if for any surjection $B\twoheadrightarrow A$ and any morphism $C\to A$ in $\cC_{\Lambda}$, the natural map
\begin{equation}\label{eq:RS maps}
    F(B\times_AC)\to F(B)\times_{F(A)}F(C)
\end{equation}
is bijective (see \cite[06J2]{stacks-project}).
If $F$ is a predeformation functor and satisfies (RS), then the set $T(F):=F(k[\varepsilon])$ (the tangent space of $F$) has a natural $k$-vector space structure, where $k[\varepsilon]$ has the $\Lambda$-algebra structure given by $\Lambda\twoheadrightarrow k\hookrightarrow k[\varepsilon]$. Schlessinger's theorem states that a predeformation functor $F$ is prorepresentable if and only if it satisfies (RS) and has finite dimensional tangent space.


\begin{proposition}\label{prop:its a deformation category}
    For any $C\to A\twoheadleftarrow B$ as above, the map~\eqref{eq:RS maps} with $F=\Def_{\ms X/S_{\Lambda}}$ is surjective. If for any object $(\ms X_B,\varphi)$ of $\sDef_{\ms X/S_{\Lambda}}(B)$ the map $\Aut_B(\ms X_B,\varphi)\to\Aut_A(\ms X_A,\varphi)$ is surjective, then~\eqref{eq:RS maps} is bijective.
\end{proposition}
\begin{proof}
    Consider a diagram
    \[
        \begin{tikzcd}
            \ms X_C\arrow{d}&\ms X_A\arrow{l}\arrow[hook]{r}\arrow{d}&\ms X_B\arrow{d}\\
            S_C&S_A\arrow{l}\arrow[hook]{r}&S_B
        \end{tikzcd}
    \]
    where the top row are maps of $\mathbf{G}_{\Lambda}$-gerbes over $S_\Lambda$, the vertical arrows are flat, and the squares are 2-cartesian. We let $\ms X_{B\times_AC}$ be the 2-pushout of $\ms X_C\leftarrow\ms X_A\to\ms X_B$ as in Proposition \ref{prop:pushouts of gerbes}, applied with $T=S_{B\times_AC}$ and $\mathbf{G}=\mathbf{G}_{B\times_AC}$. This is an absolute $\mathbf{G}_{B\times_AC}$-gerbe over $S_{B\times_AC}$. By \cite[Lemma A.4]{MR3589351}, $\ms X_{B\times_AC}$ is flat over $S_{B\times_AC}$. This implies surjectivity of~\eqref{eq:RS maps}.
    
    We now consider the injectivity. Suppose that $(\ms X_{B\times_AC},\varphi)$ and $(\ms Y_{B\times_AC},\psi)$ are two objects of $\sDef(\ms X/S_{B\times_AC})$ whose images in $\sDef(\ms X/S_B)$ and in $\sDef(\ms X/S_C)$ are isomorphic, via isomorphisms say $g_B:\ms X_B\to\ms Y_B$ and $f_C:\ms X_C\to\ms Y_C$. By assumption, we may lift the automorphism $g_A^{-1}\circ f_A$ of $(\ms X_A,\varphi)$ to an automorphism $\theta_B$ of $(\ms X_B,\varphi)$. We obtain a 2-commutative diagram
    \[
        \begin{tikzcd}[row sep=1.5em, column sep = 1.5em]
            \ms X_A \arrow[rr,hook] \arrow[dr] \arrow[dd,swap,"f_A"] &&
            \ms X_B \arrow{dd}[yshift=-2ex]{g_B\circ\theta_B} \arrow[dr] \\
            & \ms X_C \arrow[hook]{rr}\arrow{dd}[yshift=-2ex]{f_C} &&
            \ms X_{B\times_AC}\arrow[dd,dashed] \\
            \ms Y_A\arrow[rr,hook] \arrow[dr] && \ms Y_B \arrow[dr] \\
            & \ms Y_C \arrow[rr,hook] && \ms Y_{B\times_AC}.
        \end{tikzcd}
    \]
    Being flat base changes of the pushout of $\Spec C\leftarrow\Spec A\to\Spec B$, the top and bottom faces are 2-cocartesian. We therefore find a dashed isomorphism making the diagram 2-commutative. Moreover, by our choice of $f_C$, the composition $\ms X\hookrightarrow\ms X_C\xrightarrow{f_C}\ms Y_C$ is isomorphic to $\ms X\hookrightarrow\ms Y_C$. It follows that the isomorphism $\ms X_{B\times_AC}\iso\ms Y_{B\times_AC}$ is compatible with the maps from $\ms X$ up to 2-isomorphism, and hence the deformations $(\ms X_{B\times_AC},\varphi)$ and $(\ms Y_{B\times_AC},\psi)$ are isomorphic in $\sDef(\ms X/S_{B\times_AC})$.
\end{proof}


We now give some conditions implying prorepresentability of the functors $\Def_{\ms X/S_{\Lambda}}$. We first consider the extremal case when $X\to S$ is an isomorphism. For ease of notation, we identify the two, and put $X_{\Lambda}=S_{\Lambda}$.
The groups $\Phi^m(\mathbf{G}/X_A)$ are covariantly functorial with respect to maps in $\cC_{\Lambda}$.
We let $\Phi^m_{\mathbf{G}/X_{\Lambda}}$ denote the covariant functor on $\cC_{\Lambda}$ defined by
\begin{equation}\label{eq:def of Phi}
    \Phi^m_{\mathbf{G}/X_{\Lambda}}(A):=\Phi^m(\mathbf{G}/X_A)=\H^m(X_A,[\mathbf{G}_A\to i_{A*}\mathbf{G}])
\end{equation}
where $i_A:X\hookrightarrow X_A$ is the inclusion.
In particular, we have $\Phi^2_{\mathbf{G}/X_{\Lambda}}=\Def_{\mathrm{B}\mathbf{G}/X_{\Lambda}}$.
\begin{theorem}\label{thm:relative deformations of the gerbe}
    Let $X$ be a proper $k$-scheme and let $X_{\Lambda}$ be a flat formal scheme over $\Lambda$ equipped with an isomorphism $X_{\Lambda}\otimes_{\Lambda}k\iso X$. If the functor $\Phi^1_{\mathbf{G}/X_{\Lambda}}$ on $\cC_{\Lambda}$ is formally smooth, then for any $\mathbf{G}$-gerbe $\ms X$ over $X$ the functor $\Def_{\ms X/X_{\Lambda}}$ is prorepresentable.
\end{theorem}
\begin{proof}
    We use Schlessinger's theorem \cite[06JM]{stacks-project}. We apply Lemma \ref{lem:LES on infinitesimal auts and such} with $X_{\Lambda}=S_{\Lambda}$ (thus, $T_{X/S}=0$). As $X$ is assumed to be proper, we conclude that the tangent space to $\Def_{\ms X/X_\Lambda}$ is finite dimensional over $k$. To verify (RS), we check the condition of Proposition \ref{prop:its a deformation category}. Given a surjection $B\twoheadrightarrow A$ in $\cC_{\Lambda}$ and an object $(\ms X_B,\varphi)$ of $\sDef(\ms X/X_B)$, by Proposition \ref{prop:auts in terms of Phi} the restriction map on automorphism groups is identified with the map $\Phi^1(\mathbf{G}/X_B)\to\Phi^1(\mathbf{G}/X_A)$. By assumption, $\Phi^1_{\mathbf{G}/X_{\Lambda}}$ is formally smooth, so this map is surjective.
\end{proof}

    
    

We now consider the case when $S_{\Lambda}=\Spf\Lambda$. We write $\sDef_{\ms X/\Lambda}:=\sDef_{\ms X/\Spf\Lambda}$.

\begin{theorem}\label{thm:prorep criterion the second}
    Let $X$ be a smooth proper $k$-scheme. Suppose that $\H^0(X,T_X)=0$ and that, for any deformation $(X_A,\rho)$ of $X$ over $A\in\cC_{\Lambda}$, the functor $\Phi^1_{\mathbf{G}/X_A}$ on $\cC_A$ is formally smooth. For any $\mathbf{G}$-gerbe $\ms X$ over $X$, the functor $\Def_{\ms X/\Lambda}$ is prorepresentable.
\end{theorem}
\begin{proof}
    It follows from Lemma \ref{lem:LES on infinitesimal auts and such} and our assumption that $X$ is proper that the tangent space to $\Def_{\ms X/\Lambda}$ is finite dimensional over $k$. As before, we check (RS) using Proposition \ref{prop:its a deformation category}. For any object $(\ms X_A,\varphi)$ of $\sDef(\ms X/A)$, we have an exact sequence
    \[
        0\to\Phi^1(\mathbf{G}/X_A)\to\Aut_A(\ms X_A,\varphi)\to\Aut_A(X_A,\rho)
    \]
    where $(X_A,\rho)=(|\ms X_A|,|\varphi|)$. Our assumption that $\H^0(X,T_X)=0$ implies that $\Aut_A(X_A,\rho)=0$ (see eg. the proof of \cite[Corollary 18.3]{MR2583634}), and so $\Phi^1(\mathbf{G}/X_A)\cong\Aut_A(\ms X_A,\varphi)$. Given a surjection $B\twoheadrightarrow A$ in $\cC_{\Lambda}$ and an object $(\ms X_B,\varphi)$ of $\sDef(\ms X/B)$, our assumption that $\Phi^1_{\mathbf{G}/X_B}$ is formally smooth (where $X_B=|\ms X_B|$) implies that $\Phi^1(\mathbf{G}/X_B)\to\Phi^1(\mathbf{G}/X_A)$ is surjective, which gives the result.
\end{proof}

\begin{corollary}\label{cor:prorep criterion the third}
  Let $X$ be a smooth proper $k$-scheme such that $\H^0(X,T^1_X)=0$ and $\H^{1}(X,\colie^{\vee}_{\mathbf{G}})=0$. For any $\mathbf{G}$-gerbe $\ms X$ over $X$, the functor $\Def_{\ms X/\Lambda}$ is prorepresentable.
\end{corollary}
\begin{proof}
    The assumption that $\H^{1}(X,\colie^{\vee}_{\mathbf{G}})$ implies that $\Phi^1(\mathbf{G}/X_A)=0$ for any deformation $X_A$ of $X$. The result follows from Theorem \ref{thm:prorep criterion the second}.
\end{proof}


\subsection{Comparison with cohomological deformations}

We compare deformations of $\ms X$ with deformations of its cohomology class $\alpha=[\ms X]\in\H^2(X,\mathbf{G})$.

\begin{definition}
    The \textit{cohomological deformation functor} of $\alpha$ over $X_{\Lambda}$ is the (covariant) functor $\oDef_{\alpha/X_{\Lambda}}$ on $\cC_{\Lambda}$ defined by $A\mapsto\oDef(\alpha/X_A)$ (the set of classes $\alpha_A\in\H^i(X_A,\mathbf{G}_{A})$ such that $\alpha_A|_X=\alpha$).
\end{definition}

\begin{remark}\label{rem:RS criterion for cohomological def functor}
    Arguing as in \cite[Lemma 2.9]{MR0457458}, one can show that the cohomological deformation functor $\oDef_{\alpha/X_{\Lambda}}$ is prorepresentable if the functor $A\mapsto \H^1(X_A,\mathbf{G}_A)$ is formally smooth.
\end{remark}

The maps~\eqref{eq:definition of varpi, non functor version} induced by $(\ms X_A,\varphi)\mapsto [\ms X_A]$ give rise to a map of functors
\begin{equation}\label{eq:deformation functor comparison map}
    \Def_{\ms X/X_{\Lambda}}\to\oDef_{\alpha/X_{\Lambda}}.
\end{equation}

\begin{lemma}\label{lem:hull lemma}
    The map~\eqref{eq:deformation functor comparison map} is formally smooth and induces an isomorphism on tangent spaces. It is an isomorphism if and only if for all $A\in\cC_{\Lambda}$ the restriction map $\H^1(X_A,\mathbf{G}_A)\to\H^1(X,\mathbf{G})$ is surjective.
\end{lemma}
\begin{proof}
    Let $B\twoheadrightarrow A$ be a surjection in $\cC_\Lambda$ whose kernel $I$ has square zero. Let $\ms X_A$ be a deformation of $\ms X$ over $X_A$. The obstruction class $o(\ms X_A/X_B)\in\H^3(X,\colie^{\vee}\otimes I)$ vanishes if and only if the cohomology class $\alpha_A$ of $\ms X_A$ lifts to $X_B$. This implies that~\eqref{eq:deformation functor comparison map} is formally smooth. By Proposition \ref{prop:obstructions for gerbe deformation functor}, the tangent space to $\Def_{\ms X/X_{\Lambda}}$ is $\H^2(X,\colie^{\vee}_{\mathbf{G}})$. As the map $k[\varepsilon]\to k$ splits, the sequence
    \[
        0\to\H^2(X,\colie^{\vee}_{\mathbf{G}})\to\H^2(X[\varepsilon],\mathbf{G}_{k[\varepsilon]})\to\H^2(X,\mathbf{G})\to 0
    \]
    is exact, and it follows that~\eqref{eq:deformation functor comparison map} is an isomorphism on tangent spaces. The final claim follows from Lemma \ref{lem:properties of varpi, non functor version}.
\end{proof}

\begin{remark}
   The assumption of Lemma \ref{lem:hull lemma} that the maps $\H^1(X_A,\mathbf{G}_A)\to\H^1(X,\mathbf{G})$ are all surjective holds trivially if $\H^1(X,\mathbf{G})=0$. It also holds if the functor $A\mapsto\H^1(X_A,\mathbf{G}_A)$ on $\cC_{\Lambda}$ is formally smooth.
\end{remark}

\begin{remark}\label{rem:not prorepresentable example}
  If $\Def_{\ms X/X_{\Lambda}}$ is prorepresentable, then Lemma \ref{lem:hull lemma} shows that it is a hull for the cohomological deformation functor $\oDef_{\alpha/X_{\Lambda}}$, in the sense of Schlessinger \cite[Definition 2.7]{MR217093}. An example where $\Def_{\ms X/X_{\Lambda}}$ is prorepresentable but $\oDef_{\alpha/X_{\Lambda}}$ is not is when $\Spf \Lambda$ is the universal deformation space of a K3 surface $X$, $X_{\Lambda}$ is the universal formal family, $\mathbf{G}=\mathbf{G}_m$, and $\ms X$ is any $\mathbf{G}_m$-gerbe on $X$. Indeed, in this case the functor $A\mapsto\H^1(X_A,\mathbf{G}_{m})=\Pic(X_A)$ is not formally smooth. On the other hand, as $\H^1(X,\ms O_X)=0$, the functor $\Phi^1_{\mathbf{G}_m/X_{\Lambda}}$ is trivial, and in particular formally smooth.
\end{remark}


\bibliographystyle{plain}
\bibliography{biblio}

\end{document}